\documentclass[final,3p,times]{elsarticle}
\usepackage{amsmath,amssymb,amsthm,bm}
\usepackage{booktabs,tabularx,array}
\usepackage{graphicx}
\usepackage{microtype}
\usepackage{tikz}
\usetikzlibrary{arrows,positioning,calc}
\usepackage[hidelinks]{hyperref}
\hypersetup{pdfauthor={Jiayuan Peng, Hao Dong, Jiale Linghu}}

\journal{Computer Methods in Applied Mechanics and Engineering}

\newtheorem{lemma}{Lemma}
\newtheorem{proposition}{Proposition}
\newtheorem{theorem}{Theorem}

\makeatletter
\@ifundefined{highlights}{%
}{}
\makeatother

\newcommand{\eps}{\varepsilon}
\newcommand{\dd}{\,\mathrm{d}}
\newcommand{\R}{\mathbb{R}}
\newcommand{\dt}{\Delta t}
\newcommand{\bu}{\mathbf u}
\newcommand{\bv}{\mathbf v}
\newcommand{\bU}{\mathbf U}
\newcommand{\bV}{\mathbb V}
\newcommand{\bL}{\mathbb L}
\newcommand{\norm}[1]{\left\lVert#1\right\rVert}

\begin{document}

\begin{frontmatter}

\title{POD-accelerated high-order multiscale method with corrected-basis reconstruction for transient multi-continuum problems}

\author[addr1]{Jiayuan Peng}
\author[addr1]{Hao Dong\corref{cor1}}
\ead{donghaoxd@xidian.edu.cn}
\author[addr1]{Jiale Linghu}
\cortext[cor1]{Corresponding author.}

\address[addr1]{School of Mathematics and Statistics, Xidian University, Xi'an 710071, Shaanxi, People's Republic of China}

\begin{abstract}
Multi-continuum models provide an effective macroscopic description of transient flow in fractured and highly heterogeneous porous media by introducing several interacting state variables at each macroscale location. Although high-order multiscale (HOMS) methods reduce the cost of resolving microscopic heterogeneities while recovering local oscillatory features, repeated solution of the coupled homogenized system and repeated high-order reconstruction remain expensive in long-time and large-scale three-dimensional simulations. This work develops a proper orthogonal decomposition (POD) accelerated HOMS framework for transient multi-continuum problems. A vector-valued POD space is constructed from snapshots of the coupled homogenized solution, and the resulting block system is projected onto a low-dimensional space. Two recovery strategies are considered. The first reconstructs the reduced homogenized fields and performs full HOMS assembly at every output time. The second applies the spatial part of the multiscale recovery operator to retained POD modes offline and evaluates the online solution by direct linear combinations of corrected bases and time-derivative bases. An error estimate between the fully discrete homogenized solution and its POD-reduced approximation is established for the coupled system. A recovered POD-HOMS estimate is then obtained by combining the HOMS consistency error with POD truncation errors propagated through value, gradient, inter-continuum exchange, second-derivative, and time-derivative recovery terms. Numerical experiments for two- and three-dimensional heterogeneous configurations show that the reduced solutions retain the accuracy and local oscillatory behavior of the corresponding HOMS solution while substantially decreasing computational cost. In a large three-dimensional test, corrected-basis reconstruction remains feasible when direct fine-scale FEM-based approaches exceed the available memory.
\end{abstract}


\end{frontmatter}

\section{Introduction}
\label{sec:introduction}

Flow and transport in fractured and highly heterogeneous porous media arise in a wide range of engineering and geoscientific applications, including groundwater flow, hydrocarbon recovery, geothermal exploitation, contaminant transport, and geological carbon storage. Such media commonly contain low-permeability matrix regions together with highly conductive fractures or channels, resulting in strong heterogeneity, high coefficient contrast, and pronounced multiscale characteristics. Moreover, the pressure states associated with different pore systems may not reach local equilibrium, so a single macroscopic variable is often insufficient to characterize the coupled transient behavior. Accurate and efficient simulation of these problems remains challenging, especially when microscopic heterogeneities, inter-continuum exchange, and long-time evolution coexist~\cite{berre2019}.

To describe nonequilibrium interactions among different pore systems, a variety of multi-continuum models have been developed. The classical dual-porosity concepts of Barenblatt et al.~\cite{barenblatt1960} and Warren and Root~\cite{warren1963} introduce distinct averaged pressures for interacting matrix and fracture systems. Kazemi et al.~\cite{kazemi1976} subsequently extended the dual-porosity framework to numerical simulation of two-phase flow in naturally fractured reservoirs. To represent transient exchange inside matrix blocks more accurately, Pruess and Narasimhan~\cite{pruess1985} proposed the multiple interacting continua (MINC) concept, in which the matrix is subdivided into a sequence of interacting continua. Wu and Pruess~\cite{wu1988} further developed a multiple-porosity formulation for naturally fractured reservoirs. These continuum-based models remain important tools for describing fractured-media flow, as summarized in the review by Berre et al.~\cite{berre2019}. However, when fine-scale inclusions, channels, or fracture structures are explicitly contained in each continuum, direct numerical simulation still requires very fine spatial discretizations and may lead to prohibitive computational cost.

To reduce the computational burden caused by fine-scale heterogeneities, various multiscale and upscaling techniques have been developed. The multiscale finite element method (MsFEM) constructs local basis functions adapted to heterogeneous coefficients and provides an efficient coarse-scale approximation without globally resolving all microscopic features~\cite{hou1997}. The generalized multiscale finite element method (GMsFEM) further introduces systematically enriched local approximation spaces for problems with complex and high-contrast coefficients~\cite{efendiev2013}, while the constraint energy minimizing GMsFEM (CEM-GMsFEM) provides localized basis functions with contrast-robust convergence properties~\cite{chung2018}. For multi-continuum flow, nonlocal multicontinuum (NLMC) upscaling constructs physically meaningful continuum variables through constrained local basis problems~\cite{vasilyeva2019}, and related generalized multiscale formulations have been developed for coupled unsaturated-flow systems in fractured porous media~\cite{spiridonov2020}. More recently, multicontinuum equations have been derived through homogenization-type expansions and constrained cell problems~\cite{efendiev2023,chung2024}, and the framework has been extended to coupled poroelasticity models~\cite{ammosov2026}. Efficient implicit--explicit strategies have also been investigated for decoupling transient multicontinuum systems~\cite{vasilyeva2024}. These approaches substantially reduce the cost associated with spatial heterogeneity. Nevertheless, spatial upscaling alone does not eliminate the computational expense of repeatedly advancing the resulting coupled transient system over many time levels or in repeated simulations.

Reduced-order modeling provides a complementary way to further decrease the algebraic and temporal cost of large-scale transient simulations. Proper orthogonal decomposition (POD), closely related to the method of snapshots introduced by Sirovich~\cite{sirovich1987}, constructs a low-dimensional approximation space from representative solution data. Rigorous POD-Galerkin error analyses for parabolic and evolutionary partial differential equations were developed by Kunisch and Volkwein~\cite{kunisch2001,kunisch2002}, and subsequent work related the reduced-order approximation error more explicitly to discarded POD eigenvalues and modes~\cite{singler2014}. The quality of the resulting reduced space is also influenced by the snapshot ensemble, and optimal snapshot-location strategies for time-dependent problems have therefore been investigated~\cite{kunisch2010snapshot}. More generally, POD-based numerical methods for partial differential equations have been systematically developed within finite-element and related discretization frameworks, together with corresponding stability and error analyses~\cite{luo2018pod}. More recent studies have also combined POD with Krylov-type enrichment to improve reduced approximations for partial differential equations with variable coefficients~\cite{wang2024krylov}. Projection-based model reduction has since become a broad framework for accelerating large-scale dynamical systems~\cite{benner2015}. For heterogeneous porous-media flow, POD and related mode-decomposition techniques have been combined with multiscale approximations to construct reduced models for high-contrast permeability fields~\cite{ghommem2013,ghommem2014}, and POD-based reduced-order modeling has also been applied directly to groundwater flow in heterogeneous media~\cite{dey2020}. These studies demonstrate the effectiveness of POD in reducing the dimension of transient flow systems. However, for a high-order multiscale method, reducing only the homogenized system is not sufficient, because the final quantity of interest is generally the reconstructed multiscale field rather than the smooth homogenized solution itself.

In a high-order multiscale (HOMS) method, the homogenized solution must be further reconstructed to recover microscopic oscillations. For transient multi-continuum problems, the recovery procedure involves not only the values of the homogenized pressures, but also their spatial gradients, higher-order spatial derivatives, inter-continuum pressure differences, cross-continuum gradients, and time derivatives. Consequently, two computational and analytical difficulties remain. First, repeated high-order recovery on a fine output mesh may become a dominant computational cost, particularly for long-time and large-scale three-dimensional simulations. Second, a POD error estimate established only for the homogenized system does not directly characterize how the discarded modes affect the final reconstructed multiscale fields. Therefore, an efficient reduced-order HOMS method should reduce both the cost of the coupled transient evolution and the cost of repeated high-order reconstruction, while also providing an error estimate for the recovered multiscale solution.
To overcome these difficulties, this paper develops a POD-accelerated HOMS computational method for transient multi-continuum problems in highly heterogeneous media. The HOMS framework established in our previous work~\cite{dong2026homs} is used as the spatial multiscale foundation. The interacting homogenized pressures are treated as a unified block state, and a joint mass-weighted POD space is constructed from coupled snapshots so that the correlations induced by diffusion, drift, and inter-continuum exchange are retained in the reduced system. Projection of the coupled homogenized equations onto this low-dimensional space substantially reduces the cost of the transient macroscopic evolution.


The remainder of this paper is organized as follows. Section~\ref{sec:model} introduces the transient multi-continuum model and presents the HOMS representation, cell problems, homogenized block system, fully discrete formulation, and vector-valued POD reduction. Section~\ref{sec:analysis} gives the convergence analysis of the proposed POD-HOMS method, including an error estimate between the discrete homogenized solution and its POD-reduced approximation and the recovered multiscale error estimate. Section~\ref{sec:algorithms} develops the offline--online computational framework and presents the two HOMS recovery strategies. Section~\ref{sec:numerical} reports the numerical experiments for accuracy, computational efficiency, and large-scale applicability. Finally, conclusions are given in Section~\ref{sec:conclusions}.

\section{Multi-continuum model and POD-reduced formulation}
\label{sec:model}

Let $\Omega\subset\R^d$, $d=2,3$, be a bounded Lipschitz domain, $T>0$, and $Y=(0,1)^d$. The microscopic coordinate is $\bm y=\bm x/\eps$, and repeated spatial indices are summed from $1$ to $d$. 

\subsection{Fine-scale model and HOMS representation}
\label{subsec:fine-homs}

For $l=1,2$ and $\bar l=3-l$, let $u_l^\eps$ denote the pressure in continuum $l$. The representative fine-scale system is
\begin{equation}
\left\{
\begin{aligned}
c_1(\bm y)\frac{\partial u_1^\eps}{\partial t}
&=\frac{\partial}{\partial x_i}\!\left(
\kappa_{1,ij}(\bm y)\frac{\partial u_1^\eps}{\partial x_j}\right)
+q_1+\eps^{-1}Q_1(\bm y)(u_2^\eps-u_1^\eps),
&& (\bm x,t)\in\Omega\times(0,T),\\
c_2(\bm y)\frac{\partial u_2^\eps}{\partial t}
&=\frac{\partial}{\partial x_i}\!\left(
\kappa_{2,ij}(\bm y)\frac{\partial u_2^\eps}{\partial x_j}\right)
+q_2+\eps^{-1}Q_2(\bm y)(u_1^\eps-u_2^\eps),
&& (\bm x,t)\in\Omega\times(0,T),\\
u_1^\eps(\bm x,t)&=u_2^\eps(\bm x,t)=0,
&& (\bm x,t)\in\partial\Omega\times(0,T),\\
u_1^\eps(\bm x,0)&=u_{1,0}(\bm x),\quad u_2^\eps(\bm x,0)=u_{2,0}(\bm x),
&& \bm x\in\Omega.
\end{aligned}
\right.
\label{eq:fine-system}
\end{equation}

Here $u_{l,0}$ is the prescribed initial pressure in continuum $l$. The storage coefficients $c_l$ are positive and bounded, the permeability tensors $\kappa_l$ are uniformly elliptic and $Y$-periodic, and $Q_l\ge0$ are exchange coefficients. In the conservative symmetric case $Q_1=Q_2$.

The second-order HOMS representation established in~\cite{dong2026homs} is summarized by
\begin{align}
u_l^{(2\eps)}={}&u_l^{(0)}
+\eps N_l^\alpha\frac{\partial u_l^{(0)}}{\partial x_\alpha}
+\eps M_l(u_{\bar l}^{(0)}-u_l^{(0)})
\notag\\
&+\eps^2\left[
G_l\frac{\partial u_l^{(0)}}{\partial t}
+N_l^{\alpha\beta}\frac{\partial^2 u_l^{(0)}}{\partial x_\alpha\partial x_\beta}
+C_l^\alpha\frac{\partial u_l^{(0)}}{\partial x_\alpha}
+F_l^\alpha\frac{\partial u_{\bar l}^{(0)}}{\partial x_\alpha}
+K_l(u_{\bar l}^{(0)}-u_l^{(0)})
\right],
\label{eq:homs-representation}
\end{align}
where every cell function is evaluated at $\bm y=\bm x/\eps$. For $l=1,2$ and $\alpha=1,\ldots,d$, the first-order conductivity corrector $N_l^\alpha$ and exchange corrector $M_l$ are defined by two separate cell problems:
\begin{equation}
\left\{
\begin{aligned}
\frac{\partial}{\partial y_i}\!\left(\kappa_{l,ij}(\bm y)
\frac{\partial N_l^\alpha(\bm y)}{\partial y_j}\right)
&=-\frac{\partial\kappa_{l,i\alpha}(\bm y)}{\partial y_i},
&&\bm y\in Y,\\
N_l^\alpha(\bm y)&=0,
&&\bm y\in\partial Y.
\end{aligned}
\right.
\label{eq:first-conductivity-cell}
\end{equation}
\begin{equation}
\left\{
\begin{aligned}
\frac{\partial}{\partial y_i}\!\left(\kappa_{l,ij}(\bm y)
\frac{\partial M_l(\bm y)}{\partial y_j}\right)
&=Q_l(\bm y),
&&\bm y\in Y,\\
M_l(\bm y)&=0,
&&\bm y\in\partial Y.
\end{aligned}
\right.
\label{eq:first-exchange-cell}
\end{equation}
The homogeneous Dirichlet cell closure in \eqref{eq:first-conductivity-cell} and \eqref{eq:first-exchange-cell} is used consistently in the effective coefficients and all recovery terms.

For $l,m=1,2$ and $i,j=1,\ldots,d$, the effective coefficients entering the homogenized equations are determined by the first-order cell functions $N_l^j$ and $M_l$:
\begin{align}
c_l^*&=\int_Y c_l(\bm y)\dd\bm y,
\notag\\
\kappa_{l,ij}^*&=\int_Y\left[\kappa_{l,ij}(\bm y)
+\kappa_{l,ik}(\bm y)\frac{\partial N_l^j(\bm y)}{\partial y_k}\right]\dd\bm y,
\notag\\
K_{lm,i}^*&=\int_Y\left[\kappa_{l,ij}(\bm y)
\frac{\partial M_l(\bm y)}{\partial y_j}
+Q_l(\bm y)N_m^i(\bm y)\right]\dd\bm y,
\notag\\
Q_l^*&=-\int_Y Q_l(\bm y)\bigl(M_1(\bm y)+M_2(\bm y)\bigr)\dd\bm y.
\label{eq:effective-coefficients}
\end{align}
Here $\bm K_{lm}^*=(K_{lm,1}^*,\ldots,K_{lm,d}^*)^T$; in particular, $\bm K_{ll}^*$ multiplies the gradient of the pressure in continuum $l$, whereas $\bm K_{l\bar l}^*$ multiplies that in the other continuum. Since $Y=(0,1)^d$ has unit volume, the integrals are also cell averages.

The five second-order correctors in \eqref{eq:homs-representation} complete the recovery library. For $l=1,2$ and $\alpha,\beta=1,\ldots,d$, each corrector is specified together with its boundary condition as follows. The time-derivative corrector $G_l$ satisfies
\begin{equation*}
\left\{
\begin{aligned}
\frac{\partial}{\partial y_i}\!\left(\kappa_{l,ij}(\bm y)
\frac{\partial G_l(\bm y)}{\partial y_j}\right)
&=c_l(\bm y)-c_l^*,
&&\bm y\in Y,\\
G_l(\bm y)&=0,
&&\bm y\in\partial Y.
\end{aligned}
\right.
\end{equation*}
The second-gradient corrector $N_l^{\alpha\beta}$ satisfies
\begin{equation*}
\left\{
\begin{aligned}
\frac{\partial}{\partial y_i}\!\left(\kappa_{l,ij}(\bm y)
\frac{\partial N_l^{\alpha\beta}(\bm y)}{\partial y_j}\right)
={}&\kappa_{l,\alpha\beta}^*-\kappa_{l,\alpha\beta}(\bm y)
-\frac{\partial}{\partial y_i}\!\left(\kappa_{l,i\alpha}(\bm y)N_l^\beta(\bm y)\right)
-\kappa_{l,\alpha j}(\bm y)\frac{\partial N_l^\beta(\bm y)}{\partial y_j},
&&\bm y\in Y,\\
N_l^{\alpha\beta}(\bm y)={}&0,
&&\bm y\in\partial Y.
\end{aligned}
\right.
\end{equation*}
The same-continuum first-gradient corrector $C_l^\alpha$ is defined by
\begin{equation*}
\left\{
\begin{aligned}
\frac{\partial}{\partial y_i}\!\left(\kappa_{l,ij}(\bm y)
\frac{\partial C_l^\alpha(\bm y)}{\partial y_j}\right)
={}&\frac{\partial}{\partial y_i}\!\left(\kappa_{l,i\alpha}(\bm y)M_l(\bm y)\right)
+Q_l(\bm y)N_l^\alpha(\bm y)
-K_{ll,\alpha}^*
+\kappa_{l,\alpha j}(\bm y)\frac{\partial M_l(\bm y)}{\partial y_j},
&&\bm y\in Y,\\
C_l^\alpha(\bm y)={}&0,
&&\bm y\in\partial Y.
\end{aligned}
\right.
\end{equation*}
The cross-continuum first-gradient corrector $F_l^\alpha$ satisfies
\begin{equation*}
\left\{
\begin{aligned}
\frac{\partial}{\partial y_i}\!\left(\kappa_{l,ij}(\bm y)
\frac{\partial F_l^\alpha(\bm y)}{\partial y_j}\right)
={}&-\frac{\partial}{\partial y_i}\!\left(\kappa_{l,i\alpha}(\bm y)M_l(\bm y)\right)
+K_{l\bar l,\alpha}^*
-\kappa_{l,\alpha j}(\bm y)\frac{\partial M_l(\bm y)}{\partial y_j}
-Q_l(\bm y)N_{\bar l}^\alpha(\bm y),
&&\bm y\in Y,\\
F_l^\alpha(\bm y)={}&0,
&&\bm y\in\partial Y.
\end{aligned}
\right.
\end{equation*}
Finally, the second-order exchange corrector $K_l$ is given by
\begin{equation*}
\left\{
\begin{aligned}
\frac{\partial}{\partial y_i}\!\left(\kappa_{l,ij}(\bm y)
\frac{\partial K_l(\bm y)}{\partial y_j}\right)
&=Q_l^*+Q_l(\bm y)\bigl(M_1(\bm y)+M_2(\bm y)\bigr),
&&\bm y\in Y,\\
K_l(\bm y)&=0,
&&\bm y\in\partial Y.
\end{aligned}
\right.
\end{equation*}
The first- and second-order cell problems above specify every cell function used by the second-order recovery. Their asymptotic derivation and consistency analysis follow the preceding HOMS work~\cite{dong2026homs}; they are stated here so that the model and the recovery operator are self-contained before the POD reduction is introduced.

\subsection{Homogenized coupled system}
\label{subsec:hom-system}

The macroscopic pressures satisfy
\begin{equation}
\left\{
\begin{aligned}
c_1^*\frac{\partial u_1^{(0)}}{\partial t}
&=\frac{\partial}{\partial x_i}\!\left(
\kappa_{1,ij}^*\frac{\partial u_1^{(0)}}{\partial x_j}\right)
+K_{12,j}^*\frac{\partial u_2^{(0)}}{\partial x_j}
-K_{11,j}^*\frac{\partial u_1^{(0)}}{\partial x_j}
+Q_1^*(u_2^{(0)}-u_1^{(0)})+q_1,
&& (\bm x,t)\in\Omega\times(0,T),\\
c_2^*\frac{\partial u_2^{(0)}}{\partial t}
&=\frac{\partial}{\partial x_i}\!\left(
\kappa_{2,ij}^*\frac{\partial u_2^{(0)}}{\partial x_j}\right)
+K_{21,j}^*\frac{\partial u_1^{(0)}}{\partial x_j}
-K_{22,j}^*\frac{\partial u_2^{(0)}}{\partial x_j}
+Q_2^*(u_1^{(0)}-u_2^{(0)})+q_2,
&& (\bm x,t)\in\Omega\times(0,T),\\
u_1^{(0)}(\bm x,t)&=u_2^{(0)}(\bm x,t)=0,
&& (\bm x,t)\in\partial\Omega\times(0,T),\\
u_1^{(0)}(\bm x,0)&=u_{1,0}(\bm x),\quad
u_2^{(0)}(\bm x,0)=u_{2,0}(\bm x),
&& \bm x\in\Omega.
\end{aligned}
\right.
\label{eq:hom-system}
\end{equation}
The effective coefficients in \eqref{eq:hom-system} are given explicitly by \eqref{eq:effective-coefficients}.

Set $\bV=H_0^1(\Omega)^2$ and $\bL=L^2(\Omega)^2$. For $\bu=(u_1,u_2)^T$ and $\bv=(v_1,v_2)^T$, define
\begin{equation*}
M_0(\bu,\bv)=\sum_{l=1}^2\int_\Omega c_l^*u_lv_l\dd\bm x
\end{equation*}
and $B_0=B_d+B_K+B_Q$, where
\begin{align*}
B_d(\bu,\bv)
&=\sum_{l=1}^2\int_\Omega\kappa_{l,ij}^*
\frac{\partial u_l}{\partial x_j}\frac{\partial v_l}{\partial x_i}\dd\bm x,
\\
B_K(\bu,\bv)
&=-\int_\Omega K_{12,j}^*\frac{\partial u_2}{\partial x_j}v_1\dd\bm x
+\int_\Omega K_{11,j}^*\frac{\partial u_1}{\partial x_j}v_1\dd\bm x
\\
&\quad-\int_\Omega K_{21,j}^*\frac{\partial u_1}{\partial x_j}v_2\dd\bm x
+\int_\Omega K_{22,j}^*\frac{\partial u_2}{\partial x_j}v_2\dd\bm x,
\\
B_Q(\bu,\bv)
&=\int_\Omega Q_1^*(u_1-u_2)v_1\dd\bm x
+\int_\Omega Q_2^*(u_2-u_1)v_2\dd\bm x.
\end{align*}
The weak problem is
\begin{equation}
M_0\!\left(\frac{\partial\bu^{(0)}}{\partial t},\bv\right)
+B_0(\bu^{(0)},\bv)=F(\bv;t),
\qquad \forall\bv\in\bV,
\label{eq:hom-weak}
\end{equation}
where $F(\bv;t)=\sum_{l=1}^2\int_\Omega q_lv_l\dd\bm x$. The exchange and cross-gradient contributions need not be symmetric, which motivates the G\aa rding framework in Section~\ref{sec:analysis}.

The variational problem \eqref{eq:hom-weak} gives the continuous-in-time formulation of the homogenized multi-continuum system. Since the POD basis used in this work is constructed from numerically generated snapshots, we next discretize \eqref{eq:hom-weak} in both space and time. The resulting fully discrete homogenized solution provides the snapshot data for the POD construction and serves as the reference solution for analyzing the error introduced by POD reduction.

\subsection{Fully discrete approximation of the homogenized system}
\label{subsec:block-system}

We now discretize the continuous homogenized problem \eqref{eq:hom-weak}. Let $V_h\subset H_0^1(\Omega)$ be a conforming finite element space and set $\bV_h=V_h^2$. For the temporal grid $t_n=n\dt$, $n=0,\ldots,N_t$, with $T=N_t\dt$, the finite-element and backward Euler approximation of \eqref{eq:hom-weak} reads: find $\bU_h^n\in\bV_h$ such that
\begin{equation}
M_0\!\left(\frac{\bU_h^n-\bU_h^{n-1}}{\dt},\bv_h\right)
+B_0(\bU_h^n,\bv_h)=F(\bv_h;t_n),
\qquad\forall\bv_h\in\bV_h.
\label{eq:fom-weak}
\end{equation}
Here $\bU_h^n$ denotes the fully discrete finite-element approximation of the continuous homogenized solution $\bu^{(0)}(t_n)$.

With scalar finite element basis $\{\varphi_i\}_{i=1}^{N_h}$, define
\begin{align*}
(\mathbf M_l)_{ij}&=\int_\Omega c_l^*\varphi_j\varphi_i\dd\bm x,&
(\mathbf D_l)_{ij}&=\int_\Omega\kappa_{l,pq}^*
\frac{\partial\varphi_j}{\partial x_q}\frac{\partial\varphi_i}{\partial x_p}\dd\bm x,
\\
(\mathbf K_{lm})_{ij}&=\int_\Omega K_{lm,p}^*
\frac{\partial\varphi_j}{\partial x_p}\varphi_i\dd\bm x,&
(\mathbf Q_l)_{ij}&=\int_\Omega Q_l^*\varphi_j\varphi_i\dd\bm x.
\end{align*}
The global matrices are
\begin{equation}
\mathbf M=
\begin{pmatrix}\mathbf M_1&\mathbf0\\\mathbf0&\mathbf M_2\end{pmatrix},
\qquad
\mathbf A=
\begin{pmatrix}
\mathbf D_1+\mathbf K_{11}+\mathbf Q_1&-\mathbf K_{12}-\mathbf Q_1\\
-\mathbf K_{21}-\mathbf Q_2&\mathbf D_2+\mathbf K_{22}+\mathbf Q_2
\end{pmatrix}.
\label{eq:block-matrices}
\end{equation}
Let $\mathbf U^n=((\mathbf U_1^n)^T,(\mathbf U_2^n)^T)^T\in\R^{2N_h}$ be the algebraic coefficient vector of the finite-element function $\bU_h^n\in\bV_h$. More precisely, $U_{l,h}^n(\bm x)=\sum_{i=1}^{N_h}(\mathbf U_l^n)_i\varphi_i(\bm x)$ for $l=1,2$. Thus $\mathbf U^n$ and $\bU_h^n$ represent the same fully discrete approximation in coefficient-vector and finite-element-function form, respectively. The fully discrete evolution then has the matrix form
\begin{equation}
\left(\frac{\mathbf M}{\dt}+\mathbf A\right)\mathbf U^n
=\mathbf F^n+\frac{\mathbf M}{\dt}\mathbf U^{n-1}.
\label{eq:fom-matrix}
\end{equation}
Solving \eqref{eq:fom-weak}, or equivalently the algebraic system \eqref{eq:fom-matrix}, over the training interval produces the discrete homogenized trajectory used for POD snapshot generation. The POD reduction introduced below is therefore performed on the fully discrete homogenized system rather than directly on the continuous problem \eqref{eq:hom-weak}.

\subsection{Vector-valued POD-Galerkin reduction}
\label{subsec:vector-pod}

The POD reduction is performed at the fully discrete level. Specifically, snapshots are collected from the discrete homogenized trajectory $\{\bU_h^n\}$ generated by \eqref{eq:fom-weak}. Consequently, the POD-reduced solution $\bU_r^n$ introduced below approximates the discrete homogenized solution $\bU_h^n$, rather than directly approximating the continuous solution $\bu^{(0)}(t)$. The joint state is $\bU_h^n=(U_{1,h}^n,U_{2,h}^n)^T$. From $m$ coupled snapshots, define
\begin{equation}
C_{ij}=\frac1m M_0(\bU_h^i,\bU_h^j),
\qquad i,j=1,\ldots,m.
\label{eq:pod-correlation}
\end{equation}
If $C\bm\xi^k=\lambda_k\bm\xi^k$ with $\lambda_1\ge\cdots\ge\lambda_s>0$, the mass-orthonormal vector modes are
\begin{equation}
\bm\phi_k=\frac{1}{\sqrt{m\lambda_k}}\sum_{j=1}^m\xi_j^k\bU_h^j
=\begin{pmatrix}\phi_{1,k}\\\phi_{2,k}\end{pmatrix},
\qquad k=1,\ldots,s.
\label{eq:vector-pod-modes}
\end{equation}
The rank $r$ is selected using
\begin{equation}
\frac{\sum_{k=1}^r\lambda_k}{\sum_{k=1}^s\lambda_k}\ge1-\delta_{\rm pod}.
\label{eq:pod-energy}
\end{equation}
Unlike separate componentwise reductions, the joint space $\bV_r=\operatorname{span}\{\bm\phi_1,\ldots,\bm\phi_r\}$ preserves snapshot-level correlations between matrix and fracture pressures. Writing
\begin{equation*}
\bU_r^n=\sum_{k=1}^rz_k^n\bm\phi_k
\end{equation*}
and testing by $\bV_r$ gives
\begin{equation}
\left(\frac{\mathbf M_r}{\dt}+\mathbf A_r\right)\mathbf z^n
=\mathbf f_r^n+\frac{\mathbf M_r}{\dt}\mathbf z^{n-1},
\label{eq:rom-system}
\end{equation}
where
\begin{equation*}
(\mathbf M_r)_{ij}=M_0(\bm\phi_j,\bm\phi_i),\quad
(\mathbf A_r)_{ij}=B_0(\bm\phi_j,\bm\phi_i),\quad
(\mathbf f_r^n)_i=F(\bm\phi_i;t_n).
\end{equation*}
\section{Convergence analysis of the multi-continuum POD-HOMS method}
\label{sec:analysis}

The POD basis and the reduced system in Section~\ref{subsec:vector-pod} are constructed from the fully discrete homogenized solutions introduced in Section~\ref{subsec:block-system}. Accordingly, the reduction error is analyzed at the discrete level by comparing $\bU_h^n$ and $\bU_r^n$. Since both solutions employ the same spatial finite-element discretization and backward Euler time scheme, this comparison isolates the additional error introduced by the POD reduction. The continuous variational formulation \eqref{eq:hom-weak} is used to establish the continuity and G\aa rding stability properties of the coupled bilinear form. The resulting reduction error is then propagated through the discrete multi-continuum HOMS recovery operator and combined with the inherited HOMS consistency, spatial-discretization, and time-discretization errors. The recovery analysis also involves continuum differences, cross-continuum gradients, second derivatives, and discrete time derivatives.

\subsection{Stability of the coupled bilinear form}
\label{subsec:coupled-stability}

Equip $\bV=H_0^1(\Omega)^2$ with
\begin{equation}
\norm{\bv}_{\bV}^2
=\sum_{l=1}^2\sum_{j=1}^d
\left\|\frac{\partial v_l}{\partial x_j}\right\|_{L^2(\Omega)}^2,
\qquad
\norm{\bv}_M^2=M_0(\bv,\bv).
\label{eq:vector-norms}
\end{equation}
Assume that $0<c_{\min}\le c_l^*\le c_{\max}$, that each $\kappa_l^*$ is uniformly elliptic with lower bound $\kappa_0>0$, and that the drift and exchange coefficients are essentially bounded. The mass norm therefore satisfies
\begin{equation}
c_M\norm{\bv}_{L^2(\Omega)^2}^2
\le\norm{\bv}_M^2
\le C_M\norm{\bv}_{L^2(\Omega)^2}^2
\label{eq:mass-equivalence}
\end{equation}
for constants $c_M,C_M>0$.

\begin{lemma}[Continuity and G\aa rding inequality]
\label{lem:garding}
Under the preceding coefficient assumptions, there are constants $C_B,\alpha>0$ and $\beta\ge0$ such that
\begin{align}
\left|B_0(\bu,\bv)\right|
&\le C_B\norm{\bu}_{\bV}\norm{\bv}_{\bV},
\label{eq:bilinear-continuity}\\
B_0(\bv,\bv)
&\ge\alpha\norm{\bv}_{\bV}^2-\beta\norm{\bv}_M^2
\label{eq:garding}
\end{align}
for all $\bu,\bv\in\bV$.
\end{lemma}

\begin{proof}
The uniformly elliptic diffusion part satisfies
\begin{equation}
B_d(\bv,\bv)
=\sum_{l=1}^2\int_\Omega
\kappa_{l,ij}^*\frac{\partial v_l}{\partial x_j}
\frac{\partial v_l}{\partial x_i}\dd\bm x
\ge\kappa_0\norm{\bv}_{\bV}^2.
\label{eq:diffusion-coercivity}
\end{equation}
For a representative cross-gradient term, H\"older's and Young's inequalities give, for any $\delta>0$,
\begin{align}
\left|\int_\Omega K_{12,j}^*\frac{\partial v_2}{\partial x_j}
v_1\dd\bm x\right|
&\le\norm{\bm K_{12}^*}_{L^\infty}
\left(\sum_{j=1}^d
\left\|\frac{\partial v_2}{\partial x_j}\right\|_{L^2}^2\right)^{1/2}
\norm{v_1}_{L^2}
\notag\\
&\le\delta\sum_{j=1}^d
\left\|\frac{\partial v_2}{\partial x_j}\right\|_{L^2}^2
+C_\delta\norm{v_1}_{L^2}^2.
\label{eq:representative-drift-bound}
\end{align}
Applying the same estimate to all four drift terms yields
\begin{equation}
B_K(\bv,\bv)
\ge-\delta\norm{\bv}_{\bV}^2
-C_\delta\norm{\bv}_{L^2(\Omega)^2}^2,
\label{eq:drift-garding-bound}
\end{equation}
after changing $\delta$ by a fixed multiplicative constant.

For the exchange contribution, the pointwise identity
\begin{align}
&Q_1^*(v_1-v_2)v_1+Q_2^*(v_2-v_1)v_2
\notag\\
&\quad=
\frac{Q_1^*+Q_2^*}{2}(v_1-v_2)^2
+\frac{Q_1^*-Q_2^*}{2}(v_1^2-v_2^2)
\label{eq:exchange-decomposition}
\end{align}
shows that the first term is nonnegative whenever $Q_1^*+Q_2^*\ge0$, whereas boundedness of $Q_1^*-Q_2^*$ implies
\begin{equation}
B_Q(\bv,\bv)\ge-C_Q\norm{\bv}_{L^2(\Omega)^2}^2.
\label{eq:exchange-lower-bound}
\end{equation}
If no sign condition is imposed on $Q_1^*+Q_2^*$, its bounded negative part is estimated in the same way and merely enlarges $C_Q$.
Combining \eqref{eq:diffusion-coercivity}, \eqref{eq:drift-garding-bound}, and \eqref{eq:exchange-lower-bound}, choosing $\delta<\kappa_0$, and using \eqref{eq:mass-equivalence} proves \eqref{eq:garding}. Continuity follows from the same coefficient bounds, H\"older's inequality, and the Poincar\'e inequality on $H_0^1(\Omega)$.
\end{proof}

\subsection{Vector POD projection identities}
\label{subsec:pod-identities}

Let $P_r:\bV_h\to\bV_r$ be the $M_0$-orthogonal projection. Define
\begin{equation}
\bm\eta^n=\bU_h^n-P_r\bU_h^n,
\qquad
\bm\theta^n=P_r\bU_h^n-\bU_r^n,
\qquad
\mathbf e^n=\bm\eta^n+\bm\theta^n.
\label{eq:error-split}
\end{equation}
Here $\bm\eta^n$ is the POD projection error and $\bm\theta^n$ is the reduced evolution error.

For the complete snapshot space, write
\begin{equation}
\bU_h^n=\sum_{k=1}^s c_k^n\bm\phi_k,
\qquad
c_k^n=M_0(\bU_h^n,\bm\phi_k).
\label{eq:snapshot-mode-expansion}
\end{equation}
The correlation eigenproblem in \eqref{eq:pod-correlation} and the normalization in \eqref{eq:vector-pod-modes} imply
\begin{equation}
\frac1m\sum_{n=1}^m c_i^nc_j^n
=\lambda_i\delta_{ij}.
\label{eq:pod-coefficient-covariance}
\end{equation}
Indeed, if the Euclidean-normalized correlation eigenvector is $\bm\xi^k$, then $c_k^n=\sqrt{m\lambda_k}\,\xi_n^k$, and \eqref{eq:pod-coefficient-covariance} follows from $(\bm\xi^i)^T\bm\xi^j=\delta_{ij}$.

\begin{lemma}[Vector POD projection identities]
\label{lem:pod-identities}
The joint POD projection error satisfies
\begin{align}
\frac1m\sum_{n=1}^m\norm{\bm\eta^n}_M^2
&=\sum_{k=r+1}^s\lambda_k,
\label{eq:pod-tail-mass}\\
\frac1m\sum_{n=1}^m\norm{\bm\eta^n}_{\bV}^2
&=\sum_{k=r+1}^s\lambda_k\norm{\bm\phi_k}_{\bV}^2.
\label{eq:pod-tail-energy}
\end{align}
\end{lemma}

\begin{proof}
Because
\begin{equation}
P_r\bU_h^n=\sum_{k=1}^r c_k^n\bm\phi_k,
\qquad
\bm\eta^n=\sum_{k=r+1}^s c_k^n\bm\phi_k,
\label{eq:projection-error-expansion}
\end{equation}
the $M_0$-orthonormality of the modes gives
\begin{equation}
\norm{\bm\eta^n}_M^2
=\sum_{k=r+1}^s(c_k^n)^2.
\label{eq:projection-mass-single-time}
\end{equation}
Averaging \eqref{eq:projection-mass-single-time} over the snapshots and using \eqref{eq:pod-coefficient-covariance} proves \eqref{eq:pod-tail-mass}.

The POD modes need not be orthogonal in the $\bV$ inner product. At a single time,
\begin{equation}
\norm{\bm\eta^n}_{\bV}^2
=\sum_{k,j=r+1}^s c_k^nc_j^n
(\bm\phi_k,\bm\phi_j)_{\bV}.
\label{eq:projection-energy-single-time}
\end{equation}
After averaging, \eqref{eq:pod-coefficient-covariance} eliminates every cross term:
\begin{align}
\frac1m\sum_{n=1}^m\norm{\bm\eta^n}_{\bV}^2
&=\sum_{k,j=r+1}^s
\left(\frac1m\sum_{n=1}^m c_k^nc_j^n\right)
(\bm\phi_k,\bm\phi_j)_{\bV}
\notag\\
&=\sum_{k=r+1}^s\lambda_k\norm{\bm\phi_k}_{\bV}^2.
\label{eq:projection-energy-average}
\end{align}
This proves \eqref{eq:pod-tail-energy}. The disappearance of the cross terms is caused by diagonalization of the snapshot coefficient covariance, not by $\bV$-orthogonality of the modes.
\end{proof}

Define
\begin{equation}
\mathcal P_r
=\max_{1\le n\le m}\norm{\bm\eta^n}_M^2
+\dt\sum_{n=1}^m\norm{\bm\eta^n}_{\bV}^2.
\label{eq:Pr-definition}
\end{equation}
If the theoretical snapshot ensemble contains one snapshot at every analyzed time level and $m\dt=T$, then Lemma~\ref{lem:pod-identities} gives
\begin{equation}
\mathcal P_r
\le m\sum_{k=r+1}^s\lambda_k
+T\sum_{k=r+1}^s\lambda_k\norm{\bm\phi_k}_{\bV}^2.
\label{eq:Pr-tail-bound}
\end{equation}
The maximum is bounded by the sum of the nonnegative mass errors, and the second term follows directly from \eqref{eq:pod-tail-energy}. If snapshots are temporally subsampled, the definition \eqref{eq:Pr-definition} remains valid, while \eqref{eq:Pr-tail-bound} must be replaced by the corresponding weighted-snapshot or temporal-interpolation estimate.

\subsection{Error of the POD-reduced homogenized system}
\label{subsec:fom-rom-error}

\begin{theorem}[Error estimate for the POD-reduced homogenized system]
\label{thm:fom-rom}
Let $\bU_h^n$ solve the fully discrete homogenized problem \eqref{eq:fom-weak}, let $\bU_r^n$ solve the POD-reduced problem \eqref{eq:rom-system}, and choose $\bU_r^0=P_r\bU_h^0$. If $\beta>0$, assume that $2\beta\dt\le\rho<1$ for a fixed $\rho$. Then
\begin{equation}
\max_{1\le n\le m}\norm{\bU_h^n-\bU_r^n}_M^2
+\dt\sum_{n=1}^m\norm{\bU_h^n-\bU_r^n}_{\bV}^2
\le C_T\mathcal P_r,
\label{eq:fom-rom-bound}
\end{equation}
where $C_T$ is independent of $h$, $r$, and the discarded eigenvalues, and is uniform in $\dt$ under the stated step-size margin.
\end{theorem}

\begin{proof}
Because $\bV_r\subset\bV_h$, the fully discrete homogenized equation \eqref{eq:fom-weak} may be tested by any $\bv_r\in\bV_r$. Subtracting the reduced equation gives
\begin{equation}
M_0\!\left(
\frac{(\bU_h^n-\bU_r^n)-(\bU_h^{n-1}-\bU_r^{n-1})}{\dt},\bv_r\right)
+B_0(\bU_h^n-\bU_r^n,\bv_r)=0.
\label{eq:fom-rom-orthogonality}
\end{equation}
Using \eqref{eq:error-split} in \eqref{eq:fom-rom-orthogonality} produces four projection and evolution terms. Since $P_r$ is the time-independent $M_0$ projection,
\begin{align}
M_0\!\left(\frac{\bm\eta^n-\bm\eta^{n-1}}{\dt},\bv_r\right)
&=\frac1{\dt}\left[
M_0(\bm\eta^n,\bv_r)-M_0(\bm\eta^{n-1},\bv_r)
\right]
=0.
\label{eq:eta-difference-orthogonality}
\end{align}
Consequently,
\begin{equation}
M_0\!\left(\frac{\bm\theta^n-\bm\theta^{n-1}}{\dt},\bv_r\right)
+B_0(\bm\theta^n,\bv_r)
=-B_0(\bm\eta^n,\bv_r),
\qquad\forall\bv_r\in\bV_r.
\label{eq:theta-equation}
\end{equation}

Taking $\bv_r=\bm\theta^n$, multiplying by $2\dt$, and using the exact polarization identity gives
\begin{align}
&\norm{\bm\theta^n}_M^2-\norm{\bm\theta^{n-1}}_M^2
+\norm{\bm\theta^n-\bm\theta^{n-1}}_M^2
+2\dt B_0(\bm\theta^n,\bm\theta^n)
\notag\\
&\qquad=-2\dt B_0(\bm\eta^n,\bm\theta^n).
\label{eq:theta-polarized-identity}
\end{align}
Lemma~\ref{lem:garding} and Young's inequality imply
\begin{equation}
2\left|B_0(\bm\eta^n,\bm\theta^n)\right|
\le2C_B\norm{\bm\eta^n}_{\bV}\norm{\bm\theta^n}_{\bV}
\le\alpha\norm{\bm\theta^n}_{\bV}^2
+\frac{C_B^2}{\alpha}\norm{\bm\eta^n}_{\bV}^2.
\label{eq:theta-young-bound}
\end{equation}
Substitution into \eqref{eq:theta-polarized-identity} yields the one-step estimate
\begin{align}
&\norm{\bm\theta^n}_M^2-\norm{\bm\theta^{n-1}}_M^2
+\norm{\bm\theta^n-\bm\theta^{n-1}}_M^2
+\alpha\dt\norm{\bm\theta^n}_{\bV}^2
\notag\\
&\qquad\le2\beta\dt\norm{\bm\theta^n}_M^2
+\frac{C_B^2}{\alpha}\dt\norm{\bm\eta^n}_{\bV}^2.
\label{eq:theta-one-step}
\end{align}

The initial condition gives $\bm\theta^0=0$. Summing \eqref{eq:theta-one-step} from $n=1$ to $j$ and discarding the nonnegative difference term gives
\begin{align}
\norm{\bm\theta^j}_M^2
+\alpha\dt\sum_{n=1}^j\norm{\bm\theta^n}_{\bV}^2
&\le2\beta\dt\sum_{n=1}^j\norm{\bm\theta^n}_M^2
+\frac{C_B^2}{\alpha}\dt\sum_{n=1}^j\norm{\bm\eta^n}_{\bV}^2.
\label{eq:theta-summed}
\end{align}
Separating the last term of the first sum on the right gives
\begin{align}
&(1-2\beta\dt)\norm{\bm\theta^j}_M^2
+\alpha\dt\sum_{n=1}^j\norm{\bm\theta^n}_{\bV}^2
\notag\\
&\qquad\le2\beta\dt\sum_{n=1}^{j-1}\norm{\bm\theta^n}_M^2
+\frac{C_B^2}{\alpha}\dt\sum_{n=1}^j\norm{\bm\eta^n}_{\bV}^2.
\label{eq:theta-last-term-separated}
\end{align}
For $\beta>0$, set
\begin{equation}
\gamma=\frac{2\beta}{1-2\beta\dt},
\qquad
C_1=\frac{C_B^2}{\alpha(1-2\beta\dt)}.
\label{eq:gronwall-constants}
\end{equation}
The step-size margin makes these constants uniformly bounded. Dropping the nonnegative $\bV$ term in \eqref{eq:theta-last-term-separated} gives
\begin{equation}
\norm{\bm\theta^j}_M^2
\le\gamma\dt\sum_{n=1}^{j-1}\norm{\bm\theta^n}_M^2
+C_1\dt\sum_{n=1}^j\norm{\bm\eta^n}_{\bV}^2.
\label{eq:theta-gronwall-form}
\end{equation}
The discrete Gronwall inequality, applied to the nondecreasing last term, gives
\begin{equation}
\max_{1\le j\le m}\norm{\bm\theta^j}_M^2
\le C_T\dt\sum_{n=1}^m\norm{\bm\eta^n}_{\bV}^2,
\qquad
C_T\le C_1\exp(\gamma T).
\label{eq:theta-mass-bound}
\end{equation}

Returning to \eqref{eq:theta-summed} with $j=m$ and using
\begin{equation}
\dt\sum_{n=1}^m\norm{\bm\theta^n}_M^2
\le T\max_{1\le n\le m}\norm{\bm\theta^n}_M^2
\label{eq:theta-time-sum-mass}
\end{equation}
together with \eqref{eq:theta-mass-bound}, we obtain
\begin{equation}
\dt\sum_{n=1}^m\norm{\bm\theta^n}_{\bV}^2
\le C_T\dt\sum_{n=1}^m\norm{\bm\eta^n}_{\bV}^2.
\label{eq:theta-energy-bound}
\end{equation}
Finally, $\mathbf e^n=\bm\eta^n+\bm\theta^n$ and the squared triangle inequality give
\begin{align}
&\max_n\norm{\mathbf e^n}_M^2
+\dt\sum_{n=1}^m\norm{\mathbf e^n}_{\bV}^2
\notag\\
&\quad\le
2\max_n\norm{\bm\eta^n}_M^2
+2\dt\sum_{n=1}^m\norm{\bm\eta^n}_{\bV}^2
+2\max_n\norm{\bm\theta^n}_M^2
+2\dt\sum_{n=1}^m\norm{\bm\theta^n}_{\bV}^2
\notag\\
&\quad\le C_T\mathcal P_r,
\label{eq:total-homogenized-error}
\end{align}
which proves \eqref{eq:fom-rom-bound}. When $\beta=0$, \eqref{eq:theta-one-step} may be summed directly, so neither Gronwall nor a time-step restriction is needed.
\end{proof}

\subsection{Stability of the multi-continuum recovery}
\label{subsec:recovery-stability}

For a discrete trajectory $\{\bv_h^n\}_{n=0}^m$, define
\begin{align}
(\mathcal R_{\eps,h}^{(2)}\bv_h^n)_l
={}&v_{l,h}^n
+\eps N_l^\alpha\frac{\partial v_{l,h}^n}{\partial x_\alpha}
+\eps M_l(v_{\bar l,h}^n-v_{l,h}^n)
\notag\\
&+\eps^2N_l^{\alpha\beta}
\frac{\partial^2 v_{l,h}^n}{\partial x_\alpha\partial x_\beta}
+\eps^2G_l\frac{v_{l,h}^n-v_{l,h}^{n-1}}{\dt}
+\eps^2C_l^\alpha\frac{\partial v_{l,h}^n}{\partial x_\alpha}
\notag\\
&+\eps^2F_l^\alpha\frac{\partial v_{\bar l,h}^n}{\partial x_\alpha}
+\eps^2K_l(v_{\bar l,h}^n-v_{l,h}^n),
\qquad l=1,2,
\label{eq:recovery-operator}
\end{align}
where every cell function is evaluated at $\bm x/\eps$. The second derivatives displayed below are evaluated directly when $P_2$ finite elements are used. For a $P_1$ discretization, the first derivatives are first interpolated and the resulting fields are then differentiated once more. Thus, the displayed expressions denote the numerically evaluated second derivatives required by the second-order recovery. Define
\begin{align}
\norm{\bv_h^n}_{\mathcal R,\eps}^2
={}&\norm{\bv_h^n}_M^2
+\eps^2\norm{\bv_h^n}_{\bV}^2
+\eps^2\norm{v_{1,h}^n-v_{2,h}^n}_{L^2(\Omega)}^2
\notag\\
&+\eps^4\sum_{l=1}^2\sum_{\alpha,\beta=1}^d
\left\|\frac{\partial^2 v_{l,h}^n}
{\partial x_\alpha\partial x_\beta}\right\|_{L^2(\Omega)}^2
+\eps^4\left\|\frac{\bv_h^n-\bv_h^{n-1}}{\dt}\right\|_M^2.
\label{eq:recovery-norm}
\end{align}

\begin{proposition}[Recovery stability]
\label{prop:recovery-stability}
Assume $0<\eps\le1$ and that all cell functions in \eqref{eq:recovery-operator} are bounded in $L^\infty(Y)$ by a common constant $C_{\rm cell}$. Then
\begin{equation}
\norm{\mathcal R_{\eps,h}^{(2)}\bv_h^n}_{L^2(\Omega)^2}
\le C_R\norm{\bv_h^n}_{\mathcal R,\eps},
\label{eq:recovery-stability}
\end{equation}
where $C_R$ depends on $C_{\rm cell}$, the dimension, and the mass-norm equivalence constants, but is independent of $h$, $\dt$, and the trajectory.
\end{proposition}

\begin{proof}
For each $l$, the principal first-order terms satisfy
\begin{align}
\norm{\eps N_l^\alpha\frac{\partial v_{l,h}^n}{\partial x_\alpha}}_{L^2}
&\le C\eps\left(\sum_{j=1}^d
\left\|\frac{\partial v_{l,h}^n}{\partial x_j}\right\|_{L^2}^2\right)^{1/2},
\notag\\
\norm{\eps M_l(v_{\bar l,h}^n-v_{l,h}^n)}_{L^2}
&\le C\eps\norm{v_{\bar l,h}^n-v_{l,h}^n}_{L^2}.
\label{eq:recovery-first-order-bounds}
\end{align}
The second-order terms satisfy
\begin{align}
\left\|\eps^2N_l^{\alpha\beta}
\frac{\partial^2 v_{l,h}^n}{\partial x_\alpha\partial x_\beta}\right\|_{L^2}
&\le C\eps^2\left(\sum_{\alpha,\beta=1}^d
\left\|\frac{\partial^2 v_{l,h}^n}
{\partial x_\alpha\partial x_\beta}\right\|_{L^2}^2\right)^{1/2},
\notag\\
\left\|\eps^2G_l\frac{v_{l,h}^n-v_{l,h}^{n-1}}{\dt}\right\|_{L^2}
&\le C\eps^2\left\|\frac{v_{l,h}^n-v_{l,h}^{n-1}}{\dt}\right\|_{L^2},
\notag\\
\norm{\eps^2C_l^\alpha\frac{\partial v_{l,h}^n}{\partial x_\alpha}}_{L^2}
+\norm{\eps^2F_l^\alpha\frac{\partial v_{\bar l,h}^n}{\partial x_\alpha}}_{L^2}
&\le C\eps^2\norm{\bv_h^n}_{\bV},
\notag\\
\norm{\eps^2K_l(v_{\bar l,h}^n-v_{l,h}^n)}_{L^2}
&\le C\eps^2\norm{v_{\bar l,h}^n-v_{l,h}^n}_{L^2}.
\label{eq:recovery-second-order-bounds}
\end{align}
For a fixed finite number of terms, $(\sum_i a_i)^2\le C\sum_i a_i^2$. Summing the squared componentwise bounds gives
\begin{align}
\norm{\mathcal R_{\eps,h}^{(2)}\bv_h^n}_{L^2(\Omega)^2}^2
\le C\big[
&\norm{\bv_h^n}_{L^2(\Omega)^2}^2
+\eps^2\norm{\bv_h^n}_{\bV}^2
+\eps^2\norm{v_{1,h}^n-v_{2,h}^n}_{L^2}^2
\notag\\
&+\eps^4\sum_{l=1}^2\sum_{\alpha,\beta=1}^d
\left\|\frac{\partial^2 v_{l,h}^n}
{\partial x_\alpha\partial x_\beta}\right\|_{L^2}^2
+\eps^4\left\|\frac{\bv_h^n-\bv_h^{n-1}}{\dt}\right\|_{L^2(\Omega)^2}^2
\notag\\
&+\eps^4\norm{\bv_h^n}_{\bV}^2
+\eps^4\norm{v_{1,h}^n-v_{2,h}^n}_{L^2}^2
\big].
\label{eq:recovery-squared-bound}
\end{align}
Because $0<\eps\le1$, the last two terms are absorbed by their $\eps^2$ counterparts. The mass equivalence \eqref{eq:mass-equivalence}, applied also to $(\bv_h^n-\bv_h^{n-1})/\dt$, converts the remaining $L^2$ terms into mass norms and proves \eqref{eq:recovery-stability}.
\end{proof}

It remains to control the recovery norm on the retained space. Write
\begin{equation}
\bm\theta^n=\sum_{k=1}^r a_k^n\bm\phi_k.
\label{eq:theta-reduced-expansion}
\end{equation}
Define the reduced Gram matrices
\begin{align}
(\mathbf S_r)_{ij}
&=(\bm\phi_j,\bm\phi_i)_{\bV},
\notag\\
(\mathbf E_r)_{ij}
&=(\phi_{1,j}-\phi_{2,j},\phi_{1,i}-\phi_{2,i})_{L^2(\Omega)},
\notag\\
(\mathbf S_r^{(2)})_{ij}
&=\sum_{l=1}^2\sum_{\alpha,\beta=1}^d
\left(\frac{\partial^2\phi_{l,j}}{\partial x_\alpha\partial x_\beta},
\frac{\partial^2\phi_{l,i}}{\partial x_\alpha\partial x_\beta}\right)_{L^2(\Omega)}.
\label{eq:recovery-gram-matrices}
\end{align}

\begin{lemma}[Retained-space recovery bound]
\label{lem:retained-recovery}
For the $M_0$-orthonormal retained modes,
\begin{equation}
\max_{1\le n\le m}\norm{\bm\theta^n}_{\mathcal R,\eps}^2
\le C_{r,\eps}
\left[
\max_{0\le n\le m}\norm{\bm\theta^n}_M^2
+\dt\sum_{n=1}^m\norm{\bm\theta^n}_{\bV}^2
\right],
\label{eq:retained-recovery-stability}
\end{equation}
where a conservative computable choice is
\begin{align}
C_{r,\eps}
={}&1+\eps^2\lambda_{\max}(\mathbf S_r)
+\eps^2\lambda_{\max}(\mathbf E_r)
+\eps^4\lambda_{\max}(\mathbf S_r^{(2)})
+\frac{4\eps^4}{\dt^2}.
\label{eq:retained-recovery-constant}
\end{align}
The notation suppresses the explicit $\dt$ dependence displayed in \eqref{eq:retained-recovery-constant}.
\end{lemma}

\begin{proof}
Since $\bm\theta^n\in\bV_r$ for $0\le n\le m$, its expansion \eqref{eq:theta-reduced-expansion} has coefficient vector
\begin{equation}
\mathbf a^n=(a_1^n,\ldots,a_r^n)^T,
\qquad
a_k^n=M_0(\bm\theta^n,\bm\phi_k),
\quad k=1,\ldots,r.
\label{eq:retained-coefficient-vector}
\end{equation}
The $M_0$-orthonormality relation $M_0(\bm\phi_j,\bm\phi_i)=\delta_{ij}$ therefore gives
\begin{align}
\norm{\bm\theta^n}_M^2
&=M_0(\bm\theta^n,\bm\theta^n)
\notag\\
&=\sum_{i,j=1}^r a_i^na_j^nM_0(\bm\phi_j,\bm\phi_i)
=\sum_{k=1}^r(a_k^n)^2
=|\mathbf a^n|_2^2.
\label{eq:retained-mass-coefficients}
\end{align}

Each matrix in \eqref{eq:recovery-gram-matrices} is a real symmetric positive-semidefinite Gram matrix. Consequently, the Rayleigh quotient estimate gives
\begin{equation}
\mathbf a^T\mathbf G\mathbf a
\le\lambda_{\max}(\mathbf G)|\mathbf a|_2^2,
\qquad
\mathbf G\in\{\mathbf S_r,\mathbf E_r,\mathbf S_r^{(2)}\},
\quad\mathbf a\in\R^r.
\label{eq:retained-rayleigh-bound}
\end{equation}
We apply this estimate separately to the gradient, inter-continuum difference, and second-derivative contributions.

First, the definition of $\mathbf S_r$ and \eqref{eq:retained-mass-coefficients} imply
\begin{align}
\norm{\bm\theta^n}_{\bV}^2
&=(\bm\theta^n,\bm\theta^n)_{\bV}
=\sum_{i,j=1}^r a_i^na_j^n(\bm\phi_j,\bm\phi_i)_{\bV}
\notag\\
&=(\mathbf a^n)^T\mathbf S_r\mathbf a^n
\le\lambda_{\max}(\mathbf S_r)|\mathbf a^n|_2^2
=\lambda_{\max}(\mathbf S_r)\norm{\bm\theta^n}_M^2.
\label{eq:retained-gradient-gram-bound}
\end{align}
For the inter-continuum difference, the componentwise expansion is
\begin{equation}
\theta_1^n-\theta_2^n
=\sum_{k=1}^r a_k^n(\phi_{1,k}-\phi_{2,k}).
\label{eq:retained-continuum-difference-expansion}
\end{equation}
Using the definition of $\mathbf E_r$, we obtain
\begin{align}
\norm{\theta_1^n-\theta_2^n}_{L^2}^2
&=\sum_{i,j=1}^r a_i^na_j^n
(\phi_{1,j}-\phi_{2,j},\phi_{1,i}-\phi_{2,i})_{L^2}
\notag\\
&=(\mathbf a^n)^T\mathbf E_r\mathbf a^n
\le\lambda_{\max}(\mathbf E_r)|\mathbf a^n|_2^2
=\lambda_{\max}(\mathbf E_r)\norm{\bm\theta^n}_M^2.
\label{eq:retained-exchange-gram-bound}
\end{align}

The numerical second-derivative evaluation described after \eqref{eq:recovery-operator} is linear on the fixed discrete space, both for direct $P_2$ differentiation and for differentiation after interpolation of the $P_1$ first derivatives. Thus,
\begin{equation}
\frac{\partial^2\theta_l^n}{\partial x_\alpha\partial x_\beta}
=\sum_{k=1}^r a_k^n
\frac{\partial^2\phi_{l,k}}{\partial x_\alpha\partial x_\beta},
\qquad l=1,2,\quad\alpha,\beta=1,\ldots,d.
\label{eq:retained-second-derivative-expansion}
\end{equation}
Expanding the squared norms, summing over both continua and all derivative directions, and using the definition of $\mathbf S_r^{(2)}$ gives
\begin{align}
\sum_{l=1}^2\sum_{\alpha,\beta=1}^d
\left\|\frac{\partial^2\theta_l^n}
{\partial x_\alpha\partial x_\beta}\right\|_{L^2}^2
&=\sum_{i,j=1}^r a_i^na_j^n
\sum_{l=1}^2\sum_{\alpha,\beta=1}^d
\left(\frac{\partial^2\phi_{l,j}}{\partial x_\alpha\partial x_\beta},
\frac{\partial^2\phi_{l,i}}{\partial x_\alpha\partial x_\beta}\right)_{L^2}
=(\mathbf a^n)^T\mathbf S_r^{(2)}\mathbf a^n
\notag\\
&\le\lambda_{\max}(\mathbf S_r^{(2)})|\mathbf a^n|_2^2
=\lambda_{\max}(\mathbf S_r^{(2)})\norm{\bm\theta^n}_M^2.
\label{eq:retained-gram-bounds}
\end{align}

Next, because the retained POD basis is independent of time,
\begin{equation}
\frac{\bm\theta^n-\bm\theta^{n-1}}{\dt}
=\sum_{k=1}^r\frac{a_k^n-a_k^{n-1}}{\dt}\bm\phi_k,
\qquad 1\le n\le m.
\label{eq:retained-time-difference-expansion}
\end{equation}
Applying mass orthonormality to this expansion and then the elementary inequality $|\mathbf x-\mathbf y|_2^2\le2|\mathbf x|_2^2+2|\mathbf y|_2^2$ yields
\begin{align}
\left\|\frac{\bm\theta^n-\bm\theta^{n-1}}{\dt}\right\|_M^2
&=\left|\frac{\mathbf a^n-\mathbf a^{n-1}}{\dt}\right|_2^2
\le\frac{2}{\dt^2}
\left(|\mathbf a^n|_2^2+|\mathbf a^{n-1}|_2^2\right)
=\frac{2}{\dt^2}
\left(\norm{\bm\theta^n}_M^2+\norm{\bm\theta^{n-1}}_M^2\right)
\notag\\
&\le\frac{4}{\dt^2}\max_{0\le j\le m}\norm{\bm\theta^j}_M^2.
\label{eq:retained-time-difference}
\end{align}

We now substitute these bounds into the recovery norm. For $\bm\theta^n$, definition \eqref{eq:recovery-norm} reads
\begin{align}
\norm{\bm\theta^n}_{\mathcal R,\eps}^2
={}&\norm{\bm\theta^n}_M^2
+\eps^2\norm{\bm\theta^n}_{\bV}^2
+\eps^2\norm{\theta_1^n-\theta_2^n}_{L^2}^2
\notag\\
&+\eps^4\sum_{l=1}^2\sum_{\alpha,\beta=1}^d
\left\|\frac{\partial^2\theta_l^n}{\partial x_\alpha\partial x_\beta}\right\|_{L^2}^2
+\eps^4\left\|\frac{\bm\theta^n-\bm\theta^{n-1}}{\dt}\right\|_M^2.
\label{eq:retained-recovery-norm-expanded}
\end{align}
Applying \eqref{eq:retained-gradient-gram-bound}, \eqref{eq:retained-exchange-gram-bound}, \eqref{eq:retained-gram-bounds}, and \eqref{eq:retained-time-difference} term by term yields
\begin{align}
\norm{\bm\theta^n}_{\mathcal R,\eps}^2
\le{}&\bigl[1+\eps^2\lambda_{\max}(\mathbf S_r)
+\eps^2\lambda_{\max}(\mathbf E_r)
+\eps^4\lambda_{\max}(\mathbf S_r^{(2)})\bigr]
\norm{\bm\theta^n}_M^2
+\frac{4\eps^4}{\dt^2}
\max_{0\le j\le m}\norm{\bm\theta^j}_M^2
\notag\\
\le{}&C_{r,\eps}\max_{0\le j\le m}\norm{\bm\theta^j}_M^2,
\label{eq:retained-recovery-pointwise-bound}
\end{align}
where the last step uses $\norm{\bm\theta^n}_M^2\le\max_{0\le j\le m}\norm{\bm\theta^j}_M^2$ and the explicit constant in \eqref{eq:retained-recovery-constant}. Since this constant is independent of the time index $n$, taking the maximum gives the stronger estimate
\begin{equation}
\max_{1\le n\le m}\norm{\bm\theta^n}_{\mathcal R,\eps}^2
\le C_{r,\eps}\max_{0\le j\le m}\norm{\bm\theta^j}_M^2.
\label{eq:retained-recovery-mass-only-bound}
\end{equation}
The cumulative term $\dt\sum_{n=1}^m\norm{\bm\theta^n}_{\bV}^2$ is nonnegative. Adding it inside the bracket on the right of \eqref{eq:retained-recovery-mass-only-bound} proves \eqref{eq:retained-recovery-stability}. This term is retained in the lemma statement to match the discrete energy quantity controlled by the intermediate estimates \eqref{eq:theta-mass-bound} and \eqref{eq:theta-energy-bound} in the proof of Theorem~\ref{thm:fom-rom}.

The factor $4\eps^4/\dt^2$ is a conservative consequence of \eqref{eq:retained-time-difference}; the bound is not asserted to be uniform as $\dt\to0$. For a particular computed trajectory, a sharper time-difference contribution may instead be evaluated from the actual reduced coefficient differences in \eqref{eq:retained-time-difference-expansion}.
\end{proof}

\subsection{Recovery-weighted discarded POD tail}
\label{subsec:discarded-tail}

Define
\begin{equation}
\mathcal Q_{\eps,r}
=\max_{1\le n\le m}\norm{\bm\eta^n}_{\mathcal R,\eps}^2,
\qquad
\mathcal D_{t,r}
=\max_{1\le n\le m}
\left\|\frac{\bm\eta^n-\bm\eta^{n-1}}{\dt}\right\|_M^2.
\label{eq:Qr-Dtr-definition}
\end{equation}

\begin{lemma}[Weighted discarded-mode estimate]
\label{lem:weighted-tail}
The recovery-weighted tail satisfies
\begin{align}
\mathcal Q_{\eps,r}
\le{}&
m\sum_{k=r+1}^s\lambda_k
+m\eps^2\sum_{k=r+1}^s
\lambda_k\norm{\bm\phi_k}_{\bV}^2
\notag\\
&+m\eps^2\sum_{k=r+1}^s
\lambda_k\norm{\phi_{1,k}-\phi_{2,k}}_{L^2(\Omega)}^2
\notag\\
&+m\eps^4\sum_{k=r+1}^s
\lambda_k\sum_{l=1}^2\sum_{\alpha,\beta=1}^d
\left\|\frac{\partial^2\phi_{l,k}}
{\partial x_\alpha\partial x_\beta}\right\|_{L^2(\Omega)}^2
+\eps^4\mathcal D_{t,r}.
\label{eq:weighted-tail-bound}
\end{align}
\end{lemma}

\begin{proof}
The definition \eqref{eq:recovery-norm} gives
\begin{align}
\mathcal Q_{\eps,r}
\le{}&
\max_n\norm{\bm\eta^n}_M^2
+\eps^2\max_n\norm{\bm\eta^n}_{\bV}^2
+\eps^2\max_n\norm{\eta_1^n-\eta_2^n}_{L^2}^2
\notag\\
&+\eps^4\max_n\sum_{l=1}^2\sum_{\alpha,\beta=1}^d
\left\|\frac{\partial^2\eta_l^n}
{\partial x_\alpha\partial x_\beta}\right\|_{L^2}^2
+\eps^4\mathcal D_{t,r}.
\label{eq:weighted-tail-components}
\end{align}
For each of the first four terms, the maximum is bounded by the sum over $n$. Applying the coefficient covariance \eqref{eq:pod-coefficient-covariance} to the mass, $\bV$, continuum-difference, and recovered-second-derivative Gram matrices gives, respectively,
\begin{align}
\max_n\norm{\bm\eta^n}_M^2
&\le m\sum_{k=r+1}^s\lambda_k,
\notag\\
\max_n\norm{\bm\eta^n}_{\bV}^2
&\le m\sum_{k=r+1}^s\lambda_k\norm{\bm\phi_k}_{\bV}^2,
\notag\\
\max_n\norm{\eta_1^n-\eta_2^n}_{L^2}^2
&\le m\sum_{k=r+1}^s
\lambda_k\norm{\phi_{1,k}-\phi_{2,k}}_{L^2}^2,
\notag\\
\max_n\sum_{l=1}^2\sum_{\alpha,\beta=1}^d
\left\|\frac{\partial^2\eta_l^n}
{\partial x_\alpha\partial x_\beta}\right\|_{L^2}^2
&\le m\sum_{k=r+1}^s
\lambda_k\sum_{l=1}^2\sum_{\alpha,\beta=1}^d
\left\|\frac{\partial^2\phi_{l,k}}
{\partial x_\alpha\partial x_\beta}\right\|_{L^2}^2.
\label{eq:weighted-tail-component-bounds}
\end{align}
Substitution into \eqref{eq:weighted-tail-components} proves \eqref{eq:weighted-tail-bound}.
\end{proof}

Because $P_r$ is time independent,
\begin{equation}
\frac{\bm\eta^n-\bm\eta^{n-1}}{\dt}
=(I-P_r)\frac{\bU_h^n-\bU_h^{n-1}}{\dt}.
\label{eq:derivative-projection-error}
\end{equation}
Since the state-POD eigenvalues characterize $\bU_h^n$ rather than its discrete time derivative, $\mathcal D_{t,r}$ cannot in general be represented by the standard POD eigenvalue tail and is therefore retained explicitly in \eqref{eq:weighted-tail-bound}.

\subsection{Recovered POD-HOMS error}
\label{subsec:recovered-error}

\begin{proposition}[HOMS consistency for the discrete homogenized solution]
\label{prop:homs-consistency}
Let $\bu_h^{\eps,n}=\mathcal R_{\eps,h}^{(2)}\bU_h^n$. Under the regularity, cell-alignment, and second-derivative assumptions of the underlying multi-continuum HOMS analysis~\cite{dong2026homs},
\begin{equation}
\max_{1\le n\le m}
\norm{\bu^\eps(t_n)-\bu_h^{\eps,n}}_{L^2(\Omega)^2}^2
\le C_{\rm HOMS}\left(\eps^2+h^{2q}+\dt^2\right).
\label{eq:homs-consistency}
\end{equation}
The constant is independent of $\eps$, $h$, and $\dt$ in the admissible asymptotic and discretization regime.
\end{proposition}

This proposition represents the inherited HOMS consistency and discretization contribution associated with the discrete homogenized solution. The new POD analysis concerns the difference between $\bu_h^{\eps,n}$ and the recovered POD-reduced field.

\begin{theorem}[Recovered multi-continuum POD-HOMS error]
\label{thm:recovered-error}
Let $\bu_r^{\eps,n}=\mathcal R_{\eps,h}^{(2)}\bU_r^n$. Under the hypotheses of Theorem~\ref{thm:fom-rom}, Proposition~\ref{prop:recovery-stability}, Lemma~\ref{lem:retained-recovery}, and Proposition~\ref{prop:homs-consistency},
\begin{equation}
\max_{1\le n\le m}
\norm{\bu^\eps(t_n)-\bu_r^{\eps,n}}_{L^2(\Omega)^2}^2
\le C\left(
\eps^2+h^{2q}+\dt^2
+\mathcal Q_{\eps,r}
+C_{r,\eps}\mathcal P_r
\right).
\label{eq:recovered-error}
\end{equation}
\end{theorem}

\begin{proof}
At time $t_n$, add and subtract the recovered discrete-HOM field $\bu_h^{\eps,n}$. Linearity of the recovery along the complete discrete trajectory gives
\begin{align}
\bu^\eps(t_n)-\bu_r^{\eps,n}
={}&\left(\bu^\eps(t_n)-\bu_h^{\eps,n}\right)
+\mathcal R_{\eps,h}^{(2)}\bm\eta^n
+\mathcal R_{\eps,h}^{(2)}\bm\theta^n.
\label{eq:recovered-three-term-split}
\end{align}
For the time-derivative corrector, this linearity uses
\begin{equation}
\frac{(\bU_h^n-\bU_r^n)-(\bU_h^{n-1}-\bU_r^{n-1})}{\dt}
=\frac{\bm\eta^n-\bm\eta^{n-1}}{\dt}
+\frac{\bm\theta^n-\bm\theta^{n-1}}{\dt}.
\label{eq:recovery-time-linearity}
\end{equation}
The three-term squared triangle inequality applied to \eqref{eq:recovered-three-term-split} yields
\begin{align}
\norm{\bu^\eps(t_n)-\bu_r^{\eps,n}}_{L^2}^2
\le{}&
3\norm{\bu^\eps(t_n)-\bu_h^{\eps,n}}_{L^2}^2
+3\norm{\mathcal R_{\eps,h}^{(2)}\bm\eta^n}_{L^2}^2
\notag\\
&+3\norm{\mathcal R_{\eps,h}^{(2)}\bm\theta^n}_{L^2}^2.
\label{eq:recovered-three-term-inequality}
\end{align}

Proposition~\ref{prop:homs-consistency} bounds the first term:
\begin{equation}
\max_n\norm{\bu^\eps(t_n)-\bu_h^{\eps,n}}_{L^2}^2
\le C_{\rm HOMS}(\eps^2+h^{2q}+\dt^2).
\label{eq:recovered-full-order-term}
\end{equation}
Proposition~\ref{prop:recovery-stability} and \eqref{eq:Qr-Dtr-definition} bound the discarded component:
\begin{equation}
\max_n\norm{\mathcal R_{\eps,h}^{(2)}\bm\eta^n}_{L^2}^2
\le C_R^2\mathcal Q_{\eps,r}.
\label{eq:recovered-discarded-term}
\end{equation}
The retained component satisfies, by Proposition~\ref{prop:recovery-stability} and Lemma~\ref{lem:retained-recovery},
\begin{align}
\max_n\norm{\mathcal R_{\eps,h}^{(2)}\bm\theta^n}_{L^2}^2
&\le C_R^2\max_n\norm{\bm\theta^n}_{\mathcal R,\eps}^2
\notag\\
&\le C_R^2C_{r,\eps}
\left[
\max_n\norm{\bm\theta^n}_M^2
+\dt\sum_{n=1}^m\norm{\bm\theta^n}_{\bV}^2
\right]
\notag\\
&\le C_R^2C_{r,\eps}C_T\mathcal P_r,
\label{eq:recovered-retained-term}
\end{align}
where the last step uses the intermediate estimates \eqref{eq:theta-mass-bound} and \eqref{eq:theta-energy-bound} from Theorem~\ref{thm:fom-rom}.

Taking the maximum in \eqref{eq:recovered-three-term-inequality} and substituting \eqref{eq:recovered-full-order-term}, \eqref{eq:recovered-discarded-term}, and \eqref{eq:recovered-retained-term} gives
\begin{align}
&\max_n\norm{\bu^\eps(t_n)-\bu_r^{\eps,n}}_{L^2(\Omega)^2}^2
\notag\\
&\quad\le
3C_{\rm HOMS}(\eps^2+h^{2q}+\dt^2)
+3C_R^2\mathcal Q_{\eps,r}
+3C_R^2C_{r,\eps}C_T\mathcal P_r.
\label{eq:recovered-bound-with-constants}
\end{align}
Absorbing the fixed constants into a generic constant $C$ proves \eqref{eq:recovered-error}.
\end{proof}


\section{Efficient POD-HOMS algorithms}
\label{sec:algorithms}

This section recasts the formulation of Section~\ref{sec:model} as an implementable offline--online procedure. The two algorithms use the same cell solutions, effective coefficients, joint POD space, and reduced time integrator. They differ only in the organization of the second-order HOMS recovery: Algorithm~1 applies the complete recovery at every requested output time, whereas Algorithm~2 applies its spatial part to the retained POD modes once offline.

\subsection{Common offline stage and reduced dynamics}
\label{subsec:common-offline}

The common offline stage begins by solving the cell problems for $N_l^\alpha$, $M_l$, $N_l^{\alpha\beta}$, $G_l$, $C_l^\alpha$, $F_l^\alpha$, and $K_l$. These functions respectively provide the first-order spatial correction, continuum-exchange correction, second-gradient correction, time-derivative correction, same-continuum and cross-continuum gradient corrections, and higher-order continuum-difference correction. They are computed once on the representative cell. The effective coefficients $c_l^*$, $\kappa_l^*$, $\bm K_{lm}^*$, and $Q_l^*$ are then evaluated and used to assemble the coupled homogenized finite element system.

Using the algebraic coefficient vector $\mathbf U^n$ associated with $\bU_h^n$ in Section~\ref{subsec:block-system}, its fully discrete homogenized update has the block transient form
\begin{equation}
\left(\frac{\mathbf M}{\dt}+\mathbf A\right)\mathbf U^n
=\mathbf F^n+\frac{\mathbf M}{\dt}\mathbf U^{n-1},
\label{eq:algorithm-common-fom}
\end{equation}
with the matrices defined in \eqref{eq:block-matrices}. Although the two continuum equations are coupled, \eqref{eq:algorithm-common-fom} is a standard linear evolution problem and is therefore directly amenable to projection-based reduction.

Solving \eqref{eq:algorithm-common-fom} over the training interval gives the snapshot ensemble
\begin{equation*}
\mathcal S_h
=\left\{\bU_h^{n_j}\right\}_{j=1}^{N_s},
\qquad
\bU_h^{n_j}
=\begin{pmatrix}U_{1,h}^{n_j}\\ U_{2,h}^{n_j}\end{pmatrix}.
\end{equation*}
Each snapshot is treated as one coupled vector state. In particular, separate scalar POD spaces are not constructed for $u_1$ and $u_2$, because doing so would discard the snapshot-level correlation generated by inter-continuum exchange. The mass-weighted correlation problem \eqref{eq:pod-correlation} produces the joint modes
\begin{equation*}
\bm\phi_k
=\begin{pmatrix}\phi_{1,k}\\ \phi_{2,k}\end{pmatrix},
\qquad k=1,\ldots,r,
\end{equation*}
and the reduced state is $\bU_r^n=\sum_{k=1}^rz_k^n\bm\phi_k$. Projection of the block system yields
\begin{equation}
\left(\frac{\mathbf M_r}{\dt}+\mathbf A_r\right)\mathbf z^n
=\mathbf f_r^n+\frac{\mathbf M_r}{\dt}\mathbf z^{n-1},
\label{eq:algorithm-common-rom}
\end{equation}
where $(\mathbf M_r)_{ij}=M_0(\bm\phi_j,\bm\phi_i)$, $(\mathbf A_r)_{ij}=B_0(\bm\phi_j,\bm\phi_i)$, and $(\mathbf f_r^n)_i=F(\bm\phi_i;t_n)$. Thus both POD-HOMS algorithms produce exactly the same coefficient sequence $\{\mathbf z^n\}$; the computational paths separate only after the reduced homogenized update \eqref{eq:algorithm-common-rom}.

Figure~\ref{fig:pod-homs-workflow} summarizes the shared stages and the two recovery paths. Algorithm~1 retains the conventional time-by-time HOMS assembly. Algorithm~2 moves mode-dependent spatial recovery to the offline stage and leaves only coefficient combinations online.

\begin{figure}[htbp]
\centering
\includegraphics[width=\textwidth]{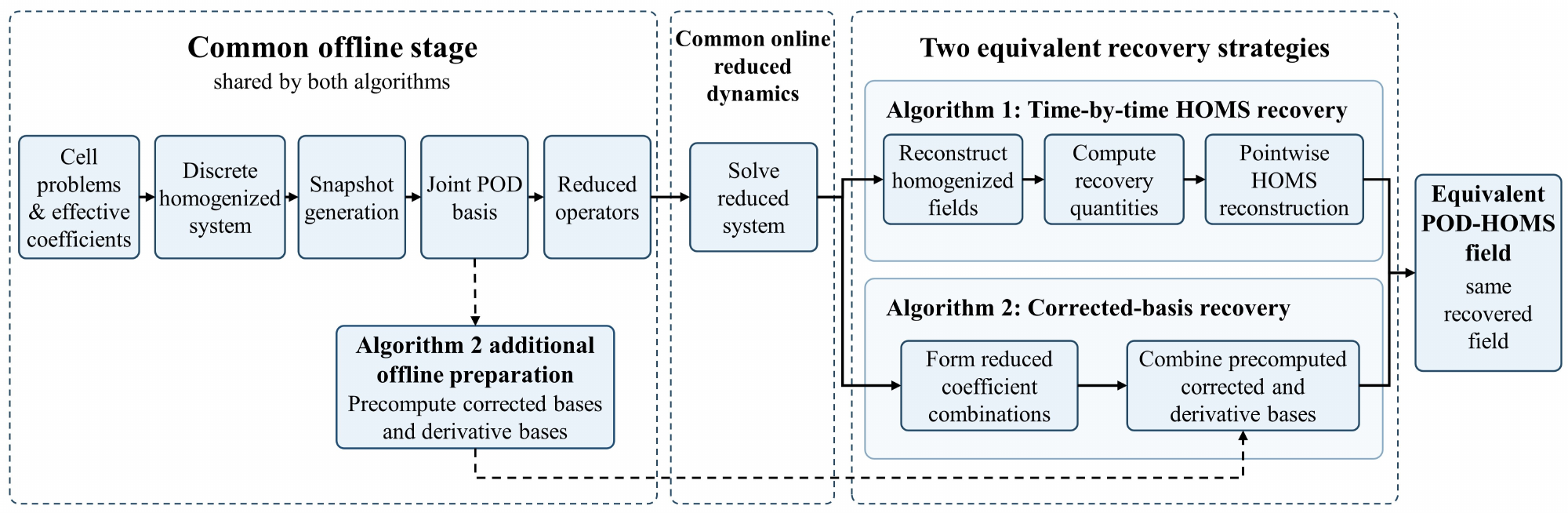}
\caption{Common offline stage, shared reduced dynamics, and recovery branches of the two POD-HOMS algorithms.}
\label{fig:pod-homs-workflow}
\end{figure}

\subsection{Algorithm 1: online HOMS reconstruction}
\label{subsec:algorithm1}

At time $t_n$, Algorithm~1 first solves \eqref{eq:algorithm-common-rom} and reconstructs each continuum component by
\begin{equation}
U_{l,r}^n=\sum_{k=1}^rz_k^n\phi_{l,k},
\qquad l=1,2.
\label{eq:algorithm1-macro-reconstruction}
\end{equation}
The quantities needed by the recovery operator are then recomputed from this field:
\begin{equation}
\left(\frac{\partial U_{l,r}^n}{\partial x_1},\ldots,
\frac{\partial U_{l,r}^n}{\partial x_d}\right),\qquad
\left\{\frac{\partial^2 U_{l,r}^n}
{\partial x_\alpha\partial x_\beta}\right\}_{\alpha,\beta=1}^d,
\qquad
U_{\bar l,r}^n-U_{l,r}^n,\qquad
\frac{U_{l,r}^n-U_{l,r}^{n-1}}{\dt}.
\label{eq:algorithm1-recovery-data}
\end{equation}
The reconstructed macrofields and the quantities in \eqref{eq:algorithm1-recovery-data} are interpolated from the homogenized mesh to the fine output mesh. At every output point $\bm x$, the local coordinate $\bm y=\bm x/\eps$ is formed and all cell functions are evaluated. The recovered component is then assembled as
\begin{align}
u_{l,r,{\rm A1}}^{\eps,n}
={}&U_{l,r}^n
+\eps N_l^\alpha\frac{\partial U_{l,r}^n}{\partial x_\alpha}
+\eps M_l(U_{\bar l,r}^n-U_{l,r}^n)
\notag\\
&+\eps^2N_l^{\alpha\beta}
\frac{\partial^2 U_{l,r}^n}{\partial x_\alpha\partial x_\beta}
+\eps^2G_l\frac{U_{l,r}^n-U_{l,r}^{n-1}}{\dt}
+\eps^2C_l^\alpha\frac{\partial U_{l,r}^n}{\partial x_\alpha}
\notag\\
&+\eps^2F_l^\alpha\frac{\partial U_{\bar l,r}^n}{\partial x_\alpha}
+\eps^2K_l(U_{\bar l,r}^n-U_{l,r}^n),
\qquad l=1,2,
\label{eq:algorithm1}
\end{align}
where the cell functions in \eqref{eq:algorithm1} are evaluated at $\bm x/\eps$. This implementation is direct and mirrors the standard HOMS recovery. Its drawback is organizational rather than mathematical: interpolation, derivative recovery, local cell evaluation, and pointwise assembly are repeated at every requested time level.

\begin{center}
\fbox{\begin{minipage}{0.94\textwidth}
\small
\textbf{Algorithm 1. POD-HOMS with online reconstruction}

\textbf{Offline stage}
\begin{enumerate}
\setlength{\itemsep}{0pt}
\item Solve all cell problems and compute the effective coefficients.
\item Assemble and solve the discrete homogenized system, generate coupled snapshots, and construct the joint POD basis.
\item Assemble $\mathbf M_r$, $\mathbf A_r$, and the reduced load representation.
\end{enumerate}
\textbf{Online stage: for $n=1,\ldots,N_t$}
\begin{enumerate}
\setlength{\itemsep}{0pt}
\item Solve \eqref{eq:algorithm-common-rom} for $\mathbf z^n$ and reconstruct $\bU_r^n$.
\item Compute first and second derivatives, continuum differences, and backward time differences.
\item Interpolate all required macro quantities to the fine output mesh.
\item Evaluate the cell functions at $\bm x/\eps$ and assemble \eqref{eq:algorithm1} pointwise.
\end{enumerate}
\end{minipage}}
\end{center}

\subsection{Algorithm 2: corrected-basis reconstruction}
\label{subsec:algorithm2}

Algorithm~2 changes only the order in which the linear recovery operations are performed. Since $U_{l,r}^n$ is a linear combination of retained POD modes, all terms paired with $z_k^n$ can be applied to each mode before time integration. Define the spatially corrected modes
\begin{align*}
\Phi_{l,k}^\eps
={}&\phi_{l,k}
+\eps N_l^\alpha\frac{\partial\phi_{l,k}}{\partial x_\alpha}
+\eps M_l(\phi_{\bar l,k}-\phi_{l,k})
+\eps^2N_l^{\alpha\beta}
\frac{\partial^2\phi_{l,k}}{\partial x_\alpha\partial x_\beta}
\\
&+\eps^2C_l^\alpha\frac{\partial\phi_{l,k}}{\partial x_\alpha}
+\eps^2F_l^\alpha\frac{\partial\phi_{\bar l,k}}{\partial x_\alpha}
+\eps^2K_l(\phi_{\bar l,k}-\phi_{l,k})
\end{align*}
and derivative bases
\begin{equation}
\Psi_{l,k}^\eps(\bm x)=G_l(\bm x/\eps)\phi_{l,k}(\bm x).
\label{eq:derivative-mode}
\end{equation}
For compact vector notation, set
$\bm\Phi_k^\eps=(\Phi_{1,k}^\eps,\Phi_{2,k}^\eps)^T$ and
$\bm\Psi_k^\eps=(\Psi_{1,k}^\eps,\Psi_{2,k}^\eps)^T$.
The corrected modes $\Phi_{l,k}^\eps$ contain the value, gradient, second-gradient, cross-gradient, and continuum-difference contributions. They are interpolated and stored directly on the fine output mesh. Consequently, the corresponding recovery operations are performed once for each retained mode rather than once for every mode at every output time.

The time-derivative correction must be kept separate. Indeed, \eqref{eq:algorithm1-macro-reconstruction} gives
\begin{equation*}
\frac{U_{l,r}^n-U_{l,r}^{n-1}}{\dt}
=\sum_{k=1}^r\frac{z_k^n-z_k^{n-1}}{\dt}\phi_{l,k},
\end{equation*}
and therefore
\begin{equation}
\eps^2G_l\frac{U_{l,r}^n-U_{l,r}^{n-1}}{\dt}
=\eps^2\sum_{k=1}^r
\frac{z_k^n-z_k^{n-1}}{\dt}\Psi_{l,k}^\eps.
\label{eq:derivative-basis-combination}
\end{equation}
The coefficients in \eqref{eq:derivative-basis-combination} are time differences of $z_k^n$, not the current coefficients themselves. Thus the term $G_l(U_{l,r}^n-U_{l,r}^{n-1})/\dt$ cannot be absorbed into the spatially corrected basis that is paired with $z_k^n$; this is why Algorithm~2 requires the additional derivative basis \eqref{eq:derivative-mode}.

The online field is
\begin{equation}
u_{l,r,{\rm A2}}^{\eps,n}
=\sum_{k=1}^rz_k^n\Phi_{l,k}^\eps
+\eps^2\sum_{k=1}^r\frac{z_k^n-z_k^{n-1}}{\dt}\Psi_{l,k}^\eps,
\qquad l=1,2.
\label{eq:algorithm2}
\end{equation}
Once $\Phi_{l,k}^\eps$ and $\Psi_{l,k}^\eps$ have been stored, the online recovery consists only of the two linear combinations in \eqref{eq:algorithm2}; neither macrofield reconstruction nor pointwise reevaluation of the spatial HOMS library is required.

The equivalence of the two algorithms follows directly from linearity. Let $\mathcal R_{\eps,h,{\rm sp}}^{(2)}$ denote the recovery operator \eqref{eq:recovery-operator} without the term $G_l(v_{l,h}^n-v_{l,h}^{n-1})/\dt$. Then
\begin{equation}
\mathcal R_{\eps,h,{\rm sp}}^{(2)}
\left(\sum_{k=1}^rz_k^n\bm\phi_k\right)
=\sum_{k=1}^rz_k^n
\mathcal R_{\eps,h,{\rm sp}}^{(2)}\bm\phi_k
=\sum_{k=1}^rz_k^n\bm\Phi_k^\eps.
\label{eq:spatial-recovery-linearity}
\end{equation}
Combining \eqref{eq:spatial-recovery-linearity} with \eqref{eq:derivative-basis-combination} yields
\begin{equation*}
u_{l,r,{\rm A1}}^{\eps,n}=u_{l,r,{\rm A2}}^{\eps,n},
\qquad l=1,2.
\end{equation*}
The equality holds when the same interpolation and discrete differentiation operators are applied, up to floating-point roundoff. Algorithm~2 is therefore not a new approximation model and does not require a separate error theory; it is a different assembly of the same POD-HOMS solution.

\begin{center}
\fbox{\begin{minipage}{0.94\textwidth}
\small
\textbf{Algorithm 2. POD-HOMS with corrected-basis reconstruction}

\textbf{Offline stage}
\begin{enumerate}
\setlength{\itemsep}{0pt}
\item Perform the common cell-problem, discrete-homogenized-system, snapshot, joint-POD, and reduced-matrix stages of Section~\ref{subsec:common-offline}.
\item For $k=1,\ldots,r$, interpolate $\bm\phi_k$ to the fine output mesh and compute its first and second derivatives.
\item Apply all spatial HOMS corrections and store $\bm\Phi_k^\eps$.
\item Construct and store the derivative basis $\bm\Psi_k^\eps$.
\end{enumerate}
\textbf{Online stage: for $n=1,\ldots,N_t$}
\begin{enumerate}
\setlength{\itemsep}{0pt}
\item Solve \eqref{eq:algorithm-common-rom} for $\mathbf z^n$.
\item Form $(\mathbf z^n-\mathbf z^{n-1})/\dt$.
\item Evaluate \eqref{eq:algorithm2} by linear combinations of the stored corrected and derivative bases.
\end{enumerate}
\end{minipage}}
\end{center}

\section{Numerical experiments}
\label{sec:numerical}

Numerical experiments are conducted to assess three aspects of the proposed POD-HOMS framework: the accuracy of the reduced multi-continuum approximation, the computational efficiency of the two recovery strategies, and the scalability of the method to large three-dimensional problems. Five representative periodic heterogeneous configurations are considered. Common numerical settings, mesh information, compared methods, and error definitions are introduced first, followed by separate investigations of reduction accuracy, computational performance, and large-scale feasibility.

All computations were performed with FreeFEM++ on a workstation with a 13th-generation Intel Core i7-13700 processor at 2.10 GHz and 32 GB of memory. The domain is $\Omega=(0,1)^d$ and the final time is $T=1$. The time step is $\dt=0.02$ in two dimensions and $\dt=0.002$ in three dimensions. The finite-element order follows the corresponding implementation. With $P_2$ elements, the second derivatives required by the second-order recovery are evaluated directly. With $P_1$ elements, the first derivatives are first interpolated and the resulting fields are differentiated once more. POD modes are retained until the cumulative energy reaches $99.9999\%$ in all numerical tests, including the large-scale three-dimensional test.

The five heterogeneous geometries comprise two-dimensional circular inclusions, a two-dimensional channel, three-dimensional spherical inclusions, a three-dimensional channel, and a two-dimensional three-ellipse configuration (Figure~\ref{fig:geometries}).

\begin{figure}[htbp]
\centering
\begin{tabular}{@{}cc@{}}
\includegraphics[width=0.47\textwidth]{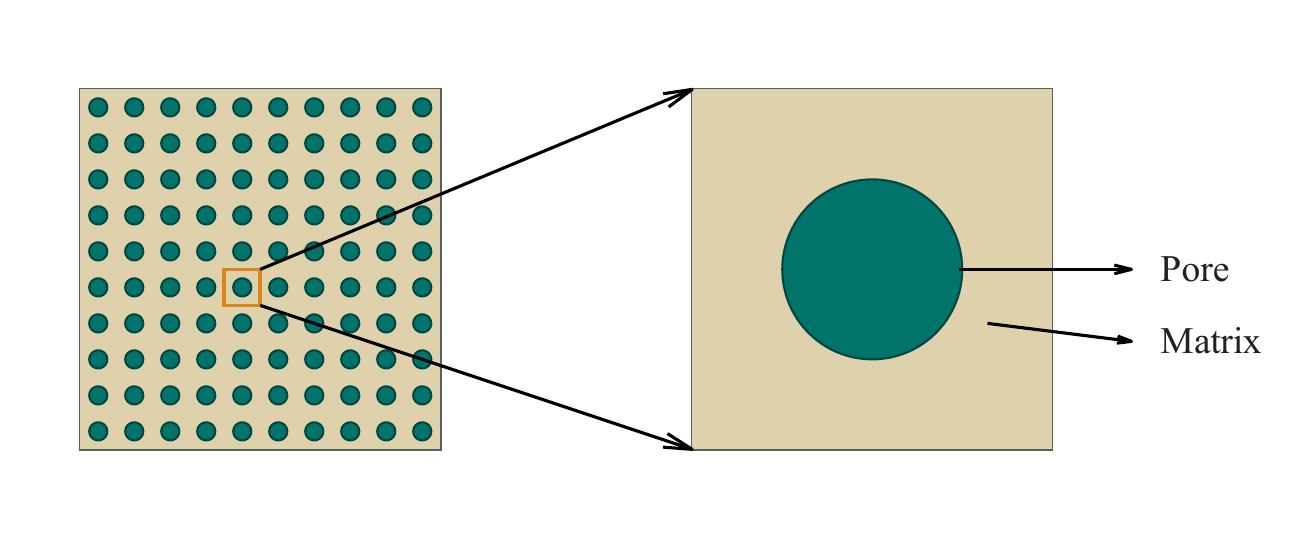} &
\includegraphics[width=0.47\textwidth]{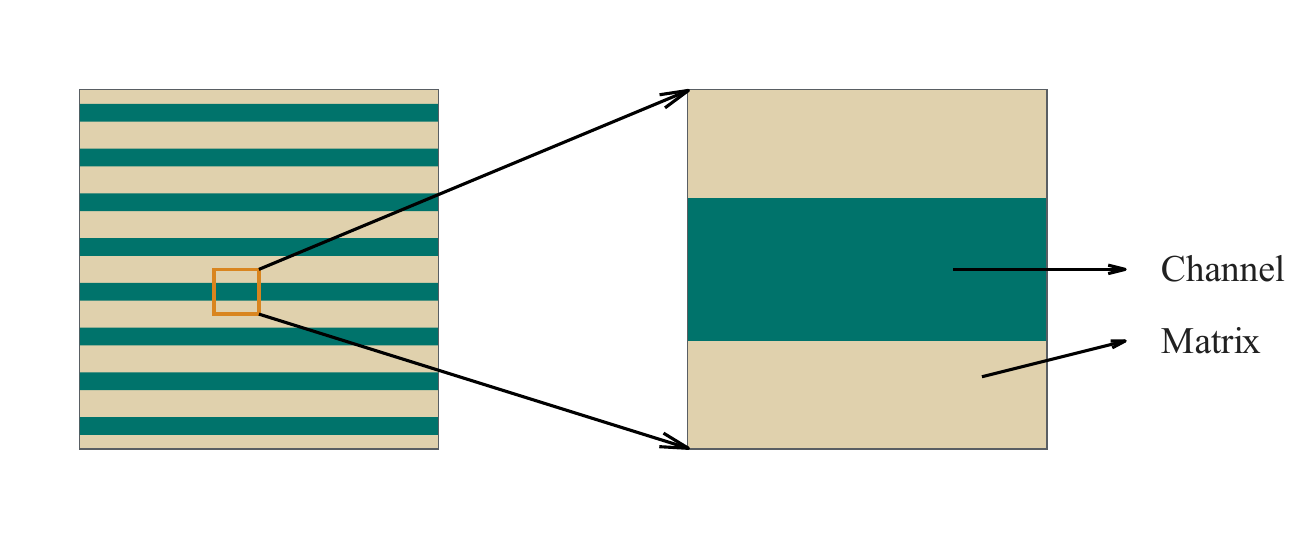}\\
(a) 2D circular inclusions & (b) 2D channel\\[2mm]
\includegraphics[width=0.47\textwidth]{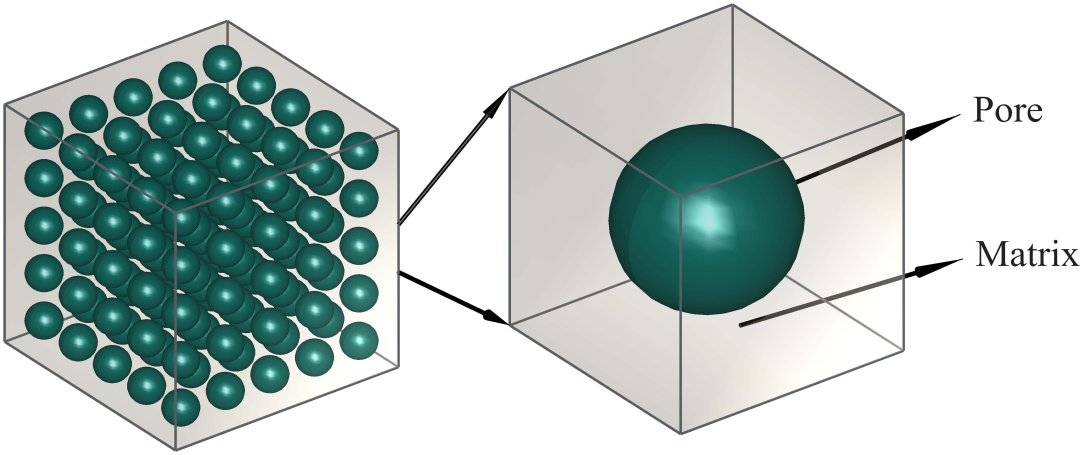} &
\includegraphics[width=0.47\textwidth]{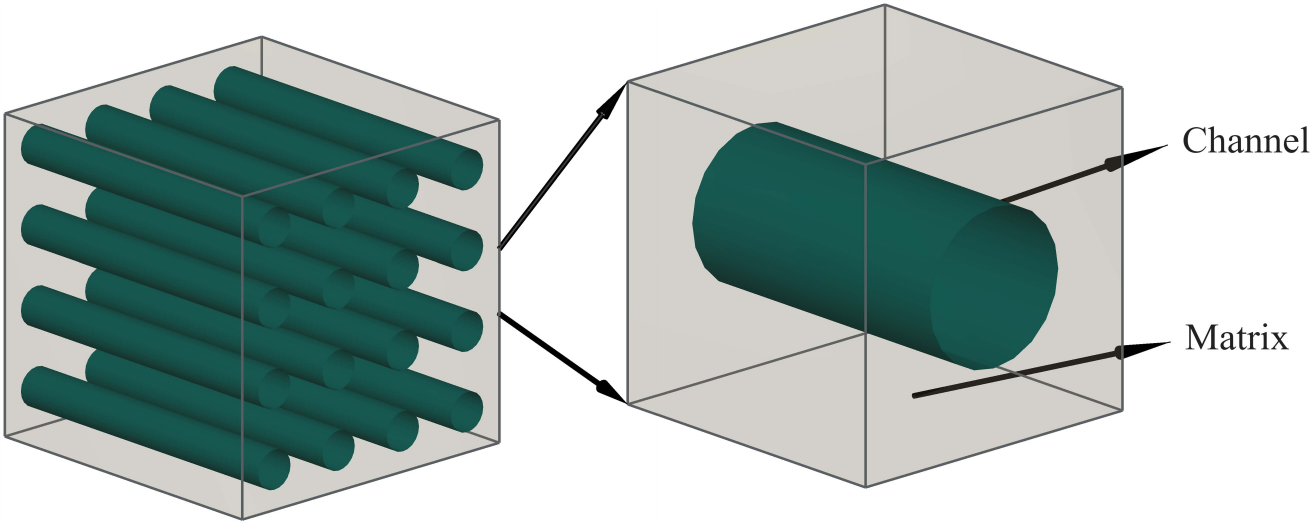}\\
(c) 3D spherical inclusions & (d) 3D channel\\[2mm]
\multicolumn{2}{c}{\includegraphics[width=0.47\textwidth]{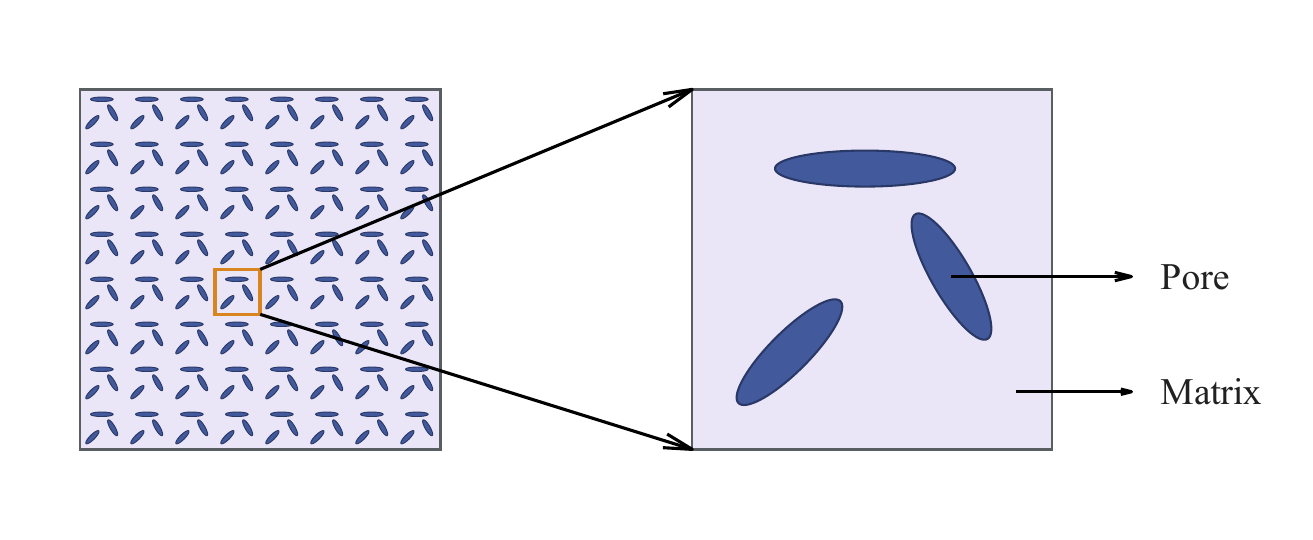}}\\
\multicolumn{2}{c}{(e) 2D three-ellipse inclusions}
\end{tabular}
\caption{Periodic heterogeneous configurations and representative cells used in the standard multi-continuum tests.}
\label{fig:geometries}
\end{figure}

The three-ellipse cell in Figure~\ref{fig:geometries}(e) contains elongated inclusions with centers $(0.27,0.27)$, $(0.72,0.48)$, and $(0.48,0.78)$, semi-axis pairs $(0.20,0.055)$, $(0.20,0.055)$, and $(0.25,0.050)$, and orientations $\pi/4$, $-\pi/3$, and $0$, respectively. An $8\times8$ periodic replication gives $\eps=1/8$. The low-permeability inclusions are embedded in a conducting matrix, with permeability contrasts $10^4$ and $10^2$ for continua~1 and~2. Their different orientations introduce direction-dependent microscopic structures beyond the circular and channel geometries. The initial pressures and exterior Dirichlet data are zero, and the spatially uniform, continuum-dependent sources are
\begin{equation*}
q_1(t)=10^5\bigl(1+0.5\sin(2\pi t)\bigr),
\qquad
q_2(t)=10^5\bigl(1+0.3\cos(2\pi t)\bigr).
\end{equation*}
These distinct temporal forcings test the joint POD approximation when the two continuum trajectories have different amplitudes and time dependence. The implementation uses $P_1$ elements for the fine-scale and cell problems and $P_2$ elements for the homogenized problem.

The phase coefficients, scale ratios, and sources are collected in Table~\ref{tab:parameters}. The separate $Q_1$ and $Q_2$ columns give the exchange coefficients used in the first and second continuum equations, respectively, without requiring them to coincide. Presenting these data together avoids repeating the physical setup in the accuracy and efficiency studies. The three-dimensional channel uses the same phase contrast as the spherical case but a different periodic geometry and scale ratio. In the three-ellipse case, phase~1 denotes the matrix and phase~2 the inclusions; its source entries are given by the source definitions above.

\begin{table}[htbp]
\centering
\caption{Phase parameters, scale ratios, and source terms for the standard tests.}
\label{tab:parameters}
\normalsize
\begin{tabular*}{\textwidth}{@{\extracolsep{\fill}}llrrrrrrrr@{}}
\toprule
Case & Phase & $\kappa_1$ & $\kappa_2$ & $Q_1$ & $Q_2$ & $c_1$ & $c_2$ & $\eps$ & Source\\
\midrule
2D circle  & 1 & 100 & 50   & 20    & 30   & 5   & 4.5 & $1/10$ & $10^5$\\
           & 2 & 1   & 1    & 0.25  & 0.50 & 2   & 1.5 & $1/10$ & $10^5$\\
2D channel & 1 & 50  & 100  & 20    & 30   & 30  & 50  & $1/8$  & $10^4$\\
           & 2 & 5   & 4    & 0.25  & 0.25 & 10  & 25  & $1/8$  & $10^4$\\
2D ellipses & 1 & $10^4$ & 100 & 20 & 20 & 10 & 5 & $1/8$ & $q_l(t)$\\
            & 2 & 1 & 1 & 2.5 & 2.5 & 1 & 1 & $1/8$ & $q_l(t)$\\
3D sphere  & 1 & 800 & 1000 & 250   & 300  & 5   & 125 & $1/5$  & $10^5$\\
           & 2 & 30  & 50   & 25    & 30   & 100 & 15  & $1/5$  & $10^5$\\
3D channel & 1 & 800 & 1000 & 250   & 300  & 5   & 125 & $1/4$  & $10^5$\\
           & 2 & 30  & 50   & 25    & 30   & 100 & 15  & $1/4$  & $10^5$\\
\bottomrule
\end{tabular*}
\end{table}

\begin{table}[htbp]
\centering
\caption{Mesh statistics for the standard multi-continuum tests.}
\label{tab:meshes}
\normalsize
\begin{tabular*}{\textwidth}{@{\extracolsep{\fill}}llrrrrrr@{}}
\toprule
Dimension & Geometry & \multicolumn{2}{c}{Fine-scale mesh} & \multicolumn{2}{c}{Cell mesh} & \multicolumn{2}{c}{Homogenized mesh}\\
\cmidrule(lr){3-4}\cmidrule(lr){5-6}\cmidrule(lr){7-8}
& & Elements & Nodes & Elements & Nodes & Elements & Nodes\\
\midrule
2D & circle  & 89800  & 45301  & 878    & 480   & 8590  & 4416\\
2D & channel & 56832  & 28737  & 1946   & 1034  & 15080 & 7701\\
2D & ellipses & 94848 & 47745 & 6018 & 3090 & 8570 & 4406\\
3D & sphere  & 584928 & 97411  & 145815 & 24741 & 93750 & 17576\\
3D & channel & 941440 & 164945 & 84120  & 15582 & 93750 & 17576\\
\bottomrule
\end{tabular*}
\end{table}

Four computational routes are distinguished throughout this section. FEM directly solves the fine-scale coupled system and serves as the reference whenever its memory requirement is admissible. FEM-POD reduces that fine-scale system but still requires fine-scale snapshot generation. HOMS advances the discrete homogenized system and then applies the second-order recovery. HOMS-POD combines the joint vector POD approximation with either Algorithm~1, which repeats HOMS recovery at every output time, or Algorithm~2, which forms corrected spatial and time-derivative bases offline and uses their linear combinations online.

\begin{samepage}
When the fine-scale FEM reference is available, the accuracy of each comparison method is measured componentwise at the discrete times $t_n$. For $l=1,2$, the symbols $E_{L^2,l}^{\mathrm{FP}}(t_n)$, $E_{L^2,l}^{1}(t_n)$, $E_{L^2,l}^{2}(t_n)$, and $E_{L^2,l}^{\mathrm{HP}}(t_n)$ denote the relative $L^2$ errors of the FEM-POD, first-order multiscale (FOMS), second-order HOMS, and HOMS-POD solutions, respectively. Their relative $H^1$-seminorm errors are $E_{H^1,l}^{\mathrm{FP}}(t_n)$, $E_{H^1,l}^{1}(t_n)$, $E_{H^1,l}^{2}(t_n)$, and $E_{H^1,l}^{\mathrm{HP}}(t_n)$. In these symbols, the subscript $H^1$ always refers to the seminorm, not the full norm. To define both measures, let $s\in\{\mathrm{FP},1,2,\mathrm{HP}\}$ and set $u_l^{[s]}(t_n)$ to the corresponding comparison field: $u_l^{[1]}(t_n)=u_l^{(1\eps)}(\cdot,t_n)$, $u_l^{[2]}(t_n)=u_l^{(2\eps)}(\cdot,t_n)$, and $u_l^{[\mathrm{HP}]}(t_n)=u_{l,r}^{\eps,n}$, the latter obtained by either of the equivalent Algorithms~1 and~2. The field $u_l^{[\mathrm{FP}]}(t_n)$ is the fine-scale FEM-POD solution. Then
\begin{align}
E_{L^2,l}^{s}(t_n)
&=\frac{\norm{u_l^{[s]}(t_n)-u_{l,{\rm FEM}}(t_n)}_{L^2(\Omega)}}
{\norm{u_{l,{\rm FEM}}(t_n)}_{L^2(\Omega)}},
\label{eq:l2-relative-error}\\
E_{H^1,l}^{s}(t_n)
&=\frac{\left|u_l^{[s]}(t_n)-u_{l,{\rm FEM}}(t_n)\right|_{H^1(\Omega)}}
{\left|u_{l,{\rm FEM}}(t_n)\right|_{H^1(\Omega)}}.
\label{eq:h1-semi-relative-error}
\end{align}
\end{samepage}
Thus, both time-dependent measures compare each method with the same fine-scale FEM reference. Equation~\eqref{eq:l2-relative-error} measures the relative field error in $L^2(\Omega)$, whereas \eqref{eq:h1-semi-relative-error} measures only the relative gradient error through the $H^1(\Omega)$ seminorm. Since $T=t_{N_t}$, the tables report the final-time values of these errors.

\subsection{POD reduction and accuracy}
\label{subsec:pod-accuracy}

This subsection examines snapshot sensitivity, mode selection, FEM-referenced accuracy, and representative recovered fields within a single numerical study. The retained rank is selected by the cumulative-energy criterion \eqref{eq:pod-energy}, whereas the number of snapshots controls the amount of training data and the offline cost.

Table~\ref{tab:snapshot} compares small and large training ensembles for Algorithm~1. The reported relative $L^2$ error is $E_{L^2}^{\max}(T):=\max_{l=1,2}E_{L^2,l}^{\mathrm{HP}}(T)$, and the reported relative $H^1$-seminorm error is $E_{H^1}^{\max}(T):=\max_{l=1,2}E_{H^1,l}^{\mathrm{HP}}(T)$. Both are the larger of the two continuum-component errors at the final time. The CPU times are the recorded end-to-end Algorithm~1 totals, including snapshot generation and recovery, for the corresponding snapshot counts. The small sets give errors comparable to the larger sets and reduce snapshot-generation cost. Therefore, five snapshots are used for the two-dimensional circular and channel tests, eight for the three-ellipse test, and ten in three dimensions. The retained rank is still determined independently by the cumulative-energy criterion \eqref{eq:pod-energy}.

\begin{table}[htbp]
\centering
\caption{Snapshot sensitivity of Algorithm~1.}
\label{tab:snapshot}
\resizebox{\textwidth}{!}{%
\begin{tabular}{llccccc}
\toprule
Dimension & Geometry & $N_s$ (low/high) & $E_{L^2}^{\max}(T)$ & $E_{H^1}^{\max}(T)$ & CPU time (s) & Selected $N_s$\\
\midrule
2D & circle  & $5/15$  & $0.08900/0.08625$ & $0.0928/0.0909$ & $133.069/184.575$ & 5\\
2D & channel & $5/15$  & $0.10405/0.10479$ & $0.1075/0.1051$ & $85.561/157.027$ & 5\\
2D & ellipses & $8/15$ & $0.104503/0.104504$ & $0.113408/0.113406$ & $148.216/193.345$ & 8\\
3D & sphere  & $10/50$ & $0.02357/0.03063$ & $0.1371/0.1354$ & $2599.200/2803.470$ & 10\\
3D & channel & $10/25$ & $0.02414/0.02478$ & $0.1664/0.1695$ & $7312.000/7313.900$ & 10\\
\bottomrule
\end{tabular}}
\end{table}

Increasing the snapshot count does not produce a systematic decrease in the FEM-referenced error: the changes are small and, in the three-dimensional spherical case, the larger ensemble even gives a slightly larger final $L^2$ error. For the three-ellipse test, increasing $N_s$ from eight to fifteen changes the maximum final $L^2$ error from $0.104503$ to $0.104504$ and the maximum $H^1$-seminorm error from $0.113408$ to $0.113406$, while increasing the Algorithm~1 time from $148.216$ to $193.345$ seconds. Over the tested range, the observed FEM-relative error is therefore not controlled primarily by snapshot insufficiency. The smaller ensembles are retained because they reduce training cost while preserving comparable accuracy; in the three-dimensional channel case, the end-to-end difference is small because output recovery dominates the runtime. The archived sensitivity summaries do not contain the selected ranks or eigenvalue sequences, so no rank or spectrum is reconstructed from these aggregate values.

Table~\ref{tab:errors} gives the final-time FEM-referenced relative errors for corrected-basis reconstruction, denoted by $E_{L^2,l}^{\mathrm{HP}}(T)$ and $E_{H^1,l}^{\mathrm{HP}}(T)$. These values assess the complete chain from homogenization through POD evolution to high-order recovery using the definitions in \eqref{eq:l2-relative-error} and \eqref{eq:h1-semi-relative-error}.

\begin{table}[htbp]
\centering
\caption{FEM-referenced relative errors at $T=1$ for Algorithm~2.}
\label{tab:errors}
\normalsize
\begin{tabular*}{\textwidth}{@{\extracolsep{\fill}}lllcccc@{}}
\toprule
Dimension & Geometry & $N_s$ & $E_{L^2,1}^{\mathrm{HP}}(T)$ & $E_{L^2,2}^{\mathrm{HP}}(T)$ & $E_{H^1,1}^{\mathrm{HP}}(T)$ & $E_{H^1,2}^{\mathrm{HP}}(T)$\\
\midrule
2D & circle  & 5  & $7.592{\times}10^{-2}$ & $8.900{\times}10^{-2}$ & $8.372{\times}10^{-2}$ & $9.277{\times}10^{-2}$\\
2D & channel & 5  & $1.040{\times}10^{-1}$ & $8.294{\times}10^{-2}$ & $1.074{\times}10^{-1}$ & $8.722{\times}10^{-2}$\\
2D & ellipses & 8 & $8.663{\times}10^{-2}$ & $1.045{\times}10^{-1}$ & $8.255{\times}10^{-2}$ & $1.134{\times}10^{-1}$\\
3D & sphere  & 10 & $3.324{\times}10^{-2}$ & $1.345{\times}10^{-2}$ & $1.366{\times}10^{-1}$ & $1.197{\times}10^{-1}$\\
3D & channel & 10 & $2.656{\times}10^{-2}$ & $2.131{\times}10^{-2}$ & $1.683{\times}10^{-1}$ & $1.495{\times}10^{-1}$\\
\bottomrule
\end{tabular*}
\end{table}

The final-time $L^2$ relative errors of Algorithm~2 remain below approximately $10.5\%$ in all five standard tests. The relative $H^1$-seminorm errors also assess the gradients, which are sensitive to phase interfaces and recovered microscopic oscillations. Figures~\ref{fig:error-circle2d}--\ref{fig:error-channel3d} and Figure~\ref{fig:error-complex2d} show the time histories of $E_{L^2,l}^{s}(t_n)$ and $E_{H^1,l}^{s}(t_n)$ for all four comparison methods and both continuum components. In each figure, the upper row concerns continuum~1, the lower row continuum~2, and the left and right panels display the $L^2$-norm and $H^1$-seminorm errors, respectively. In the plot legends, $t$ denotes the sampled time $t_n$. Every curve is evaluated against the same fine-scale FEM reference using \eqref{eq:l2-relative-error} or \eqref{eq:h1-semi-relative-error}. The FOMS curves indicate how the second-order correctors affect each error measure. The second-order HOMS and HOMS-POD curves remain close for both continuum components.

The three-ellipse histories in Figure~\ref{fig:error-complex2d} exhibit a time-dependent response to the prescribed continuum-dependent sources. For continuum~1, second-order HOMS reduces the final $L^2$ error from $42.932\%$ for FOMS to $8.665\%$ and the $H^1$-seminorm error from $98.638\%$ to $8.253\%$. For continuum~2, the corresponding $H^1$-seminorm error decreases from $29.759\%$ to $11.340\%$, whereas the first- and second-order $L^2$ errors are comparable ($10.341\%$ and $10.450\%$). Thus, the benefit of the second-order terms depends on the continuum and error measure. The final HOMS-POD errors are $8.663\%$ and $10.450\%$ in $L^2$, and $8.255\%$ and $11.341\%$ in the $H^1$ seminorm. The near-coincidence of the HOMS and HOMS-POD curves shows that the reduced approximation preserves the second-order recovered accuracy even for the high-contrast, differently oriented inclusions and time-dependent forcing. The saved Algorithm~1 and~2 error histories agree to the precision of the exported data.

\begin{figure}[htbp]
\centering
\includegraphics[width=0.75\textwidth]{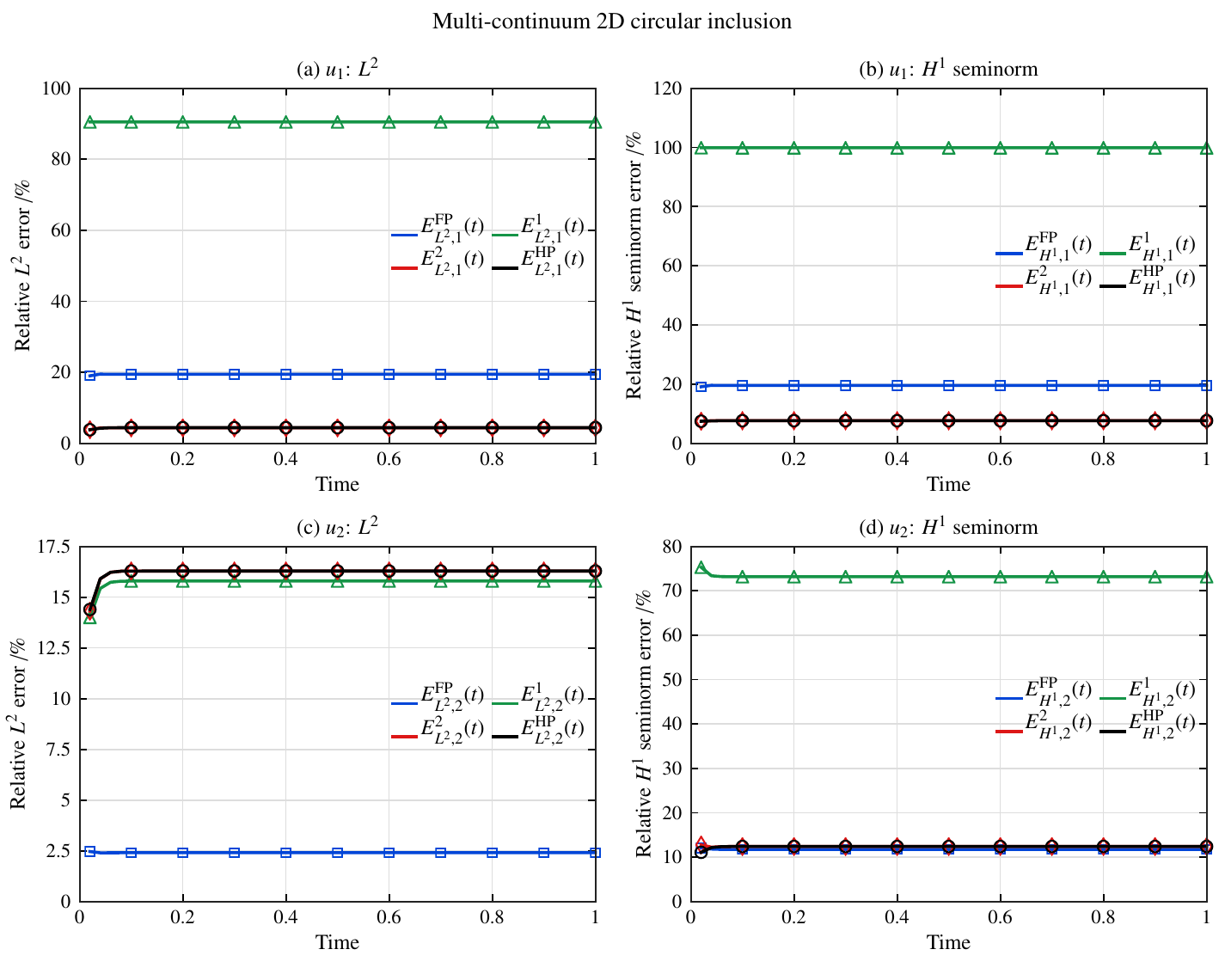}
\caption{Relative-error evolution for the two-dimensional circular-inclusion test: (a) $E_{L^2,1}^{s}(t)$; (b) $E_{H^1,1}^{s}(t)$; (c) $E_{L^2,2}^{s}(t)$; (d) $E_{H^1,2}^{s}(t)$.}
\label{fig:error-circle2d}
\end{figure}

\begin{figure}[htbp]
\centering
\includegraphics[width=0.75\textwidth]{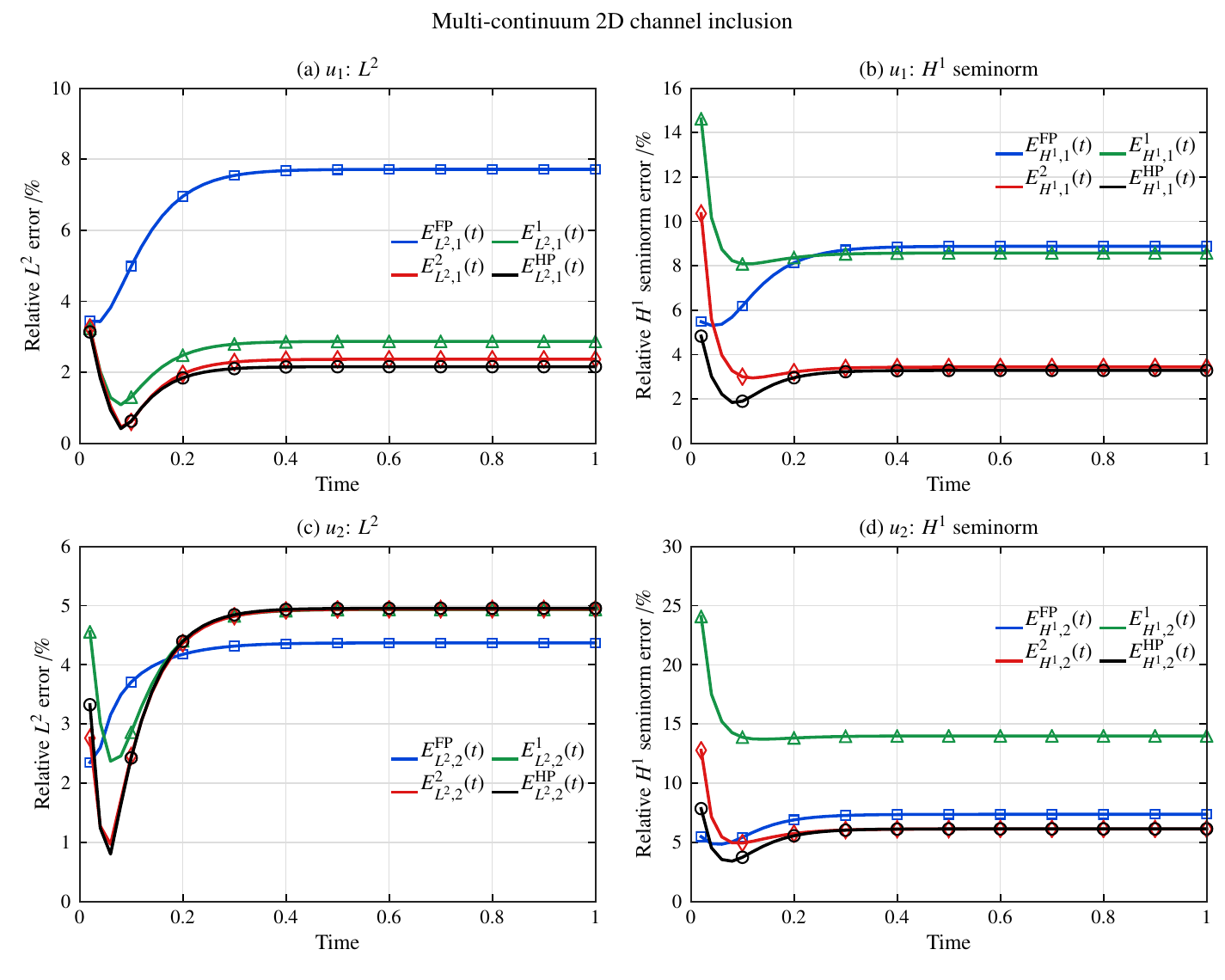}
\caption{Relative-error evolution for the two-dimensional channel-inclusion test: (a) $E_{L^2,1}^{s}(t)$; (b) $E_{H^1,1}^{s}(t)$; (c) $E_{L^2,2}^{s}(t)$; (d) $E_{H^1,2}^{s}(t)$.}
\label{fig:error-channel2d}
\end{figure}

\begin{figure}[htbp]
\centering
\includegraphics[width=0.75\textwidth]{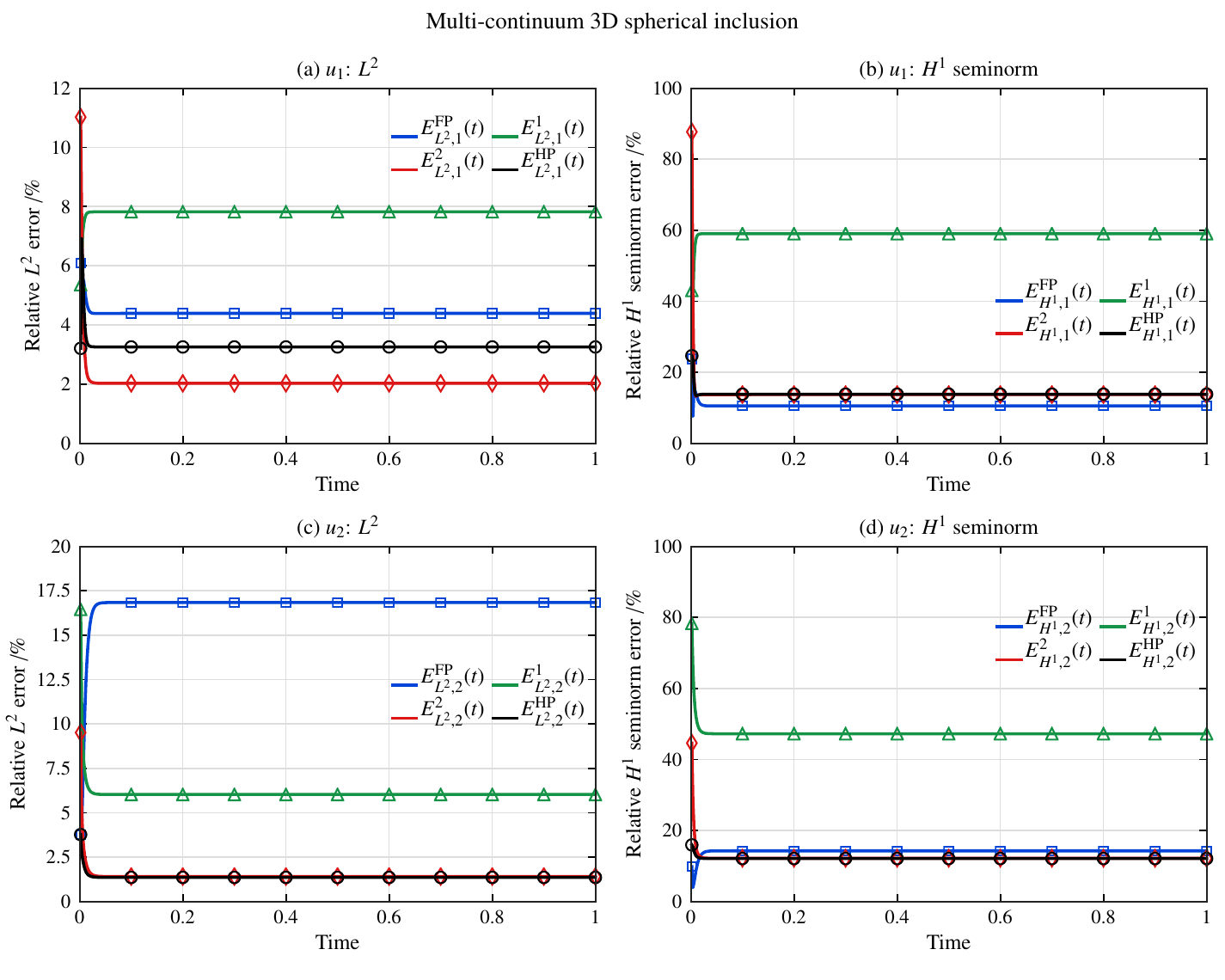}
\caption{Relative-error evolution for the three-dimensional spherical-inclusion test: (a) $E_{L^2,1}^{s}(t)$; (b) $E_{H^1,1}^{s}(t)$; (c) $E_{L^2,2}^{s}(t)$; (d) $E_{H^1,2}^{s}(t)$.}
\label{fig:error-sphere3d}
\end{figure}

\begin{figure}[htbp]
\centering
\includegraphics[width=0.75\textwidth]{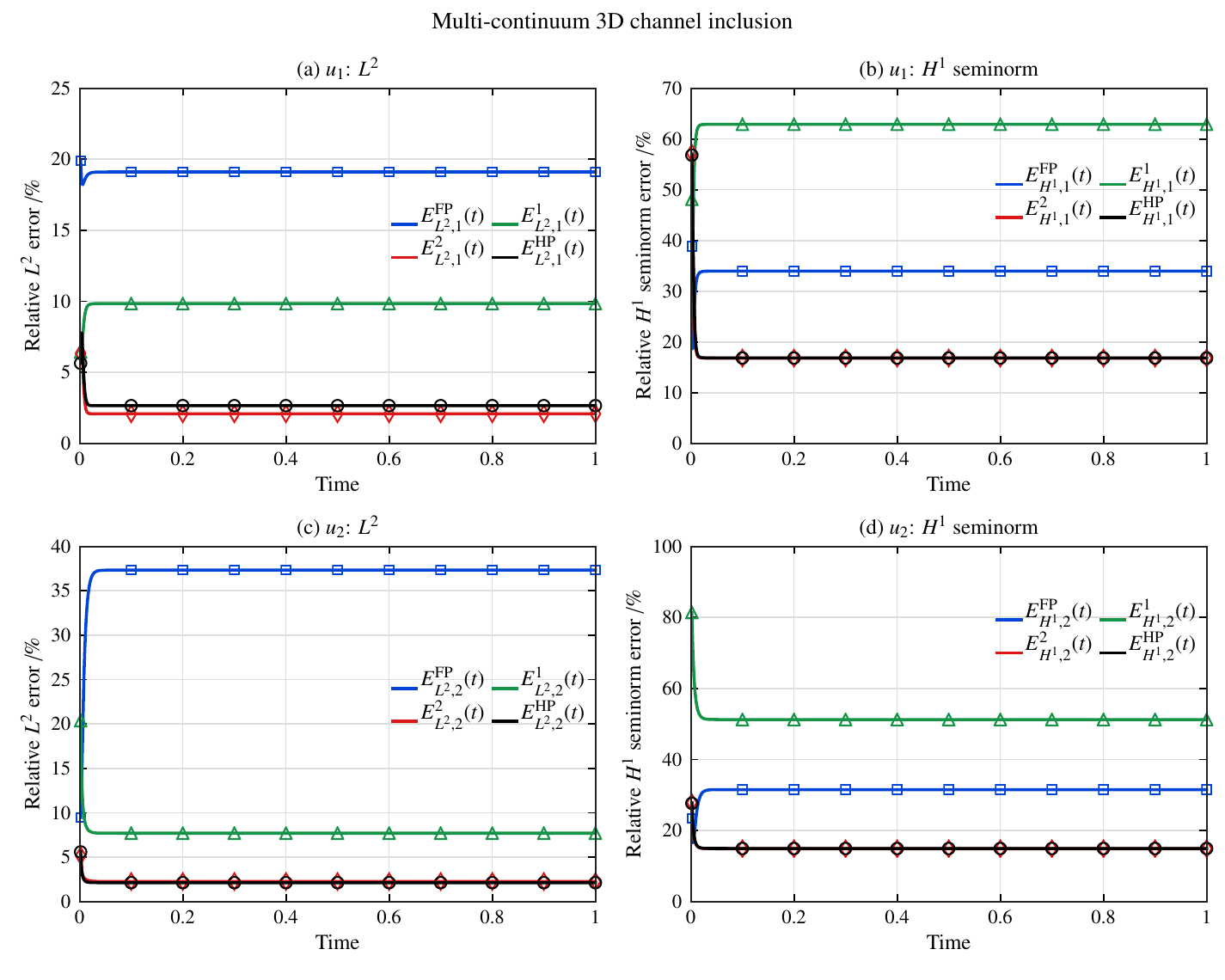}
\caption{Relative-error evolution for the three-dimensional channel-inclusion test: (a) $E_{L^2,1}^{s}(t)$; (b) $E_{H^1,1}^{s}(t)$; (c) $E_{L^2,2}^{s}(t)$; (d) $E_{H^1,2}^{s}(t)$.}
\label{fig:error-channel3d}
\end{figure}

\begin{figure}[htbp]
\centering
\includegraphics[width=0.75\textwidth]{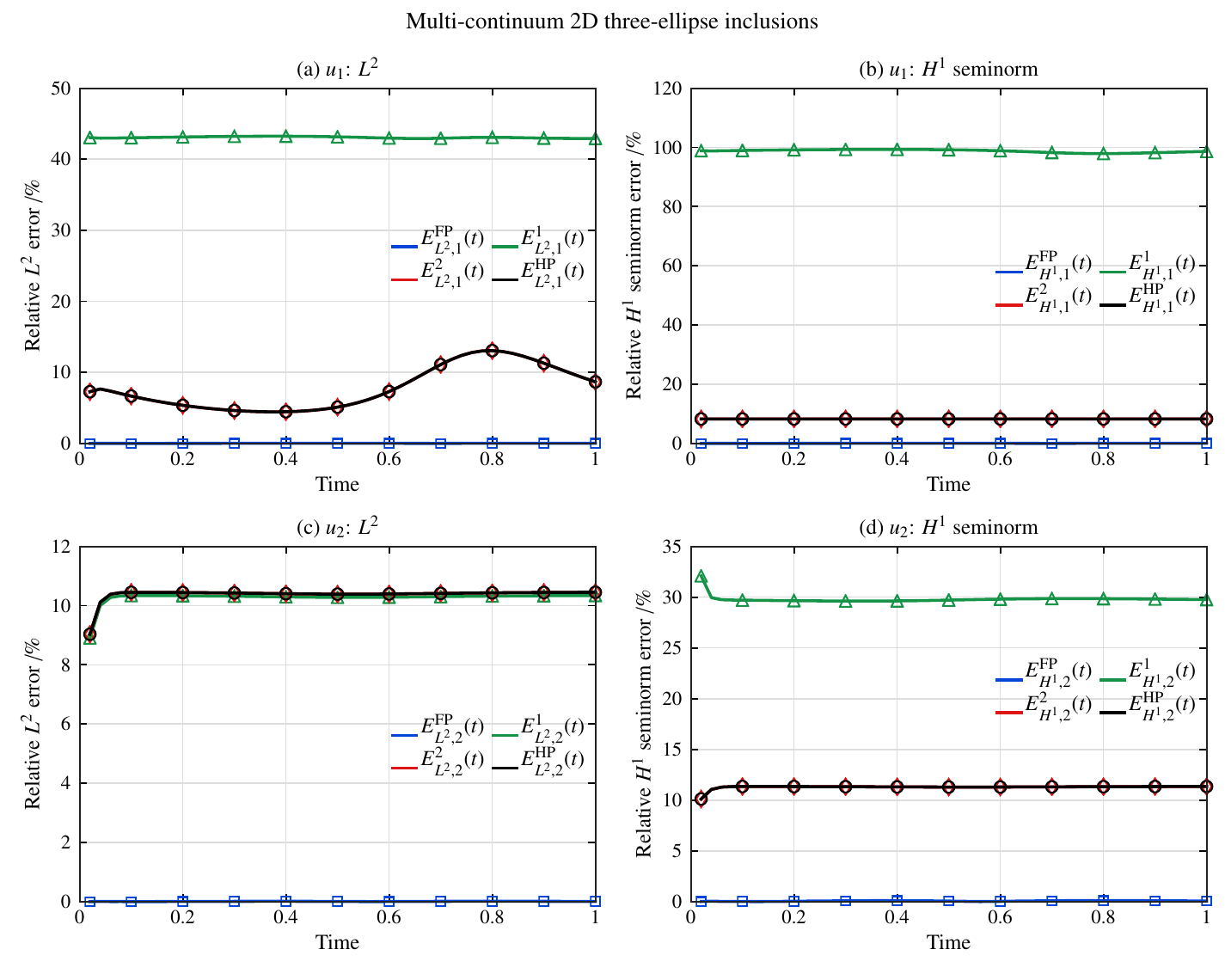}
\caption{Relative-error evolution for the two-dimensional three-ellipse test: (a) $E_{L^2,1}^{s}(t)$; (b) $E_{H^1,1}^{s}(t)$; (c) $E_{L^2,2}^{s}(t)$; (d) $E_{H^1,2}^{s}(t)$.}
\label{fig:error-complex2d}
\end{figure}

Figures~\ref{fig:circle2d-u1-fields}--\ref{fig:channel3d-u2-fields}, together with Figures~\ref{fig:complex2d-u1-fields} and~\ref{fig:complex2d-u2-fields}, compare FEM, FEM-POD, FOMS, second-order HOMS, and POD-HOMS for both continuum fields in the five configurations. Each figure arranges the five solutions in a single row and is devoted to one continuum, following the presentation used for the two continuum fields in the previous HOMS study. In these field plots, $u_l^\eps$ is approximated by the fine-scale finite-element reference solution $u_{l,{\rm FEM}}(T)$; the caption notation follows the original fine-scale system \eqref{eq:fine-system} and does not identify the numerical reference with the exact solution used in the analysis. Here $u_l^{[\mathrm{FP}]}(T)$ denotes the fine-scale FEM-POD field, $u_l^{(1\eps)}$ the first-order truncation of \eqref{eq:homs-representation}, $u_l^{(2\eps)}$ the full second-order reconstruction, and $u_{l,r}^{\eps,n}$ the recovered POD-HOMS field at $t_n=T$. The purpose is not to revalidate HOMS itself, but to check that introducing POD preserves the macroscopic envelope and the local oscillatory structures already recovered by HOMS.

\begin{figure}[htbp]
\centering
\begin{tabular}{@{}ccccc@{}}
\includegraphics[width=0.17\textwidth]{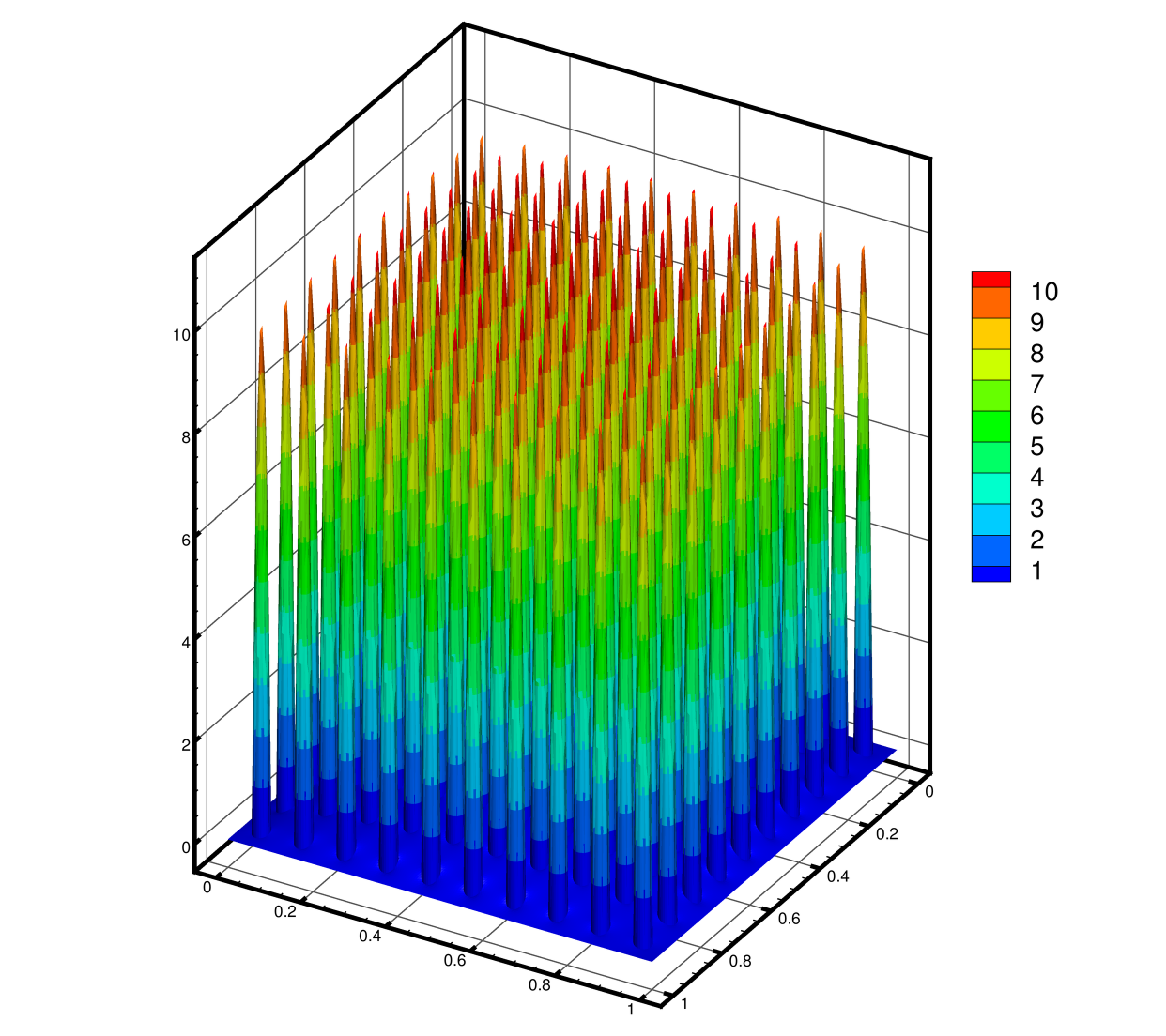} &
\includegraphics[width=0.17\textwidth]{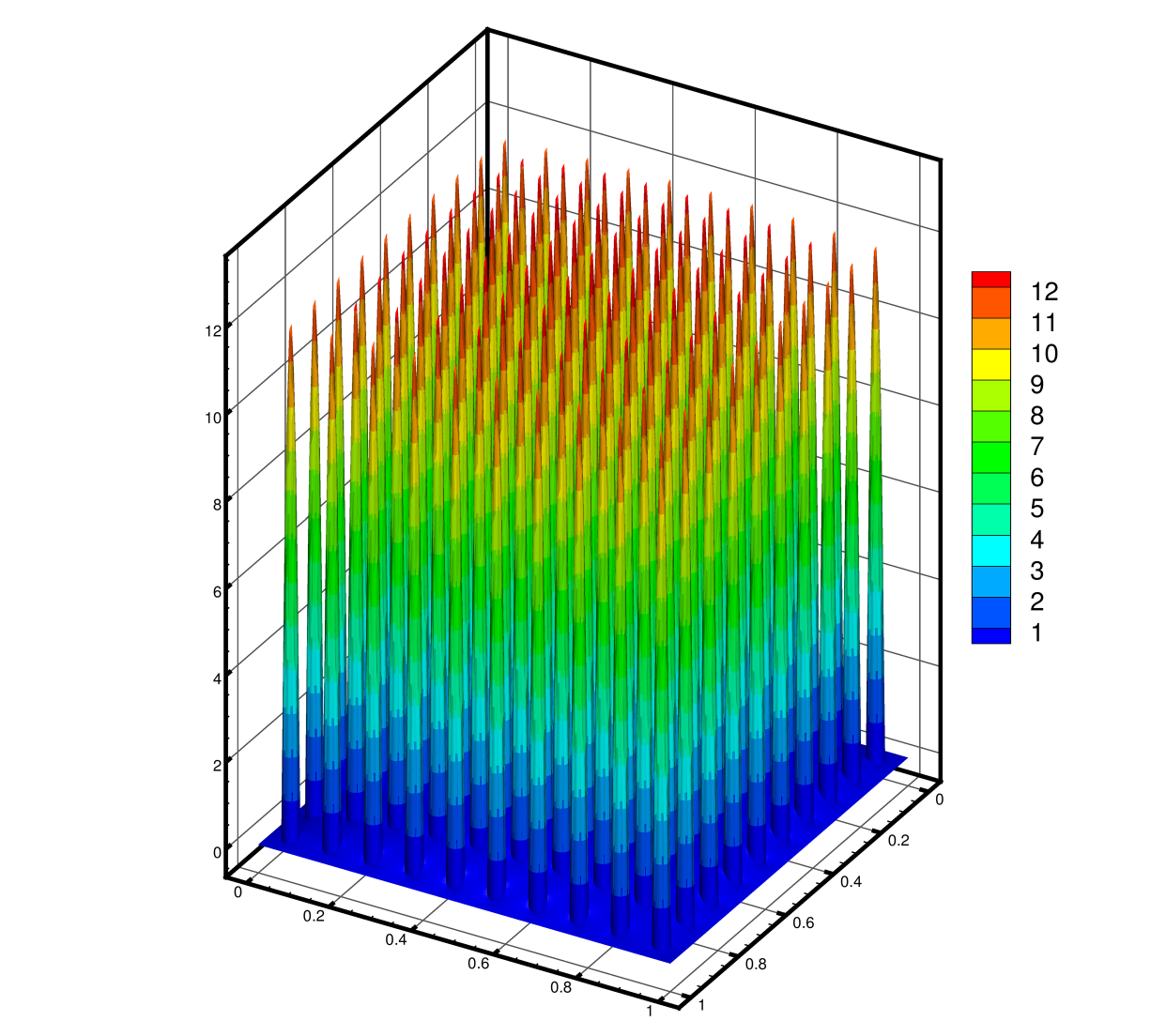} &
\includegraphics[width=0.17\textwidth]{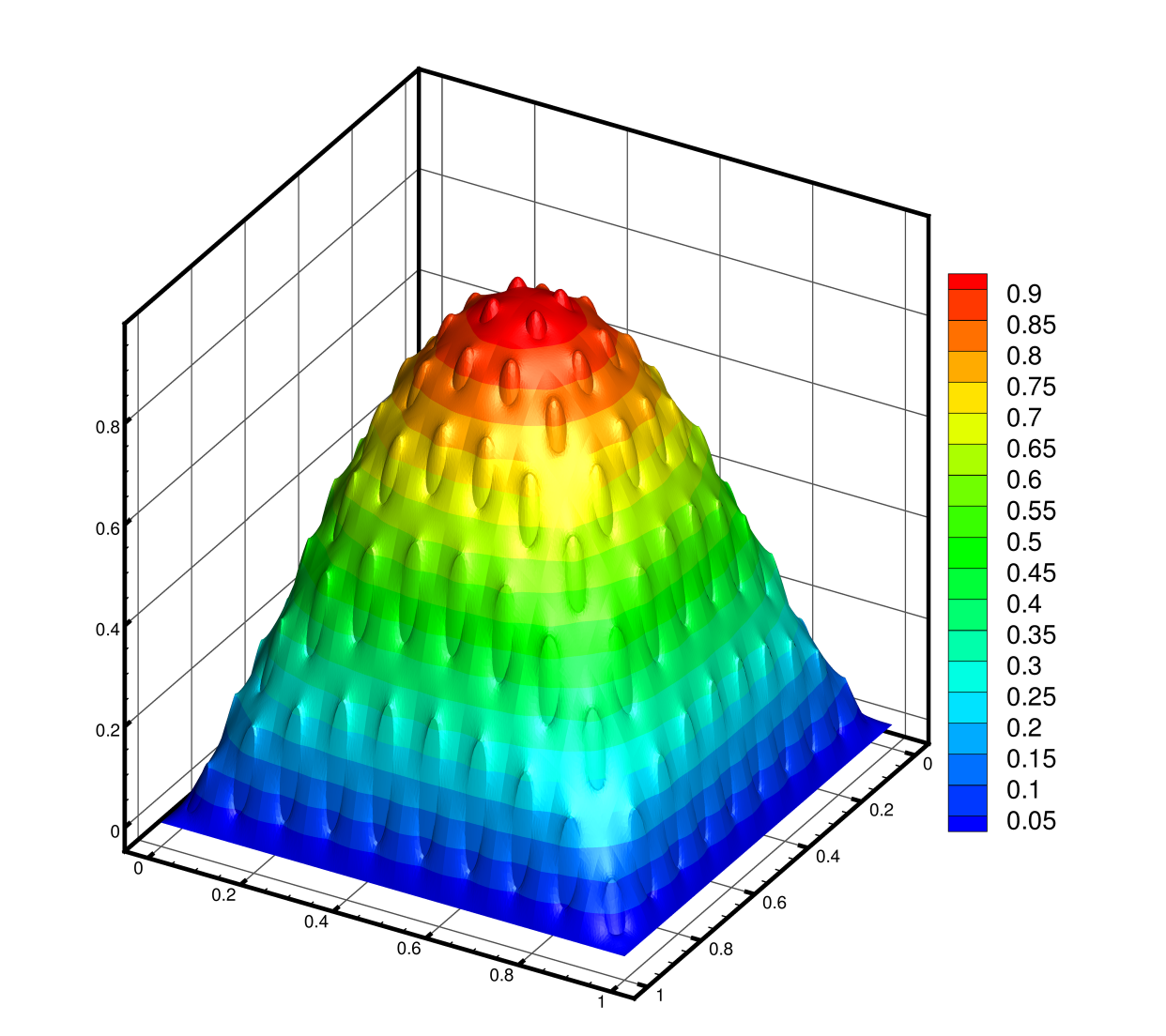} &
\includegraphics[width=0.17\textwidth]{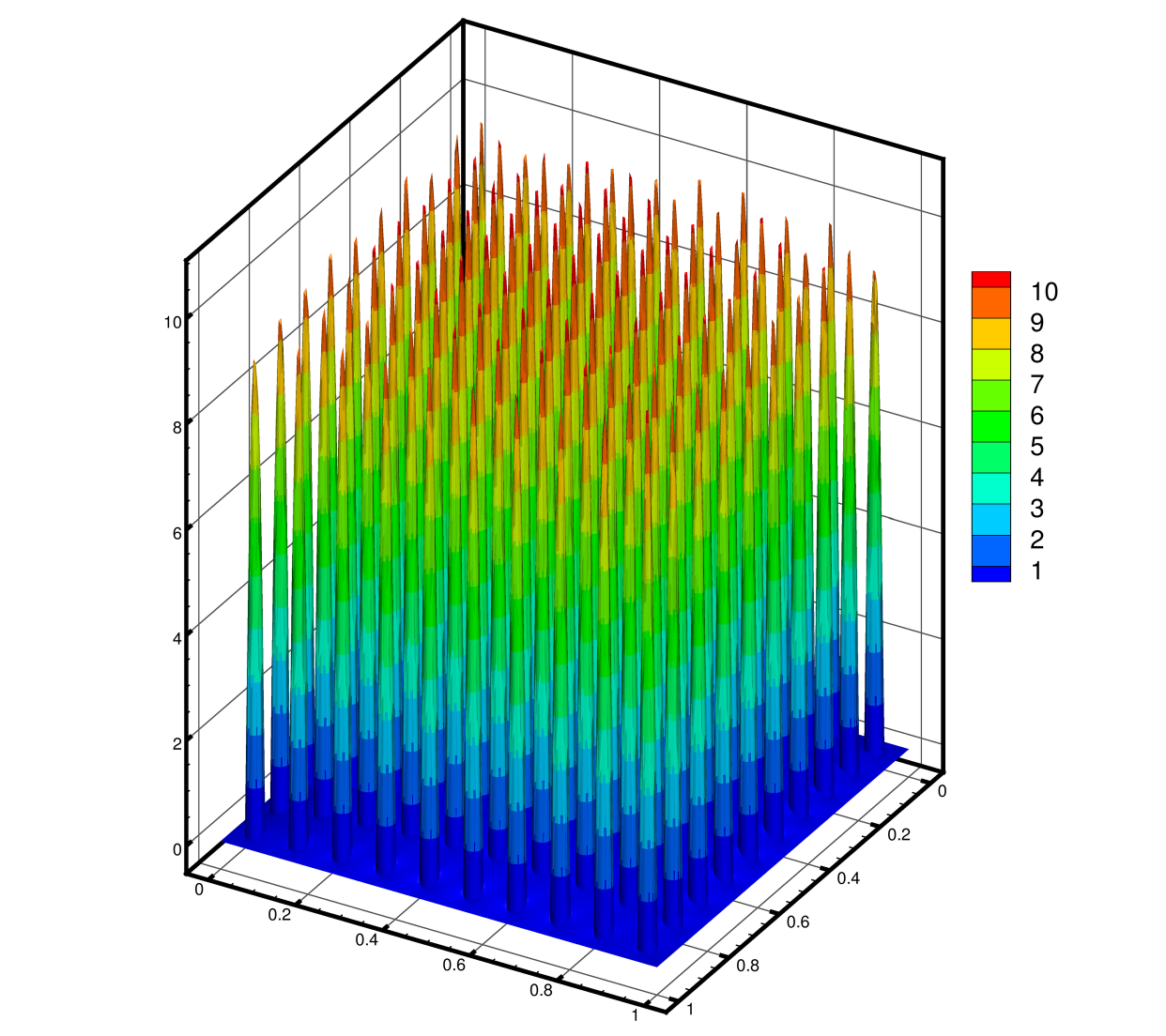} &
\includegraphics[width=0.17\textwidth]{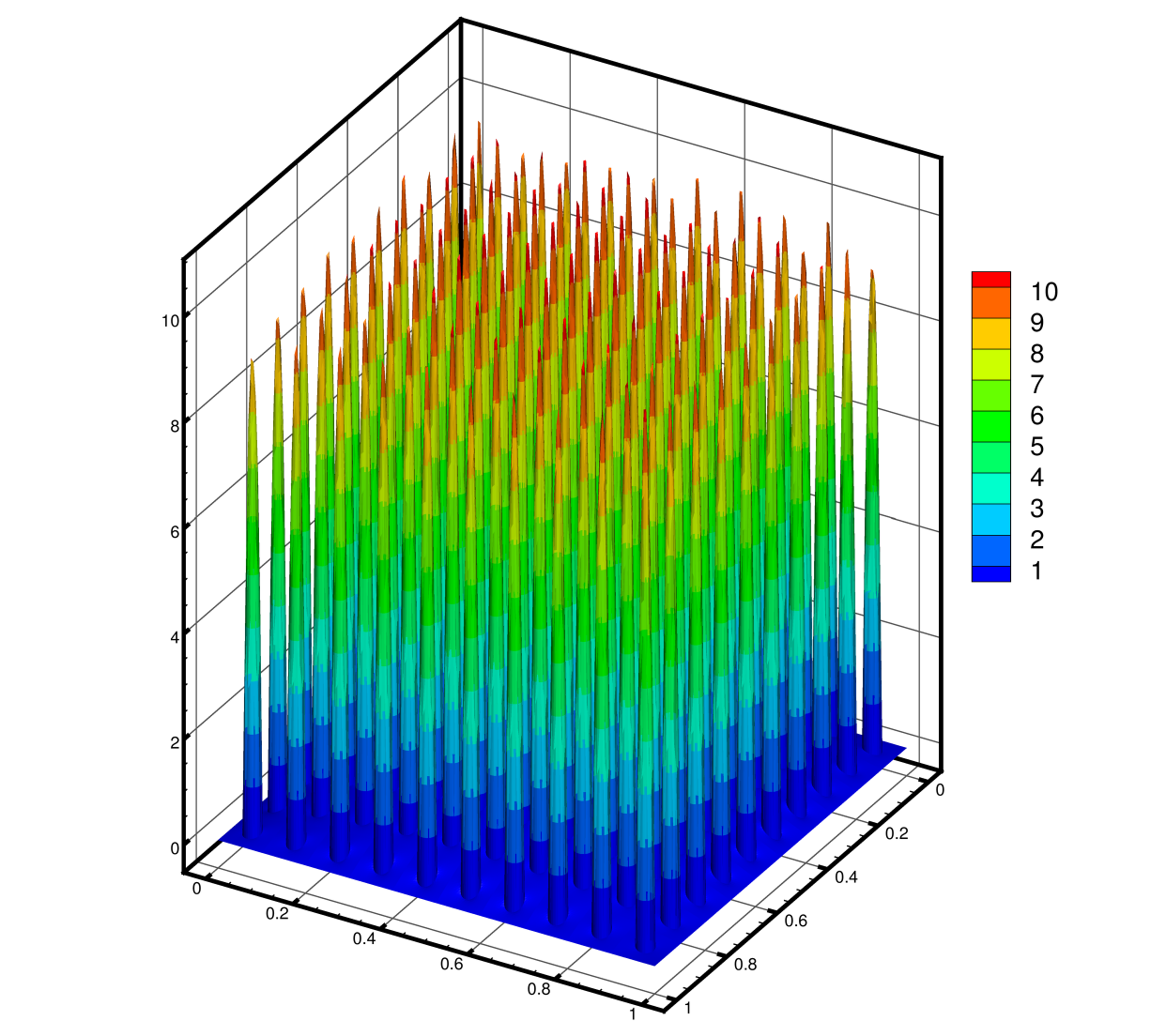}\\
(a) & (b) & (c) & (d) & (e)
\end{tabular}
\caption{Continuum 1 pressure fields for the two-dimensional circular-inclusion problem at $T=1$: (a) $u_1^\eps$; (b) $u_1^{[\mathrm{FP}]}(T)$; (c) $u_1^{(1\eps)}$; (d) $u_1^{(2\eps)}$; (e) $u_{1,r}^{\eps,n}$.}
\label{fig:circle2d-u1-fields}
\end{figure}

\begin{figure}[htbp]
\centering
\begin{tabular}{@{}ccccc@{}}
\includegraphics[width=0.17\textwidth]{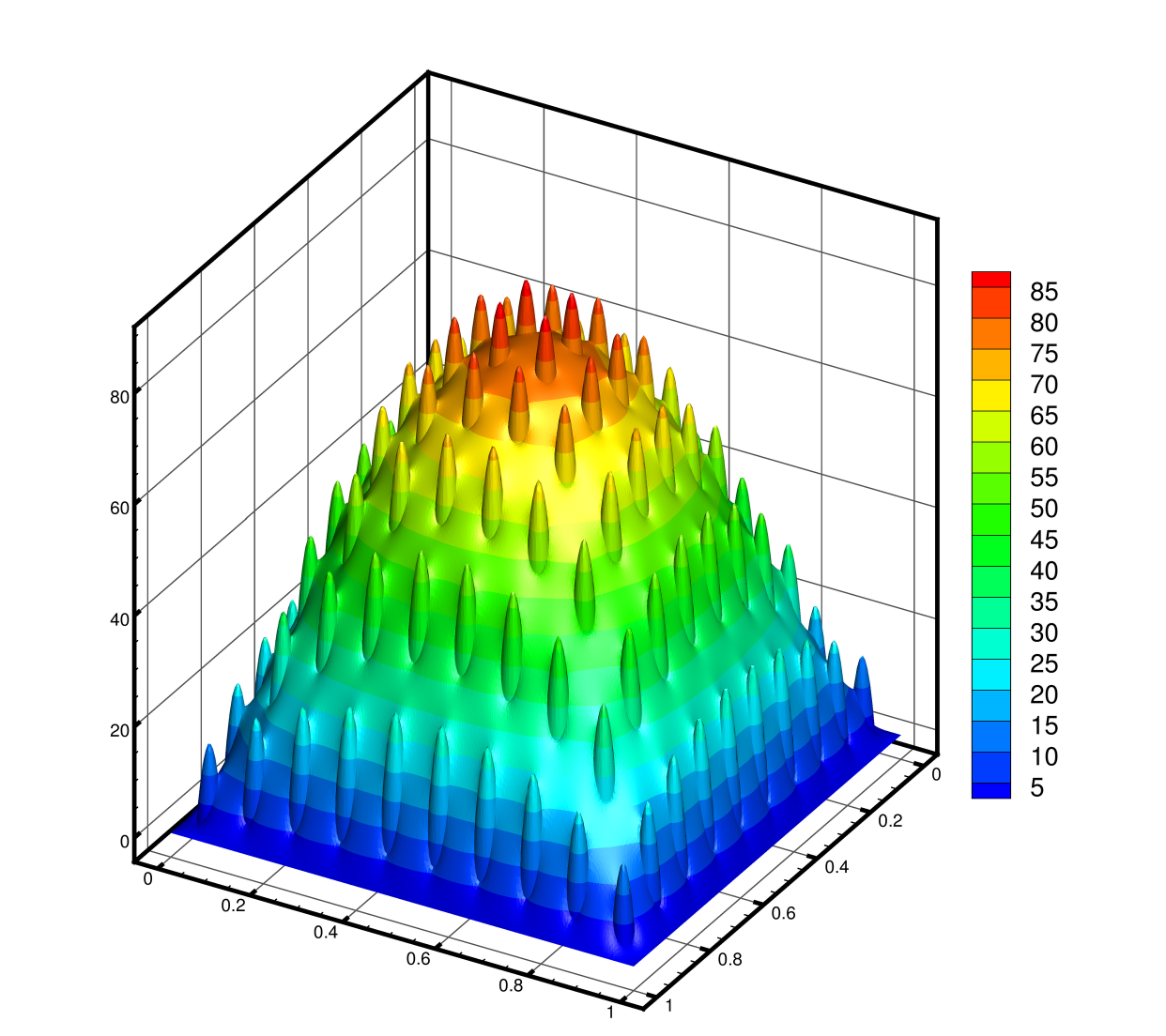} &
\includegraphics[width=0.17\textwidth]{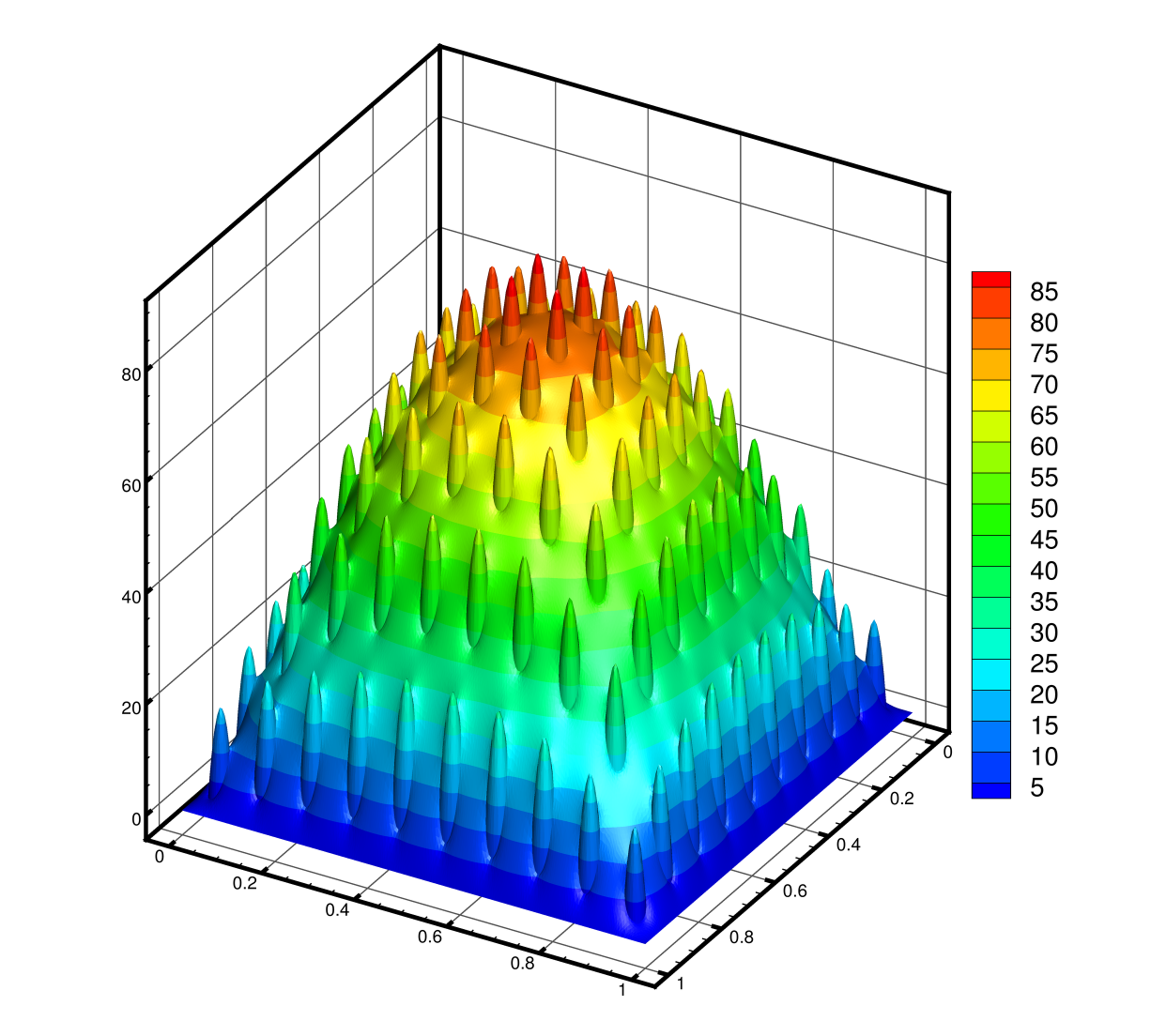} &
\includegraphics[width=0.17\textwidth]{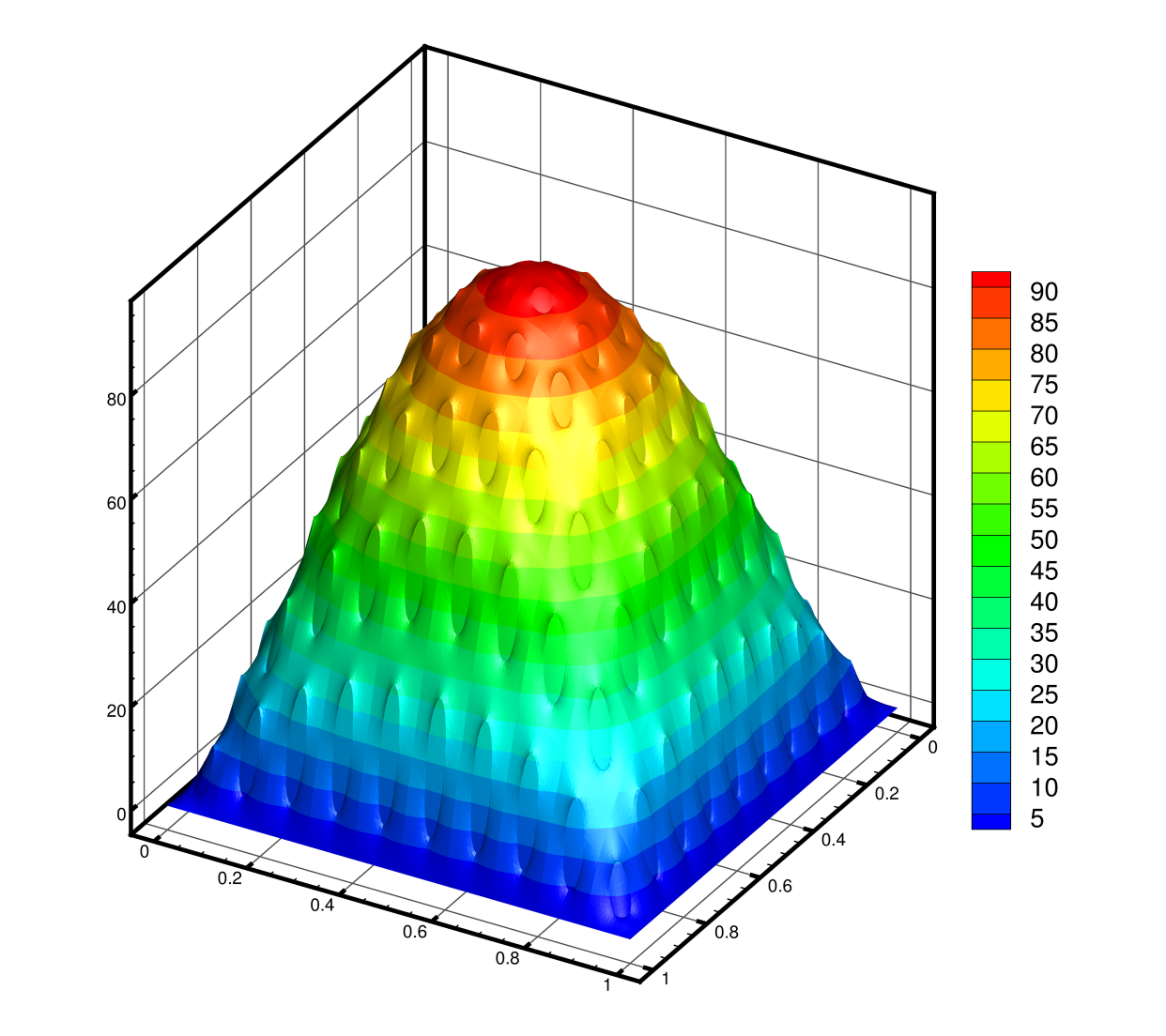} &
\includegraphics[width=0.17\textwidth]{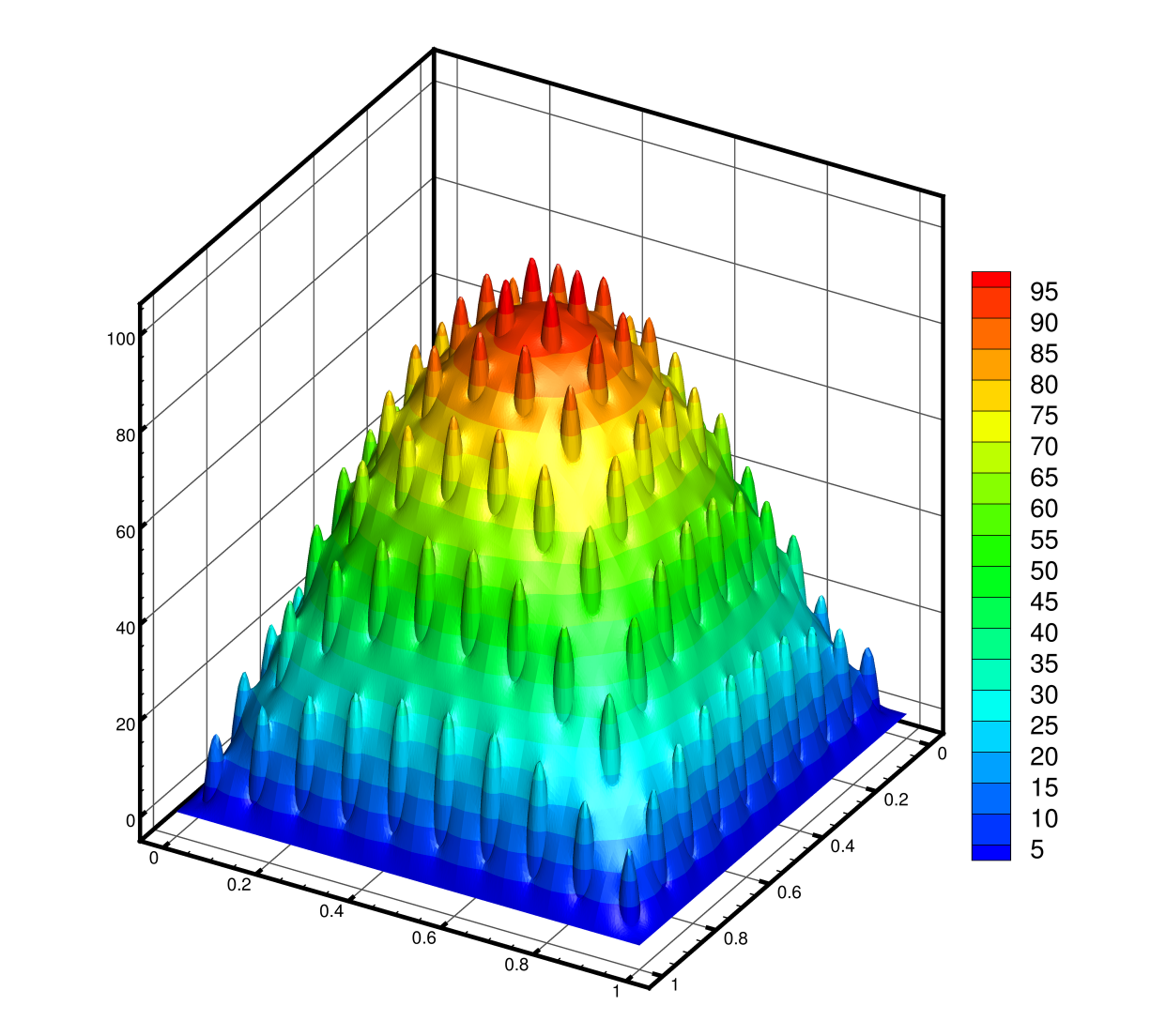} &
\includegraphics[width=0.17\textwidth]{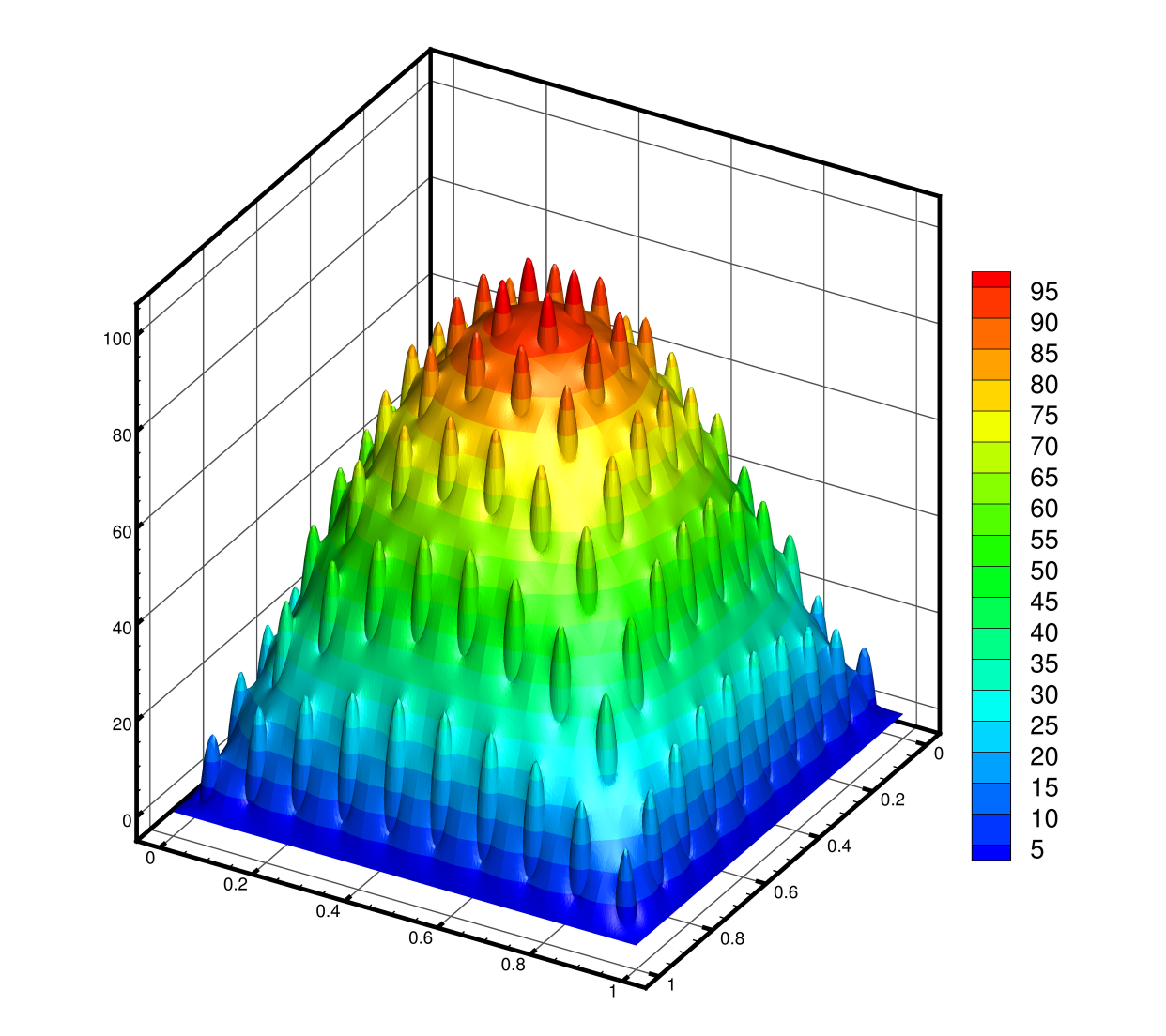}\\
(a) & (b) & (c) & (d) & (e)
\end{tabular}
\caption{Continuum 2 pressure fields for the two-dimensional circular-inclusion problem at $T=1$: (a) $u_2^\eps$; (b) $u_2^{[\mathrm{FP}]}(T)$; (c) $u_2^{(1\eps)}$; (d) $u_2^{(2\eps)}$; (e) $u_{2,r}^{\eps,n}$.}
\label{fig:circle2d-u2-fields}
\end{figure}

\begin{figure}[htbp]
\centering
\begin{tabular}{@{}ccccc@{}}
\includegraphics[width=0.17\textwidth]{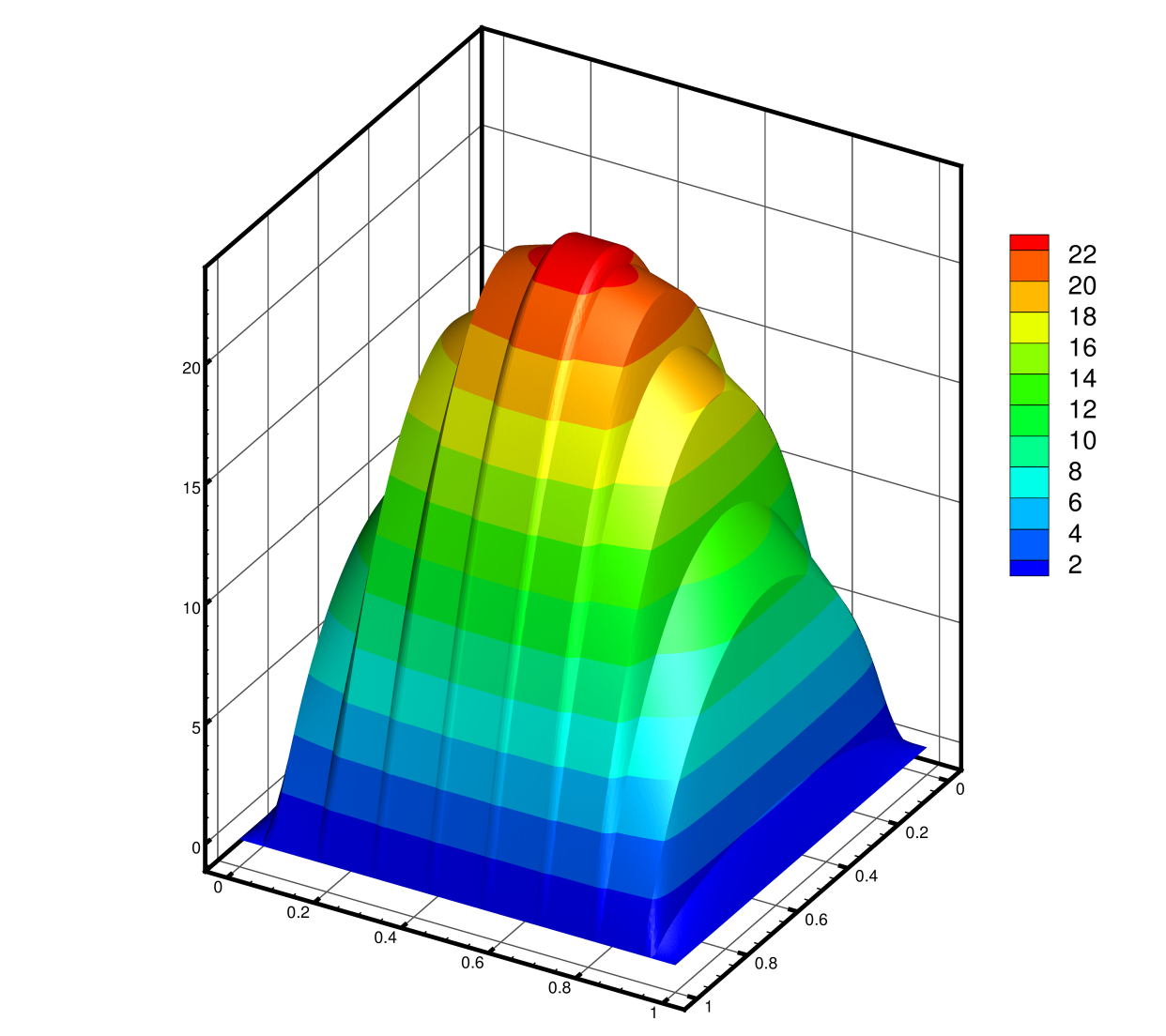} &
\includegraphics[width=0.17\textwidth]{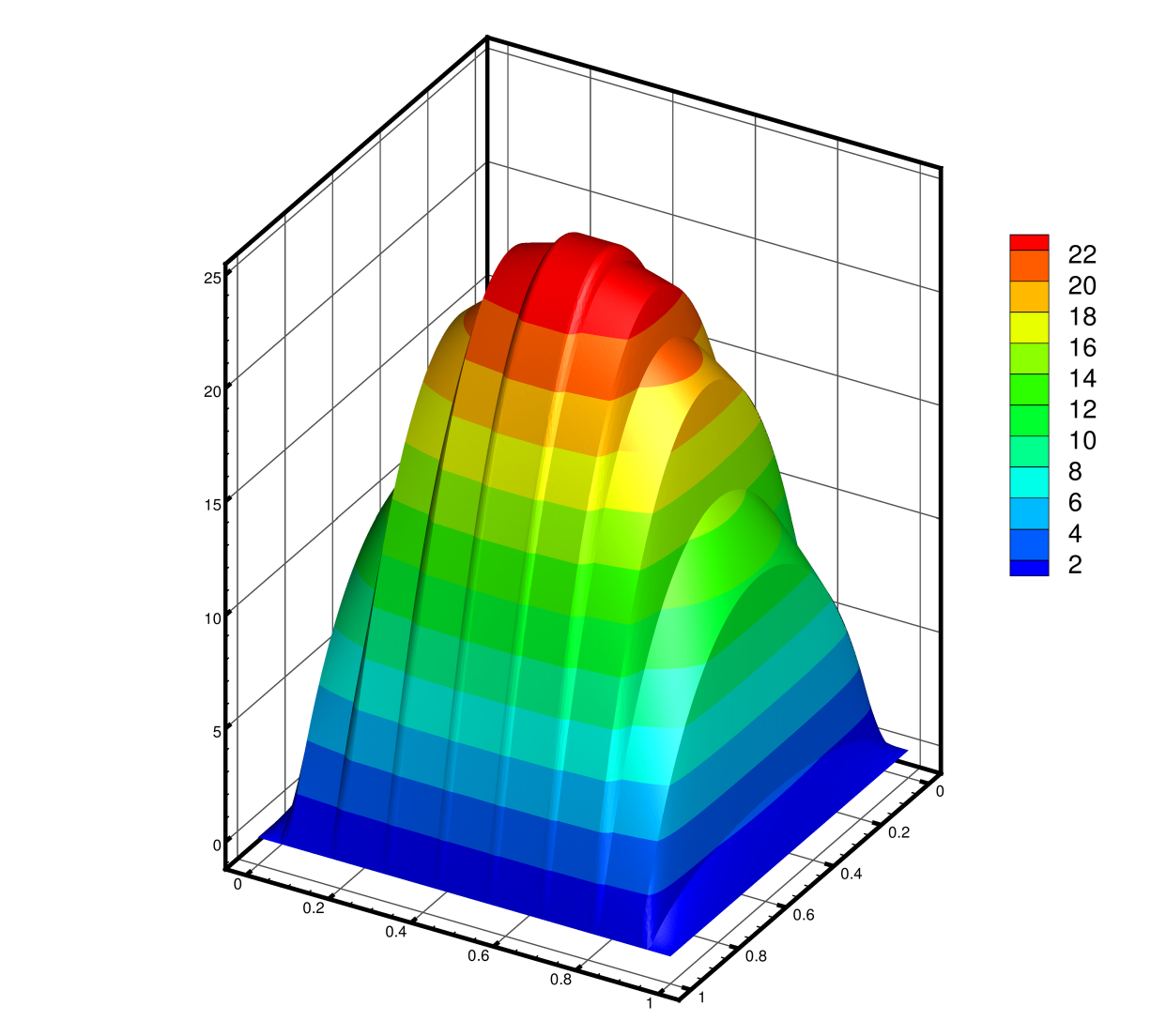} &
\includegraphics[width=0.17\textwidth]{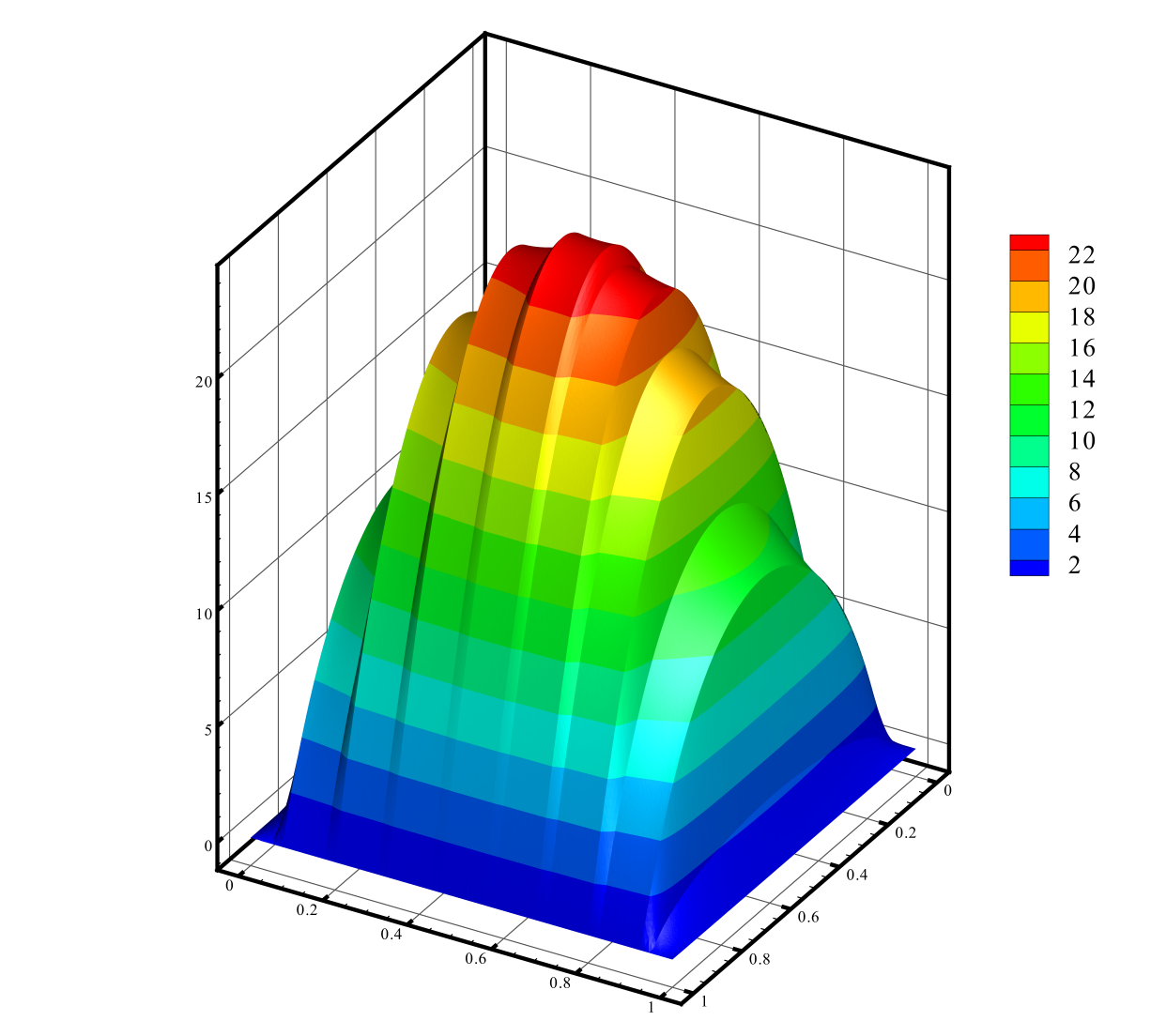} &
\includegraphics[width=0.17\textwidth]{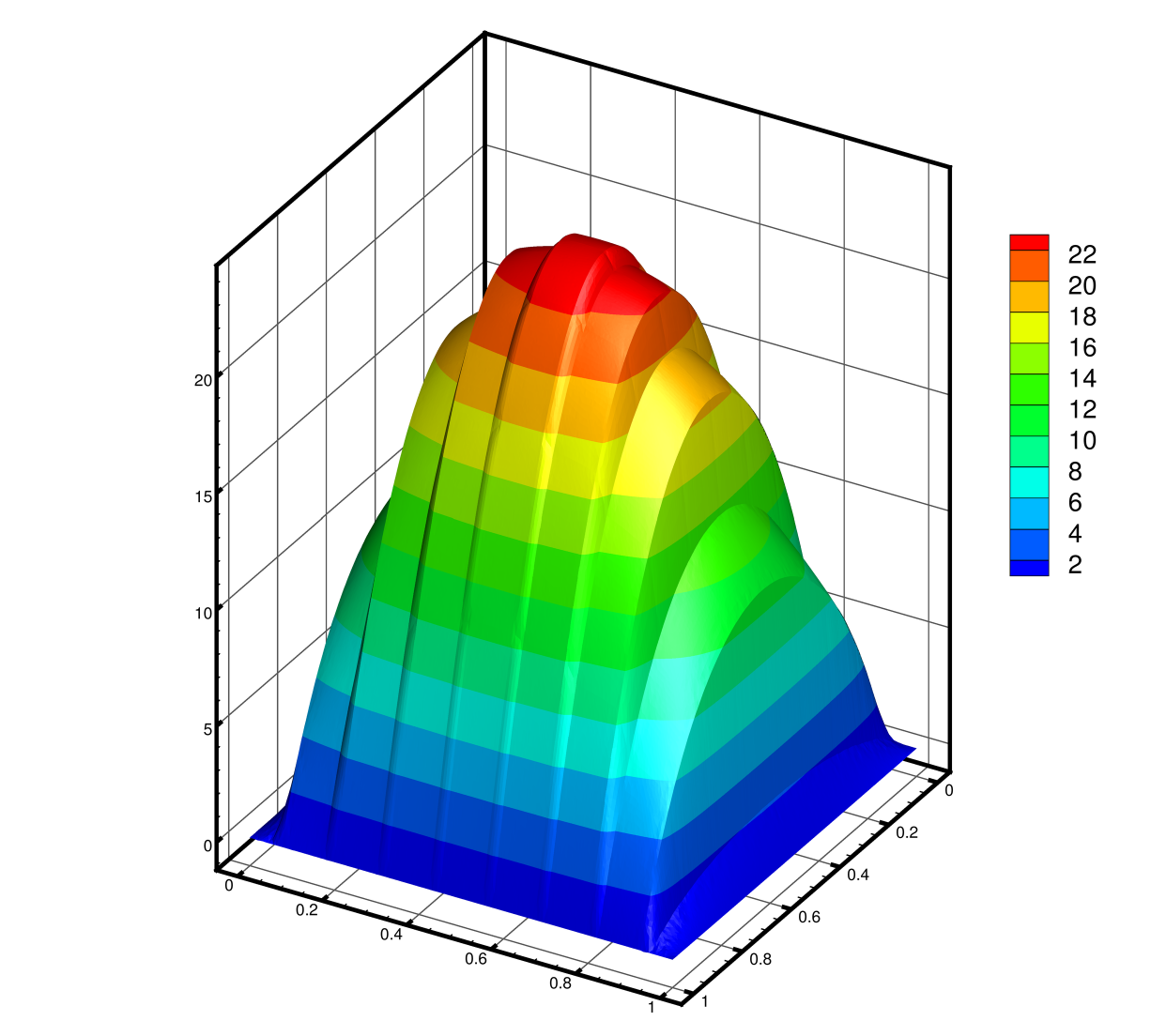} &
\includegraphics[width=0.17\textwidth]{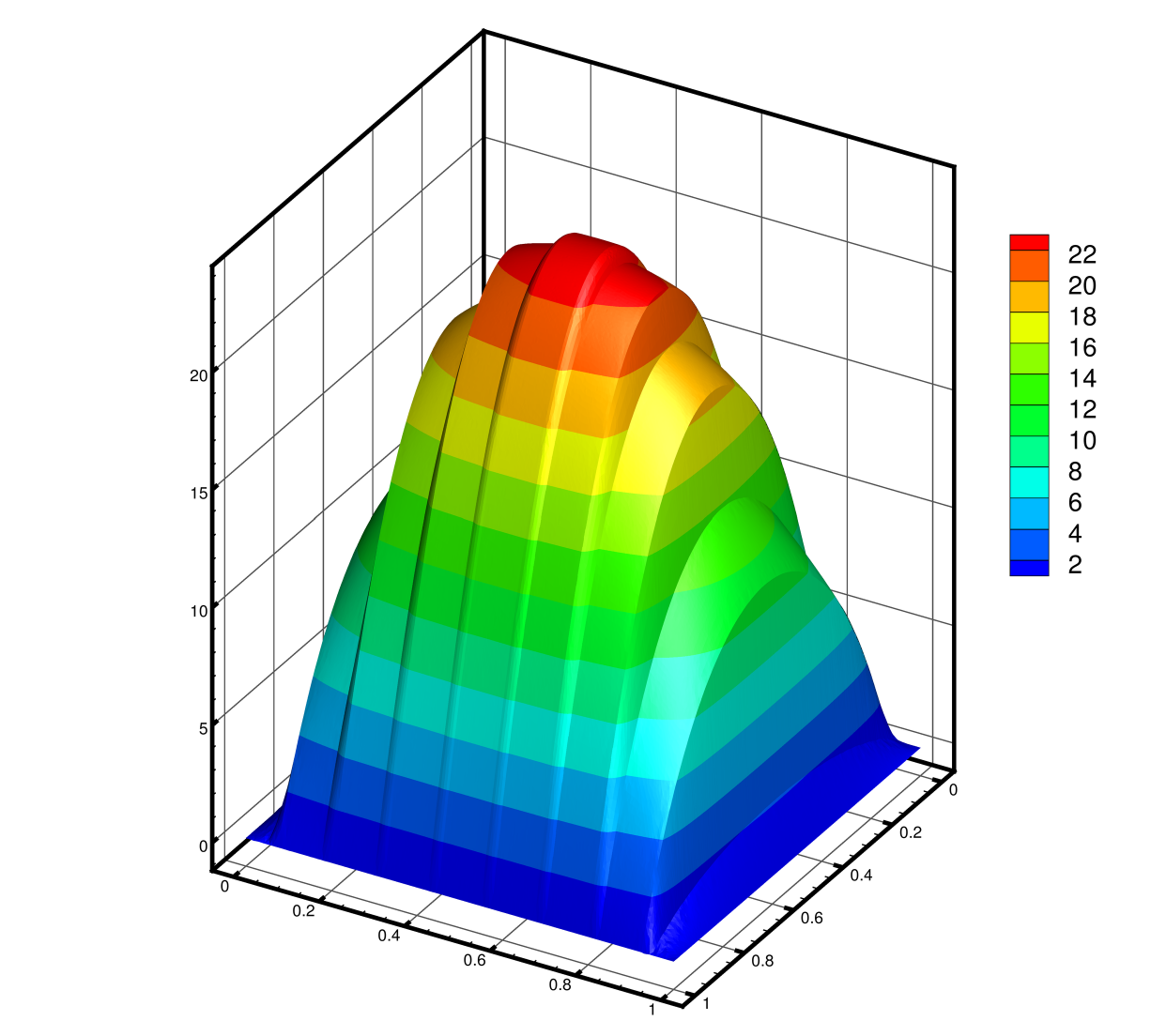}\\
(a) & (b) & (c) & (d) & (e)
\end{tabular}
\caption{Continuum 1 pressure fields for the two-dimensional channel-inclusion problem at $T=1$: (a) $u_1^\eps$; (b) $u_1^{[\mathrm{FP}]}(T)$; (c) $u_1^{(1\eps)}$; (d) $u_1^{(2\eps)}$; (e) $u_{1,r}^{\eps,n}$.}
\label{fig:channel2d-u1-fields}
\end{figure}

\begin{figure}[htbp]
\centering
\begin{tabular}{@{}ccccc@{}}
\includegraphics[width=0.17\textwidth]{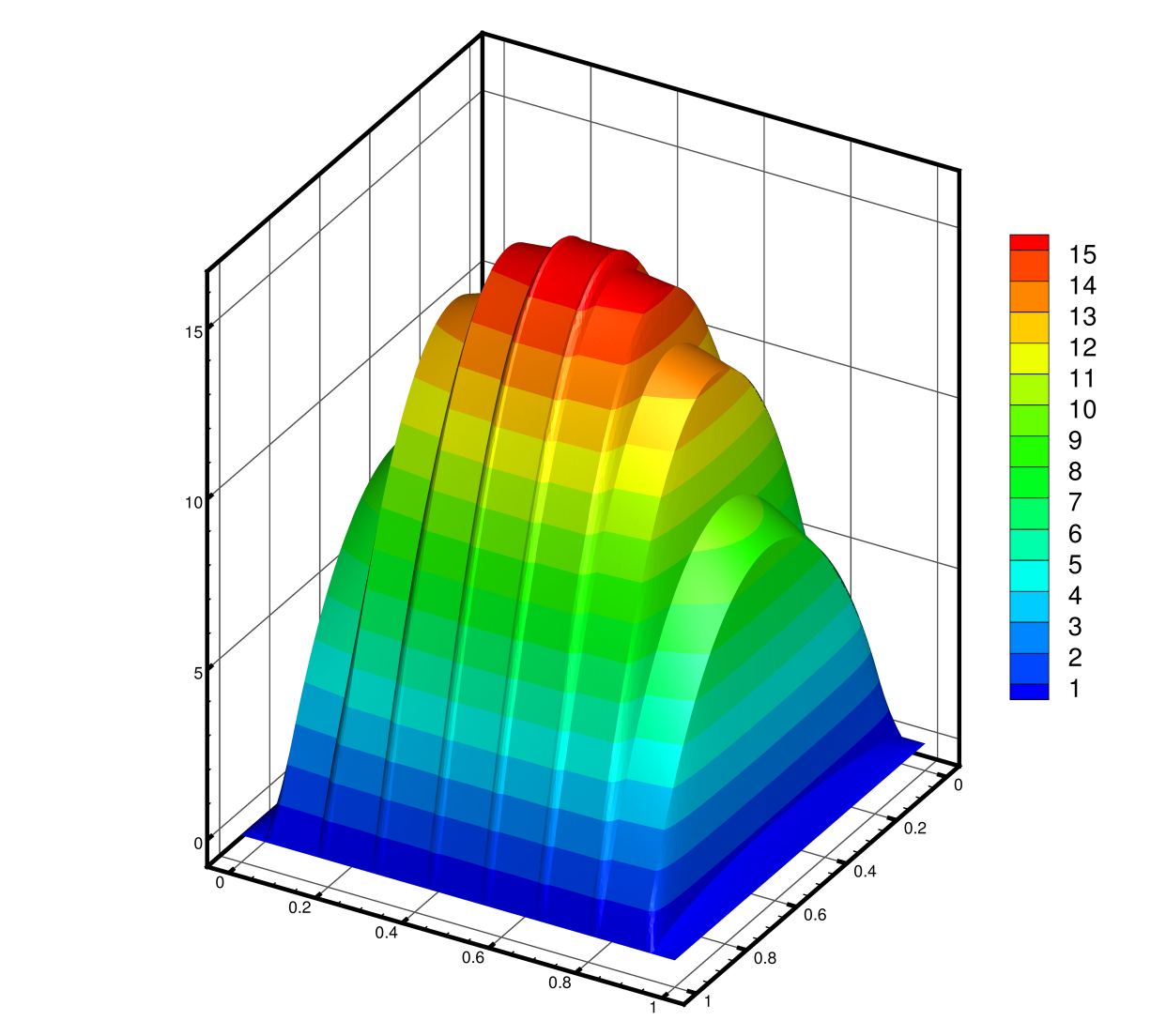} &
\includegraphics[width=0.17\textwidth]{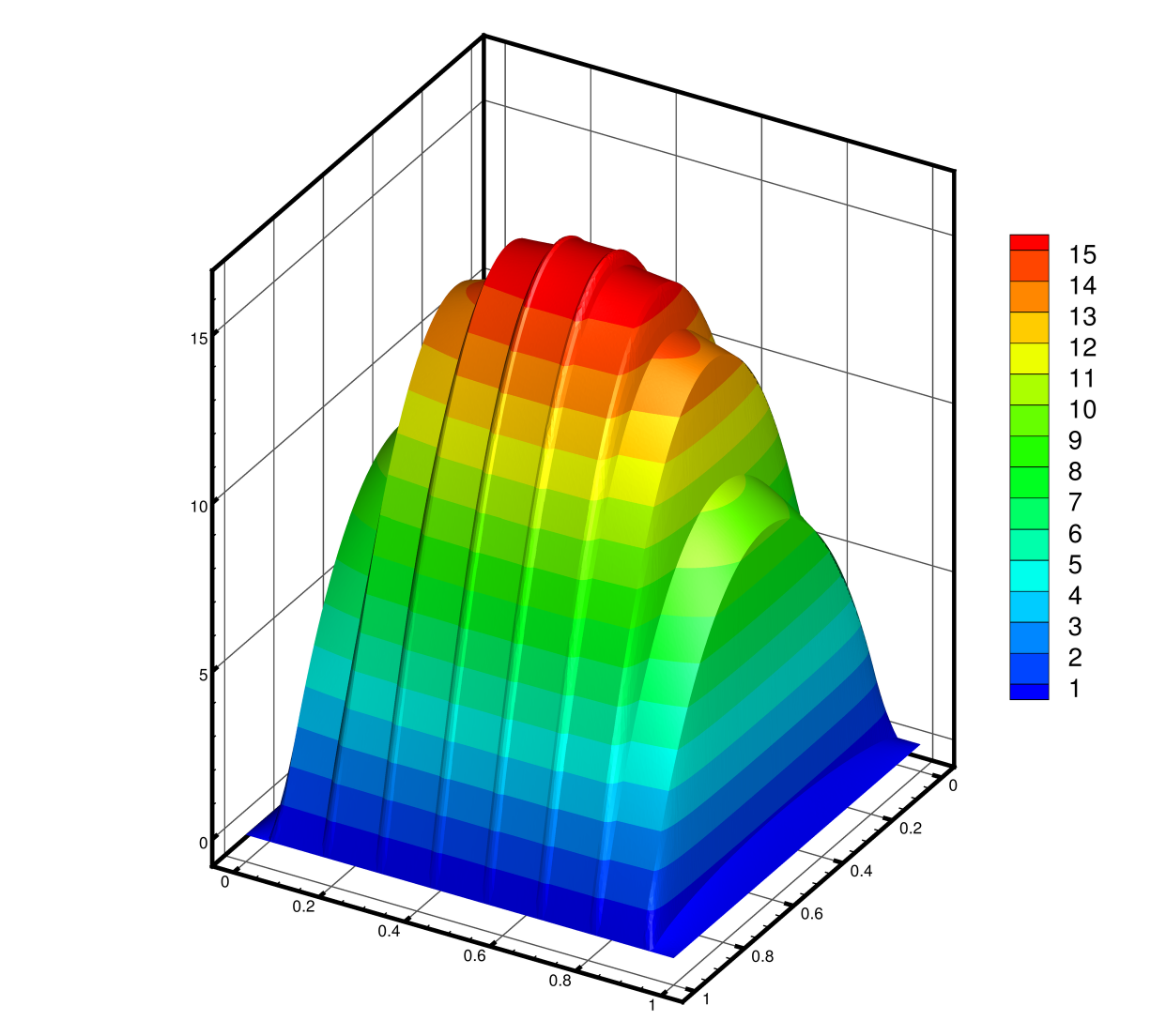} &
\includegraphics[width=0.17\textwidth]{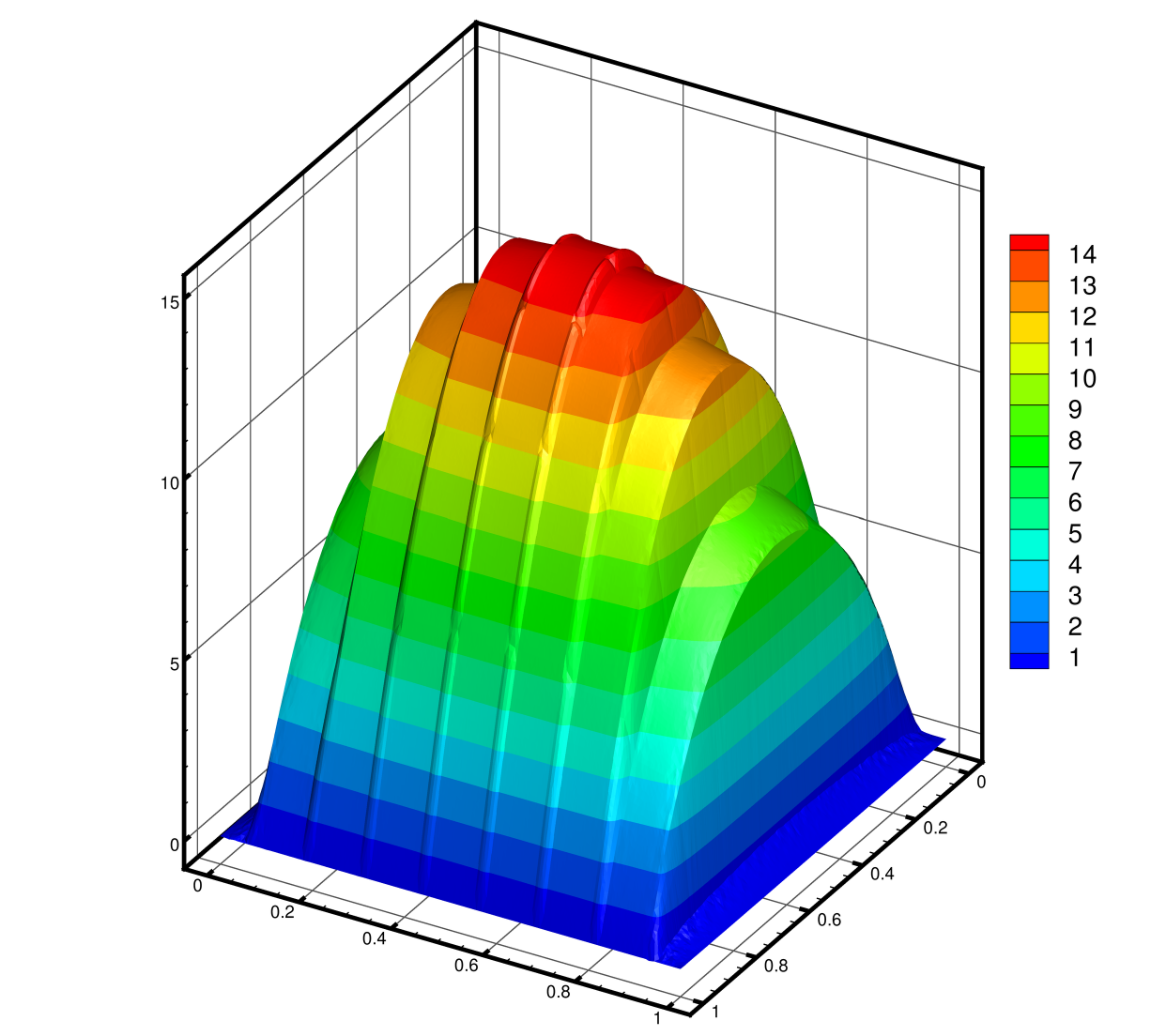} &
\includegraphics[width=0.17\textwidth]{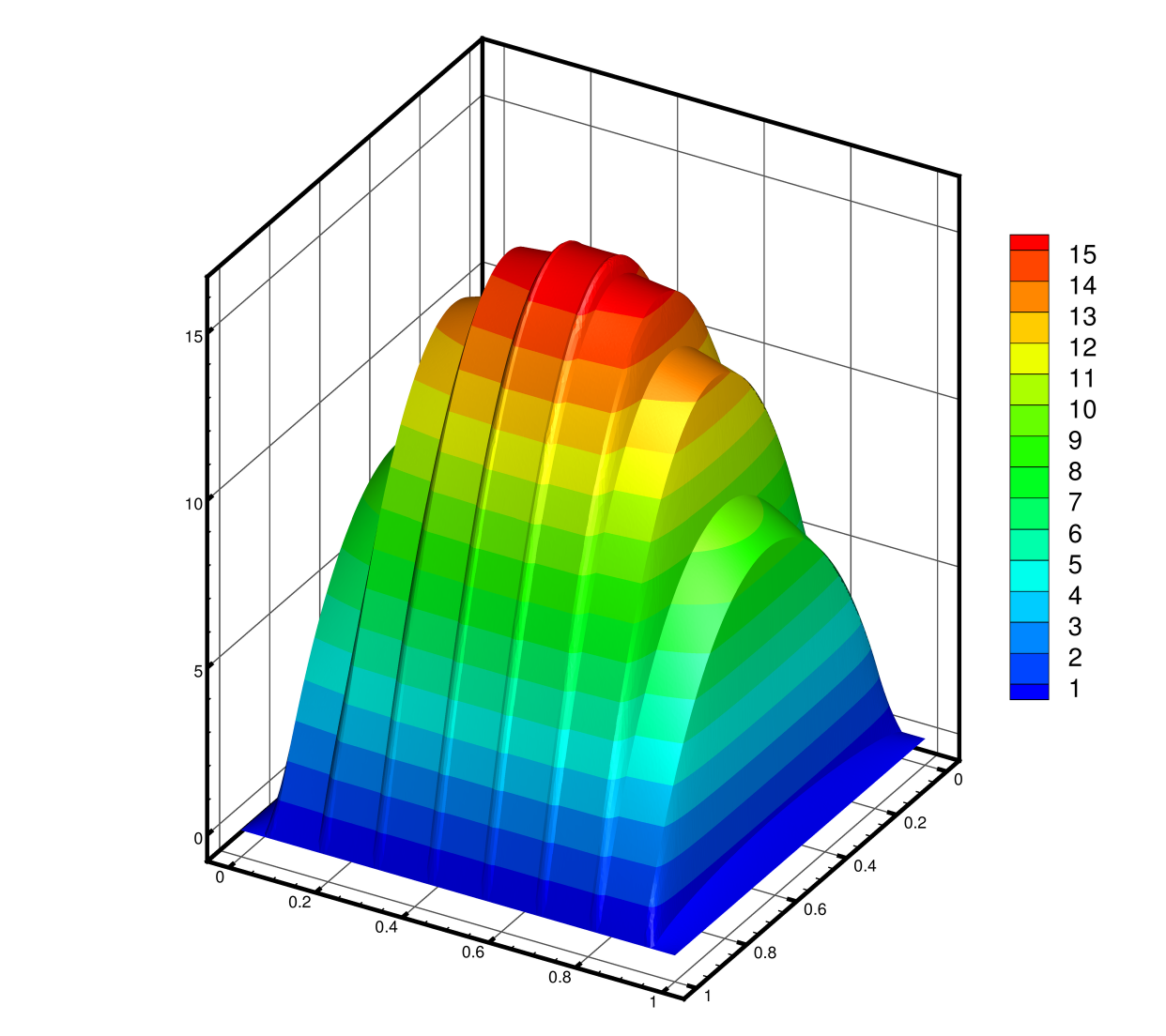} &
\includegraphics[width=0.17\textwidth]{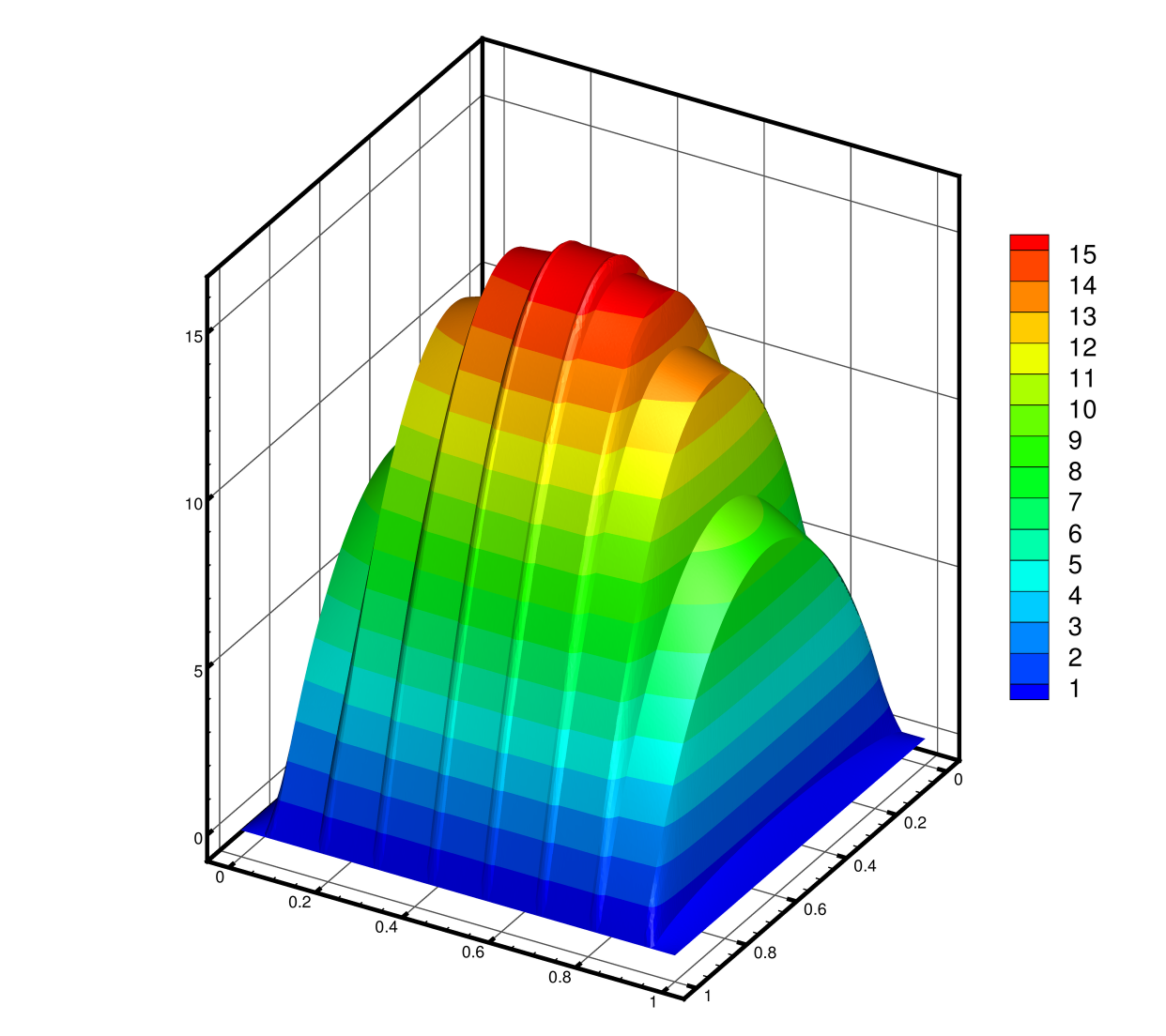}\\
(a) & (b) & (c) & (d) & (e)
\end{tabular}
\caption{Continuum 2 pressure fields for the two-dimensional channel-inclusion problem at $T=1$: (a) $u_2^\eps$; (b) $u_2^{[\mathrm{FP}]}(T)$; (c) $u_2^{(1\eps)}$; (d) $u_2^{(2\eps)}$; (e) $u_{2,r}^{\eps,n}$.}
\label{fig:channel2d-u2-fields}
\end{figure}

\begin{figure}[htbp]
\centering
\begin{tabular}{@{}ccccc@{}}
\includegraphics[width=0.17\textwidth]{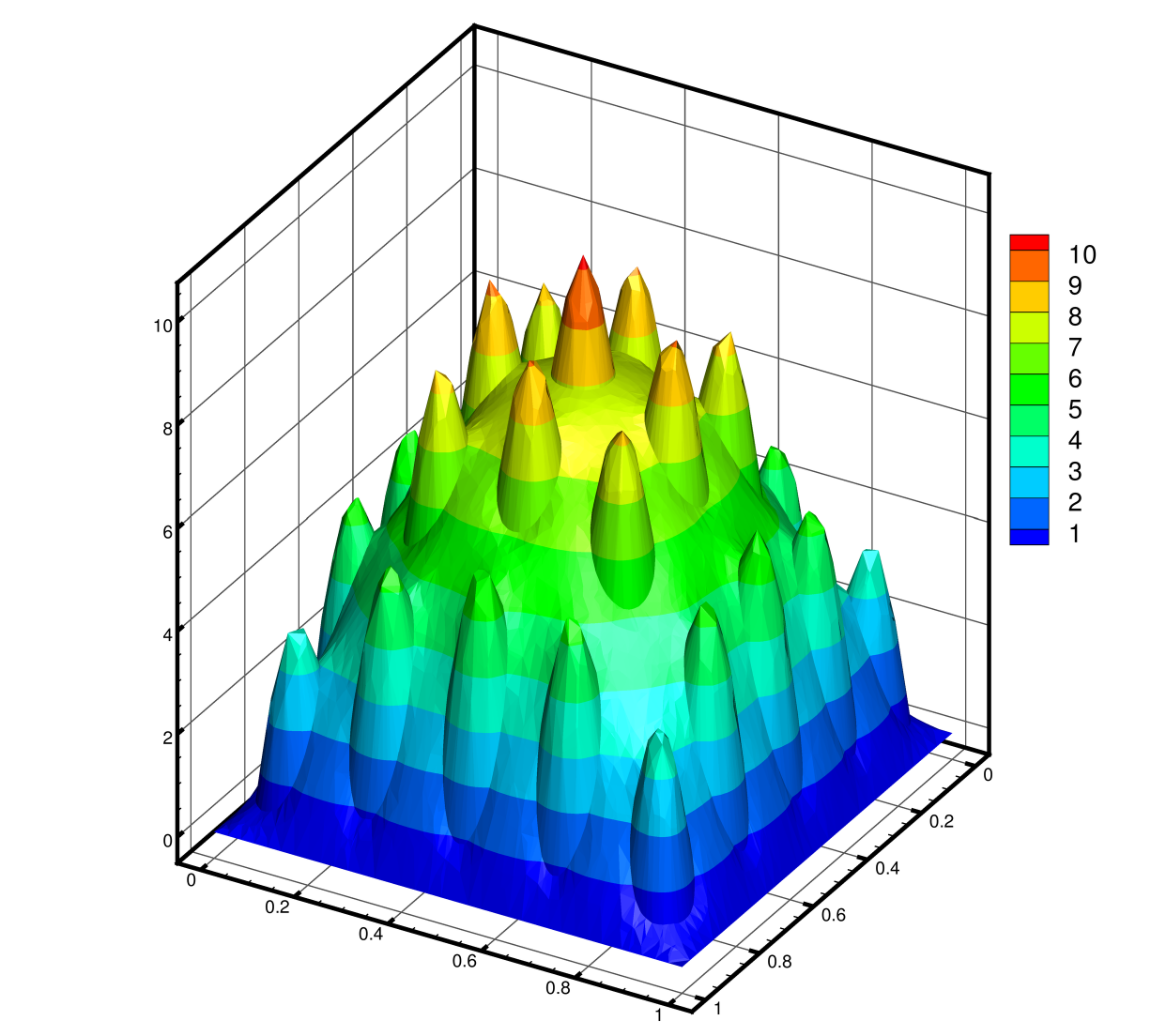} &
\includegraphics[width=0.17\textwidth]{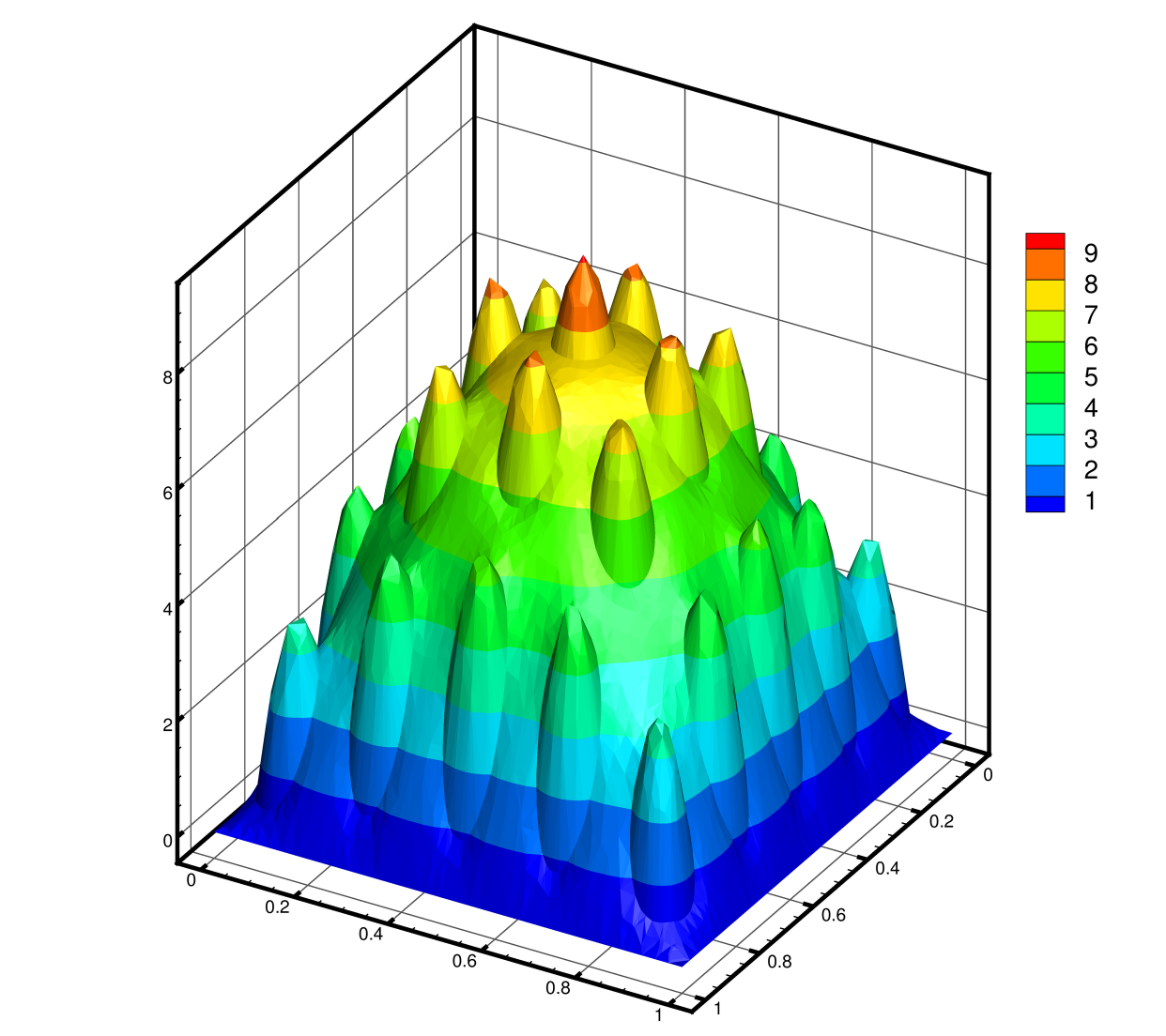} &
\includegraphics[width=0.17\textwidth]{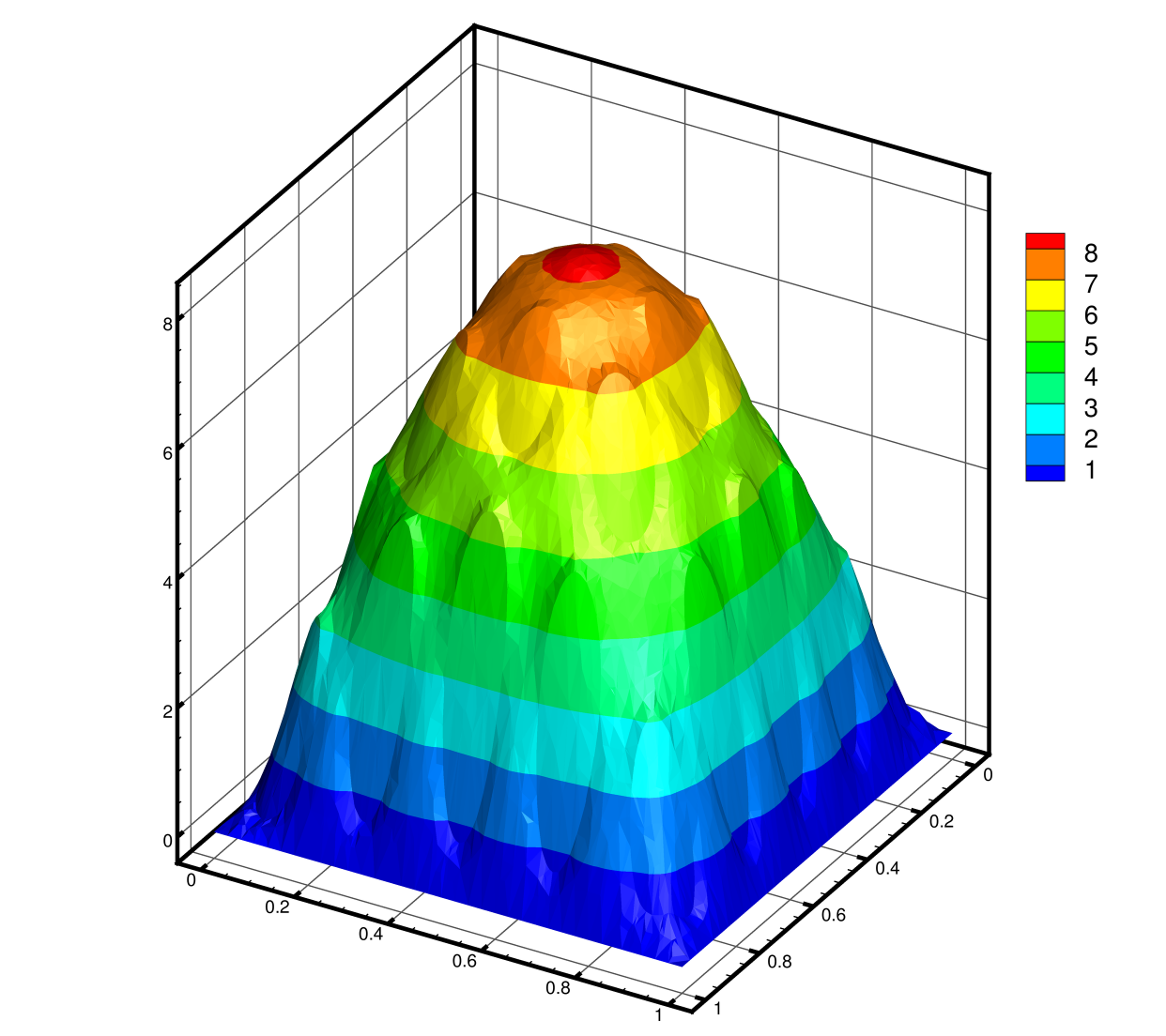} &
\includegraphics[width=0.17\textwidth]{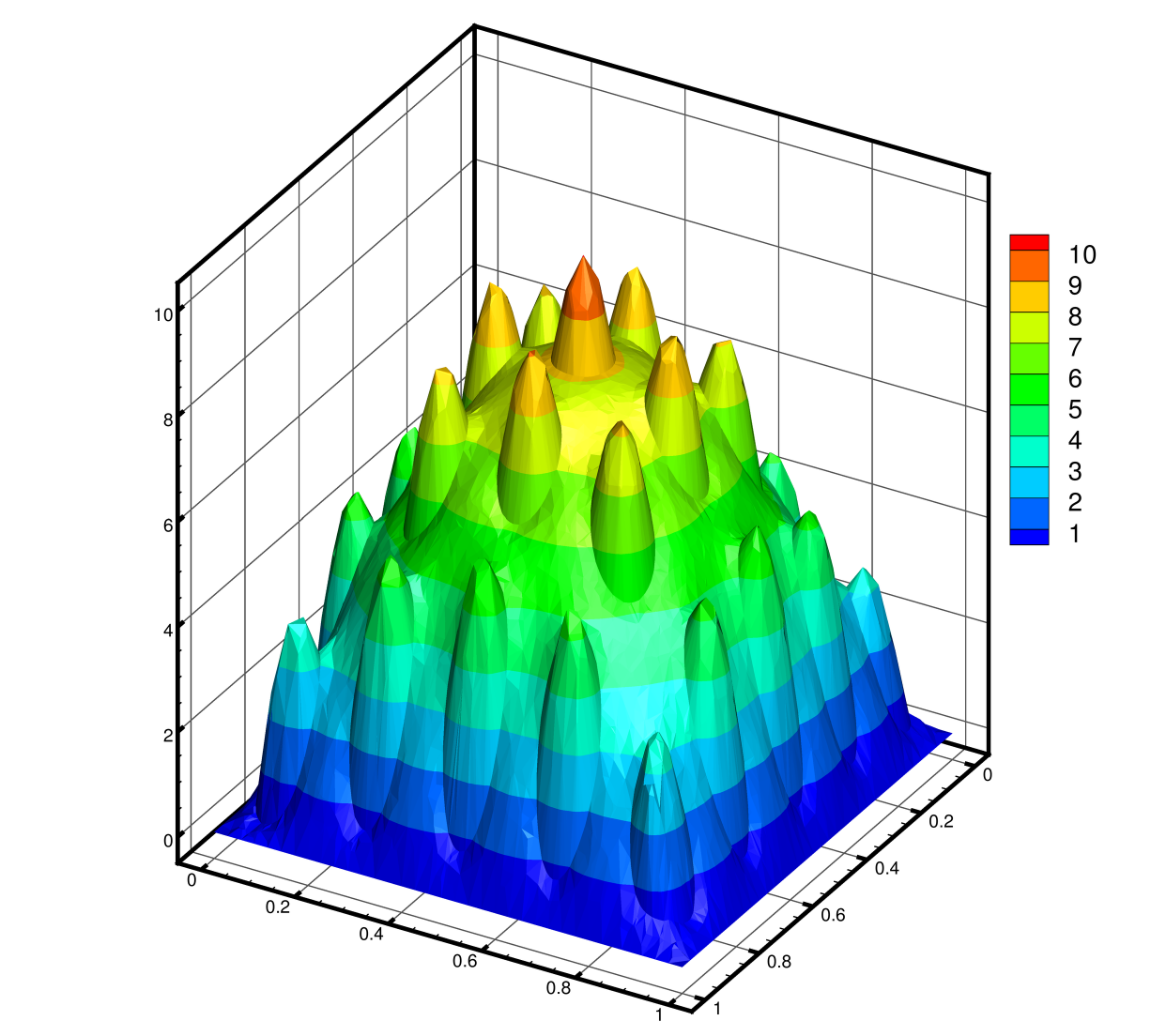} &
\includegraphics[width=0.17\textwidth]{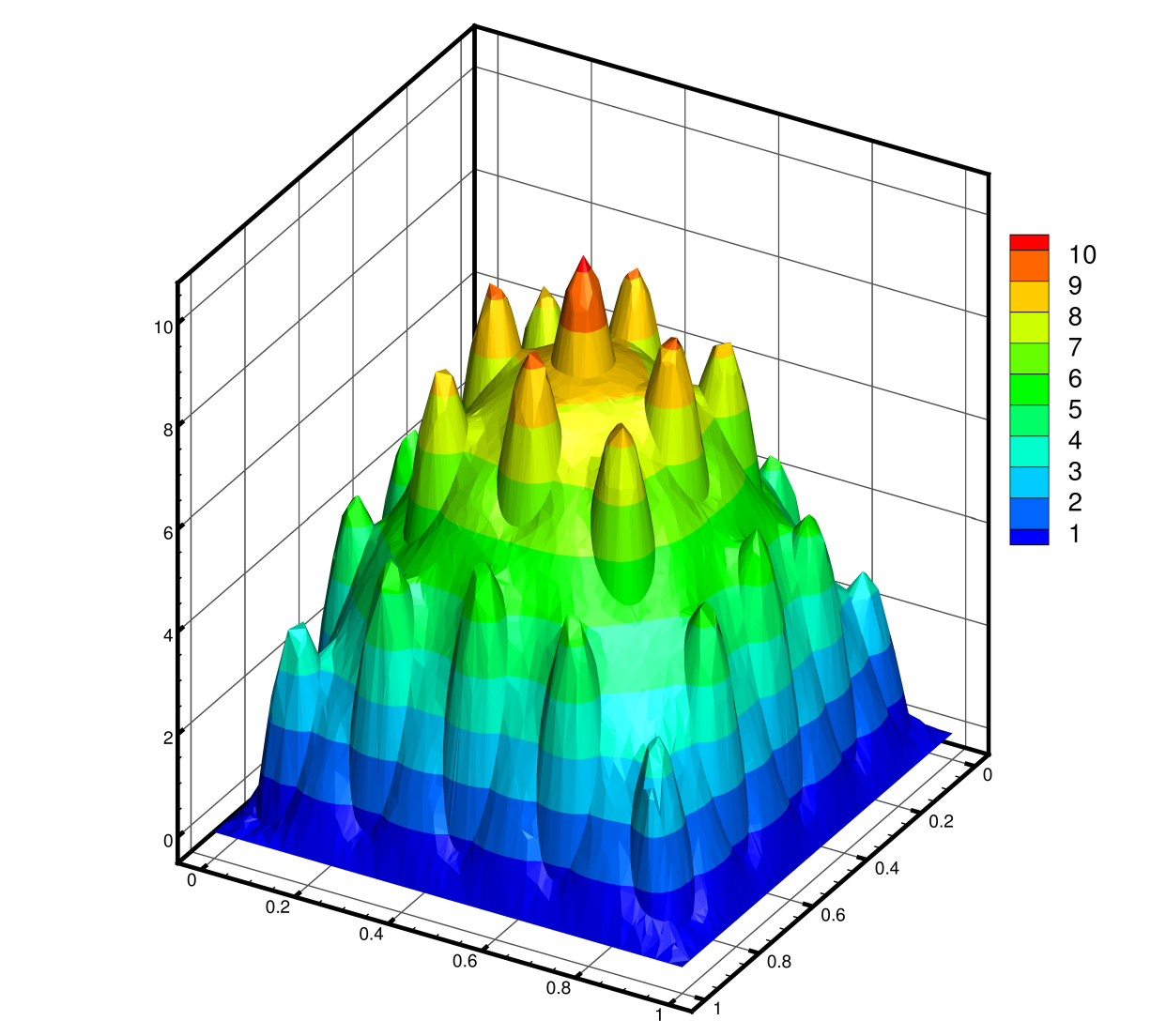}\\
(a) & (b) & (c) & (d) & (e)
\end{tabular}
\caption{Continuum 1 pressure fields for the three-dimensional spherical-inclusion problem at $T=1$ (output section): (a) $u_1^\eps$; (b) $u_1^{[\mathrm{FP}]}(T)$; (c) $u_1^{(1\eps)}$; (d) $u_1^{(2\eps)}$; (e) $u_{1,r}^{\eps,n}$.}
\label{fig:sphere3d-u1-fields}
\end{figure}

\begin{figure}[htbp]
\centering
\begin{tabular}{@{}ccccc@{}}
\includegraphics[width=0.17\textwidth]{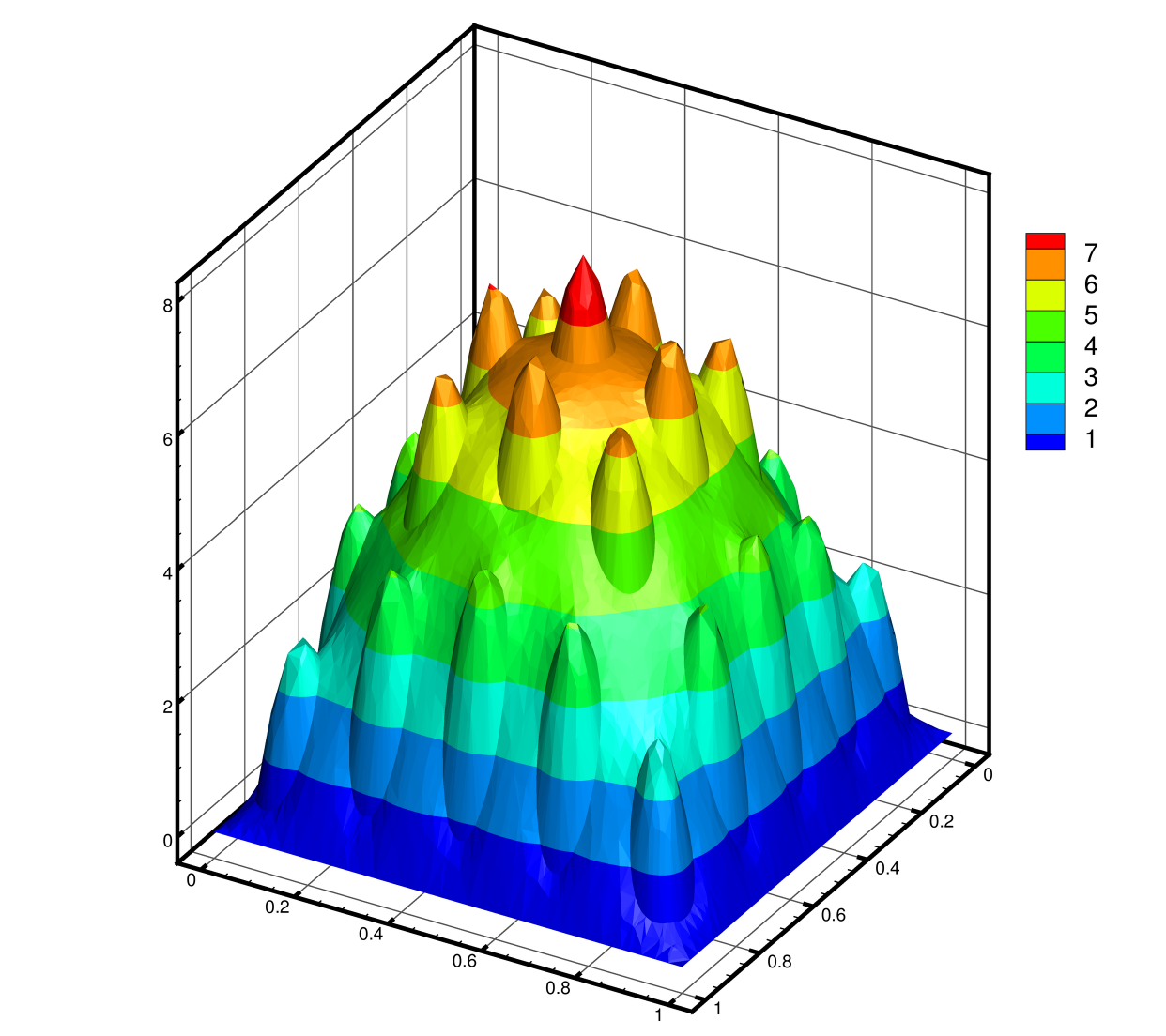} &
\includegraphics[width=0.17\textwidth]{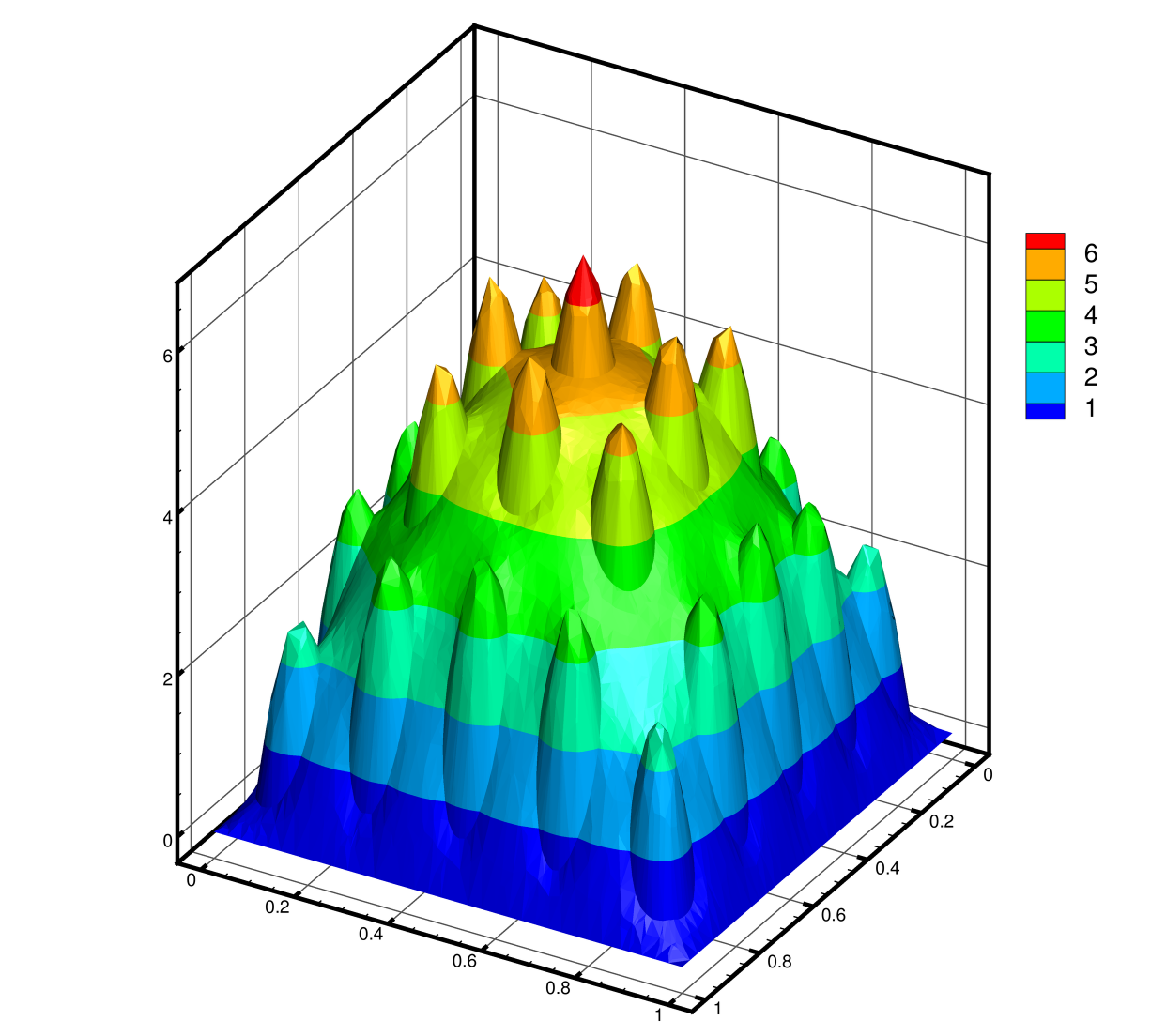} &
\includegraphics[width=0.17\textwidth]{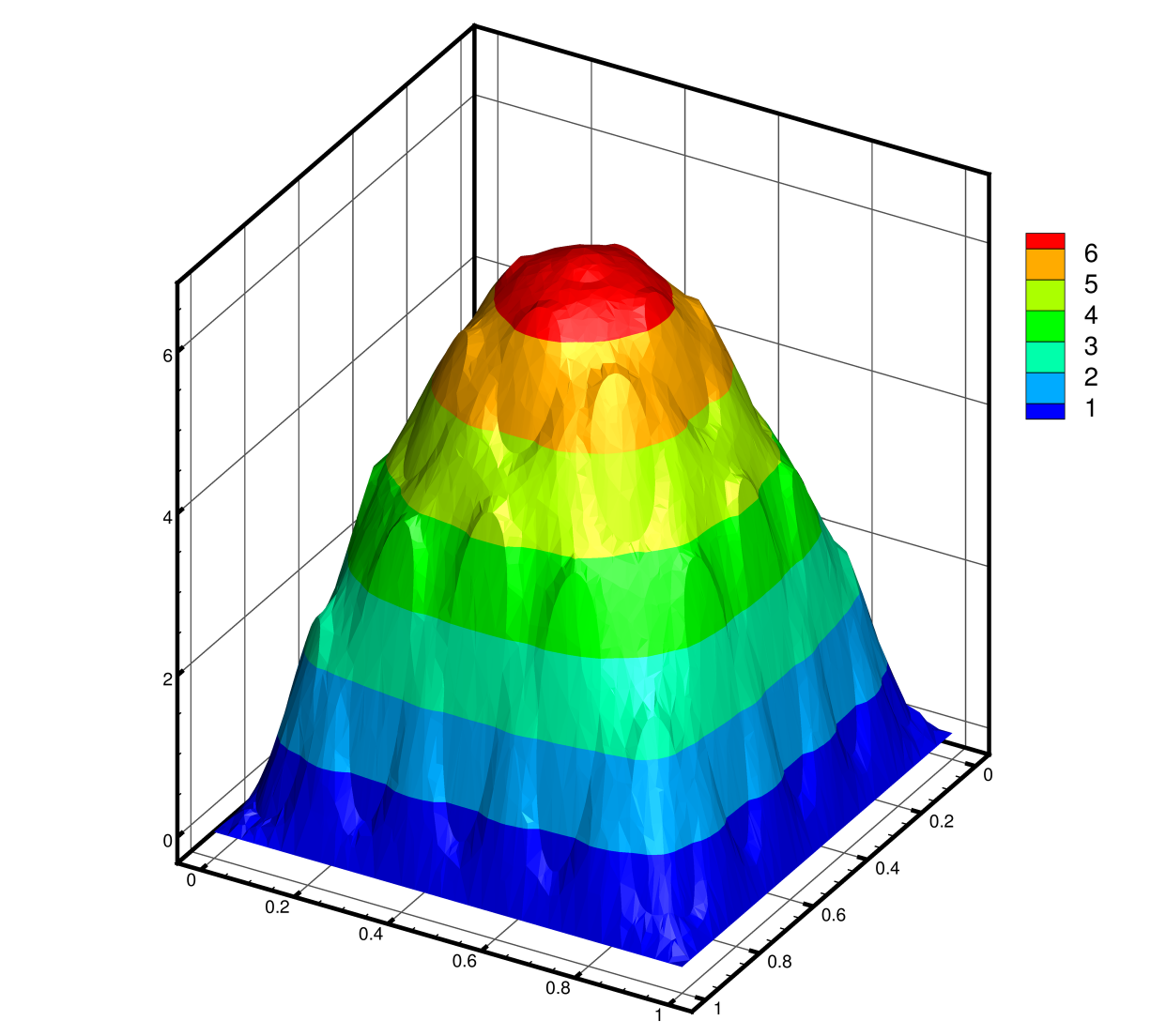} &
\includegraphics[width=0.17\textwidth]{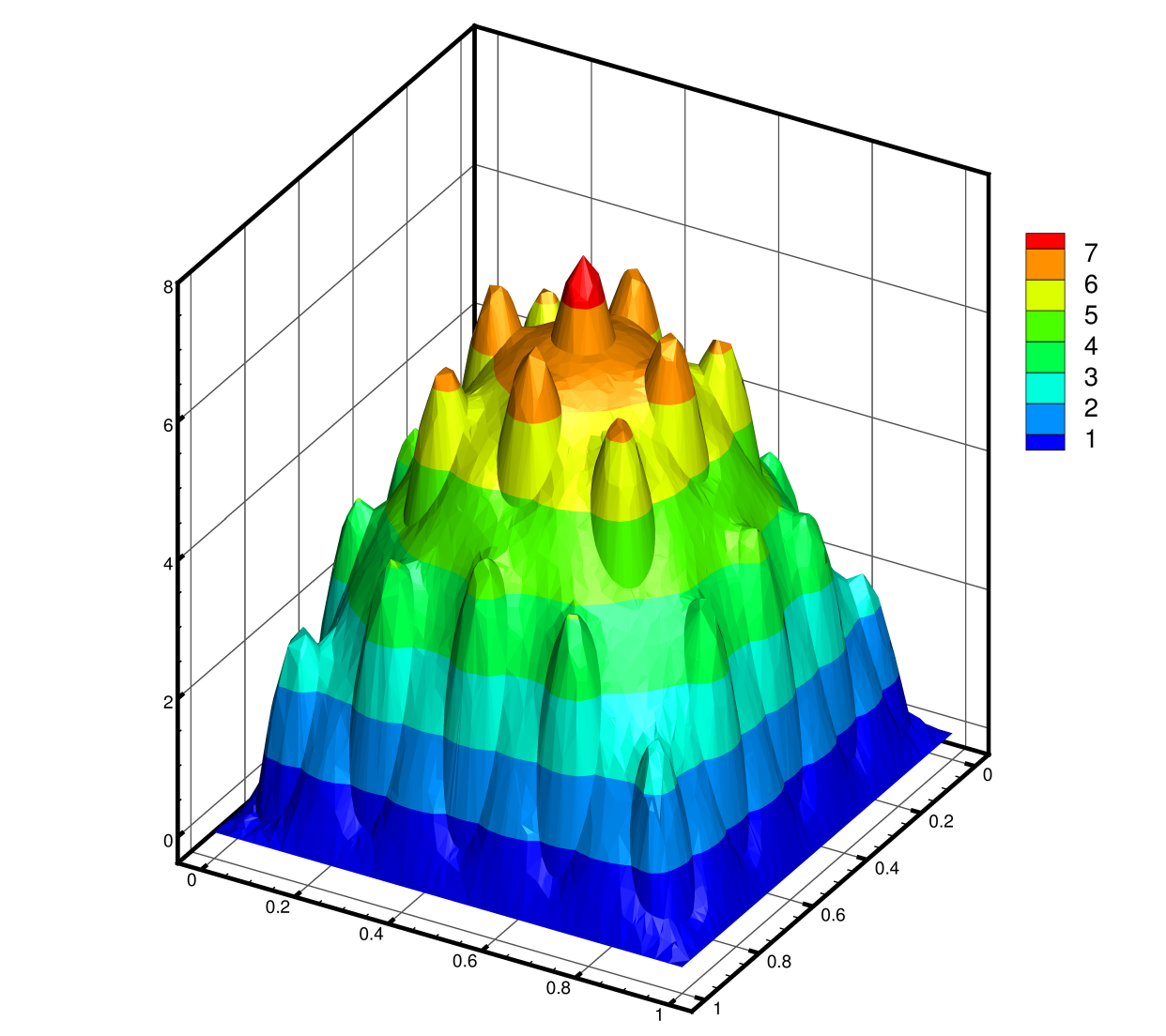} &
\includegraphics[width=0.17\textwidth]{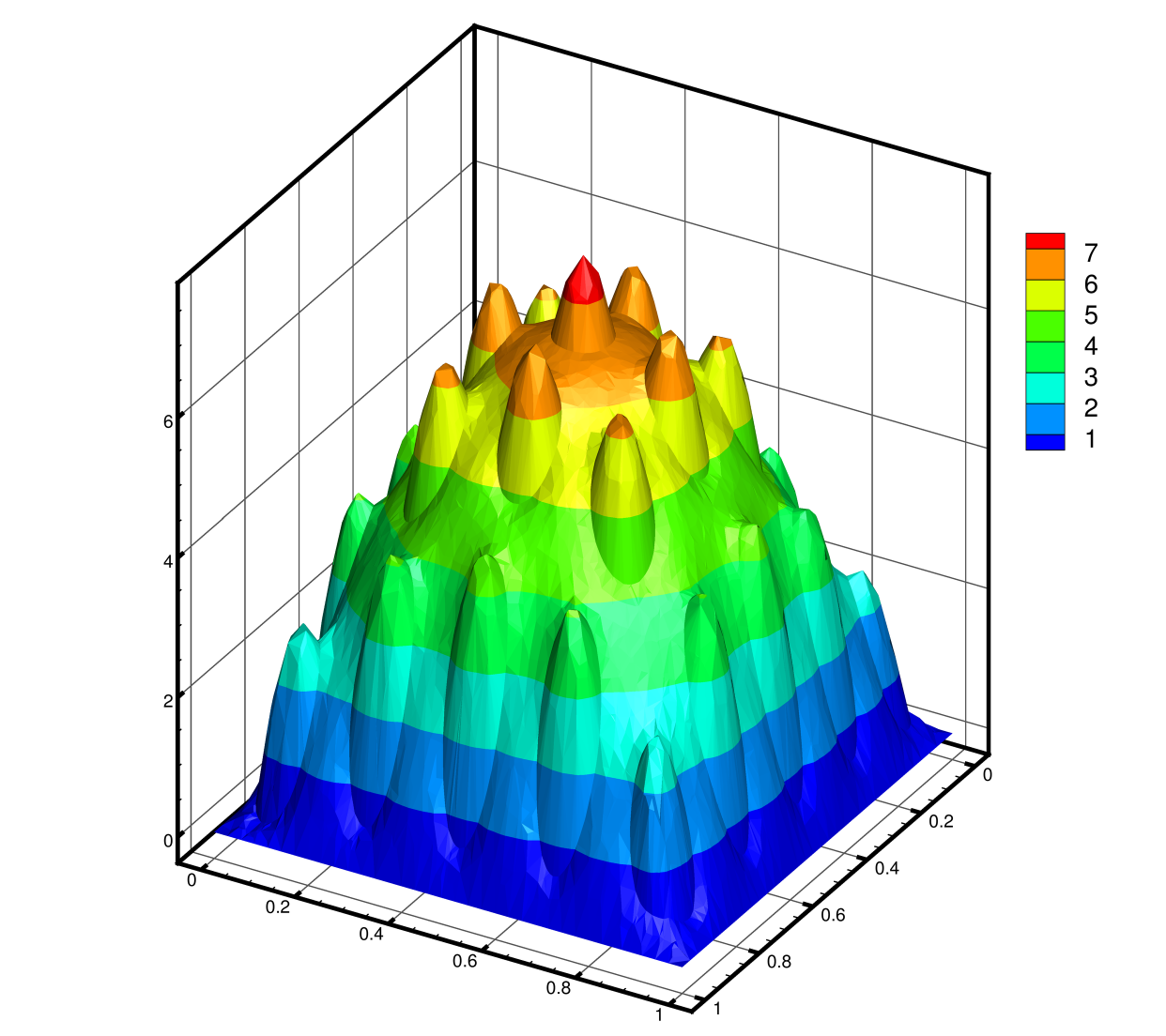}\\
(a) & (b) & (c) & (d) & (e)
\end{tabular}
\caption{Continuum 2 pressure fields for the three-dimensional spherical-inclusion problem at $T=1$ (output section): (a) $u_2^\eps$; (b) $u_2^{[\mathrm{FP}]}(T)$; (c) $u_2^{(1\eps)}$; (d) $u_2^{(2\eps)}$; (e) $u_{2,r}^{\eps,n}$.}
\label{fig:sphere3d-u2-fields}
\end{figure}

\begin{figure}[htbp]
\centering
\begin{tabular}{@{}ccccc@{}}
\includegraphics[width=0.17\textwidth]{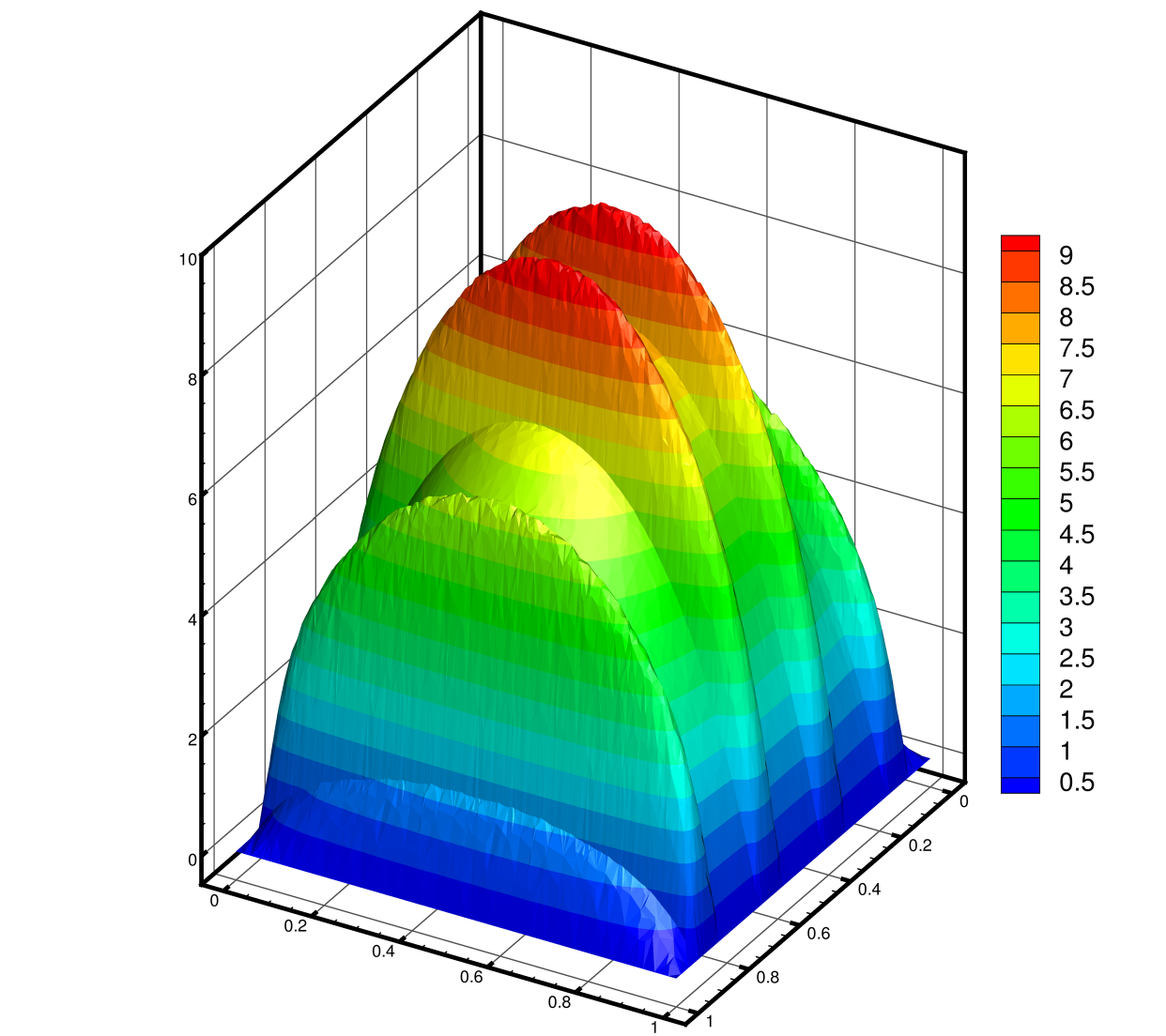} &
\includegraphics[width=0.17\textwidth]{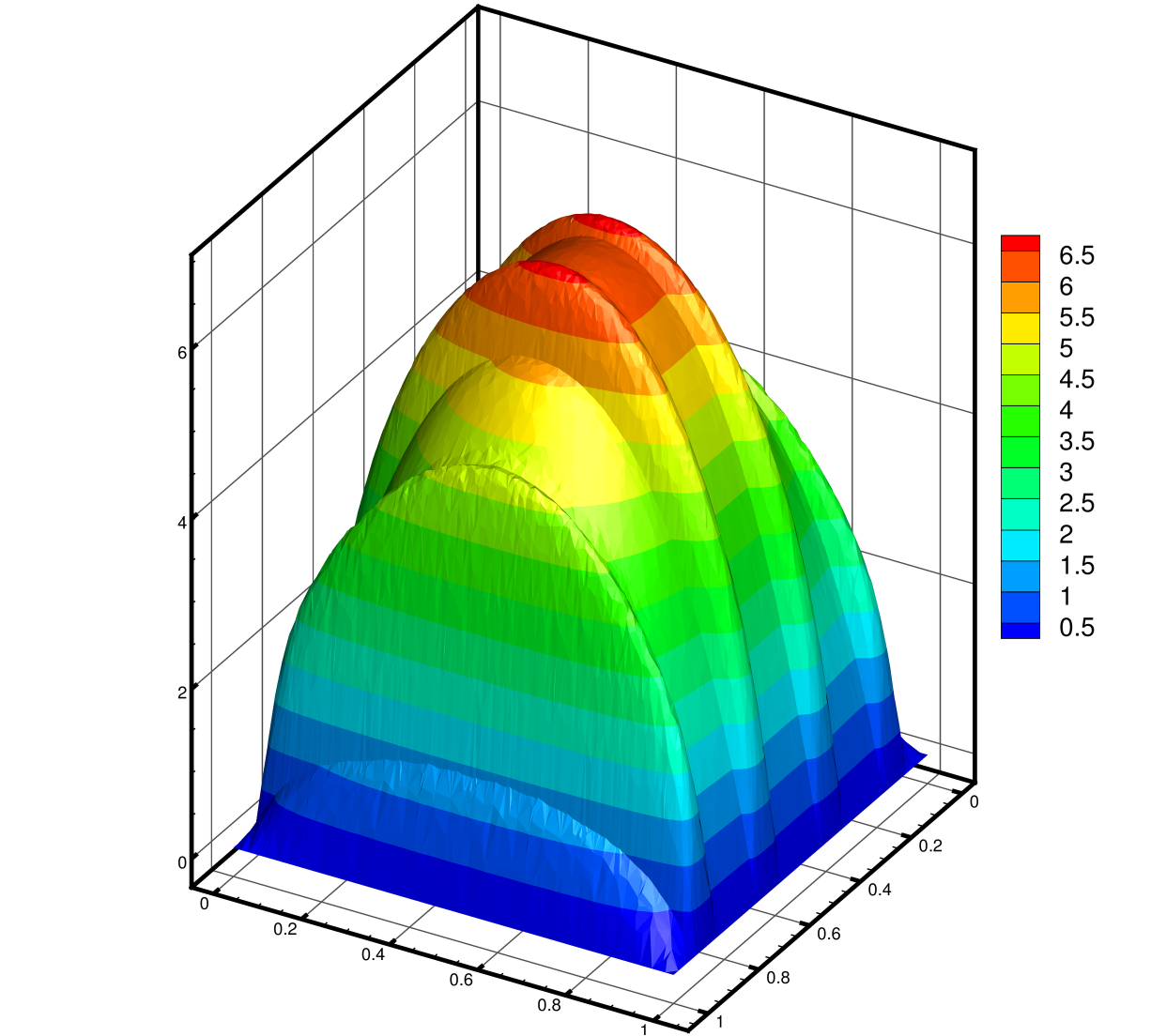} &
\includegraphics[width=0.17\textwidth]{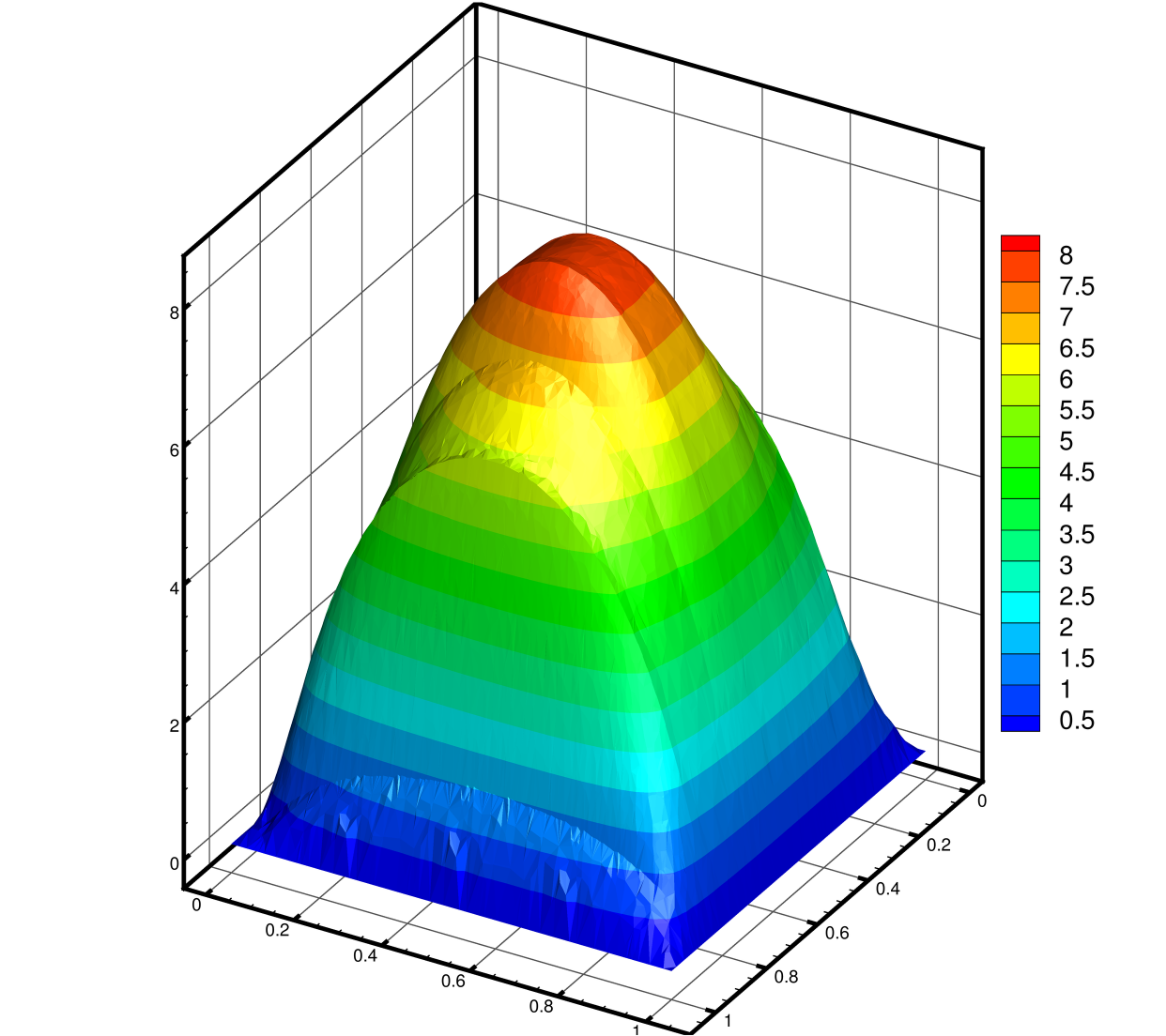} &
\includegraphics[width=0.17\textwidth]{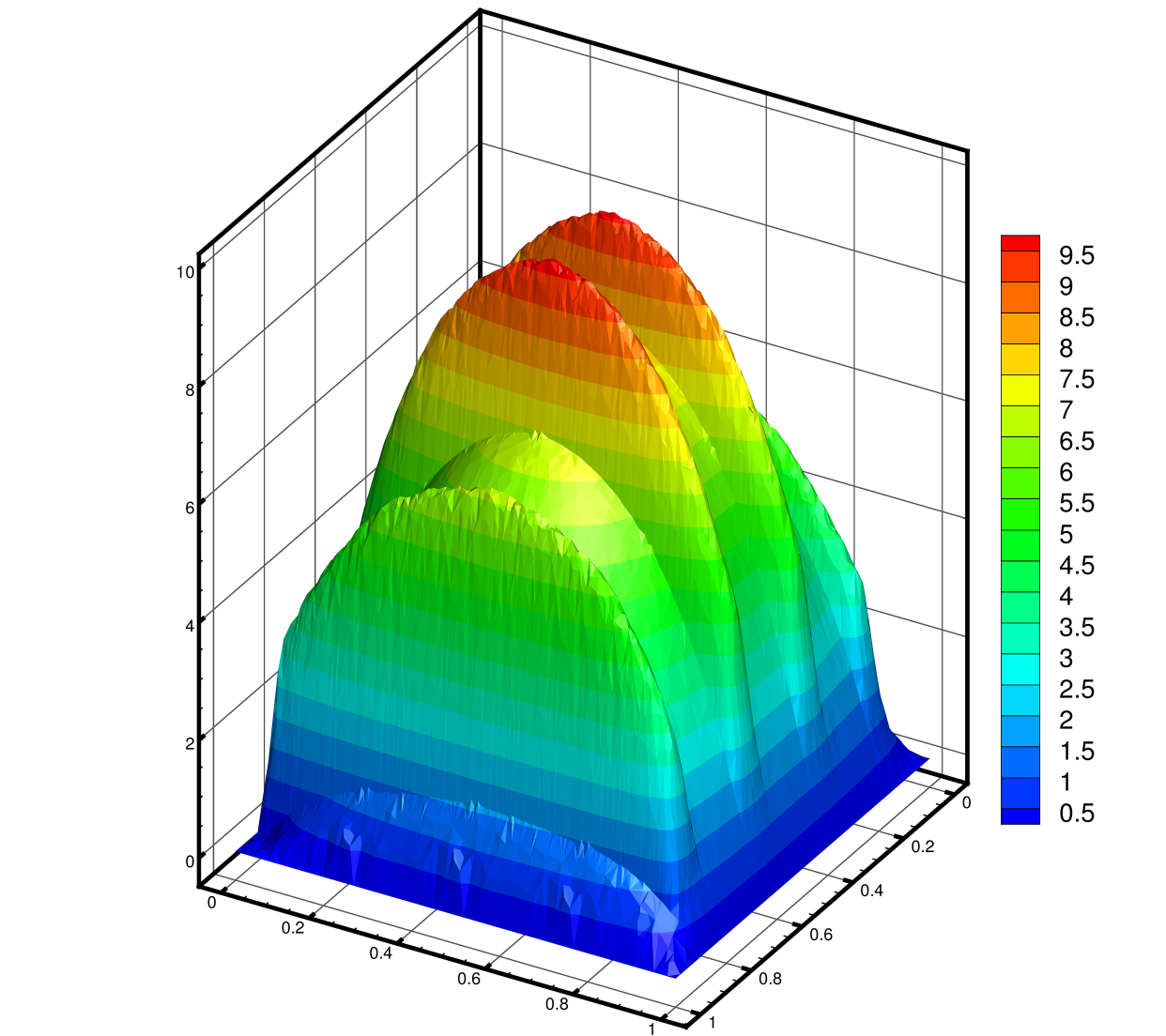} &
\includegraphics[width=0.17\textwidth]{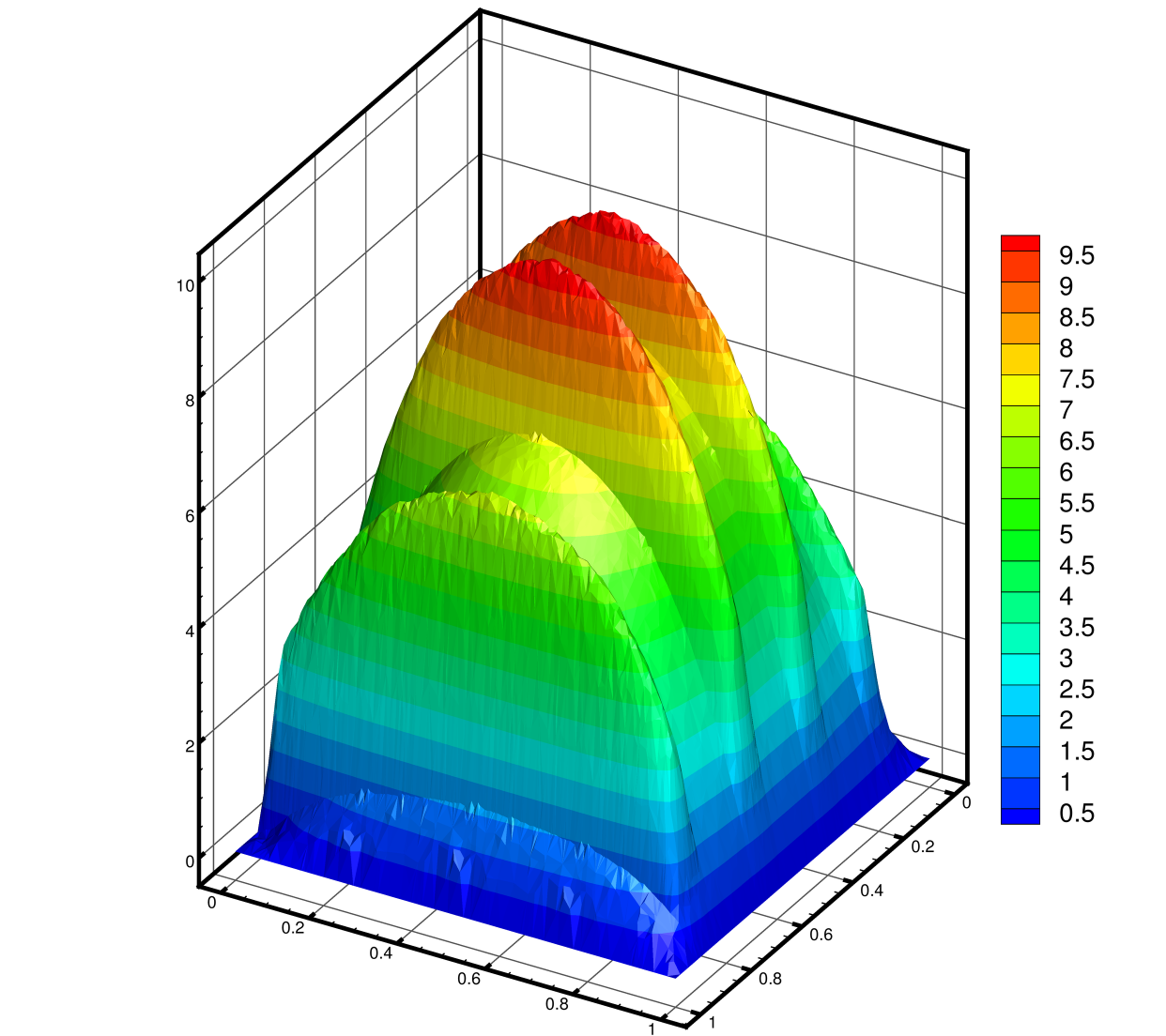}\\
(a) & (b) & (c) & (d) & (e)
\end{tabular}
\caption{Continuum 1 pressure fields for the three-dimensional channel-inclusion problem at $T=1$ (output section): (a) $u_1^\eps$; (b) $u_1^{[\mathrm{FP}]}(T)$; (c) $u_1^{(1\eps)}$; (d) $u_1^{(2\eps)}$; (e) $u_{1,r}^{\eps,n}$.}
\label{fig:channel3d-u1-fields}
\end{figure}

\begin{figure}[htbp]
\centering
\begin{tabular}{@{}ccccc@{}}
\includegraphics[width=0.17\textwidth]{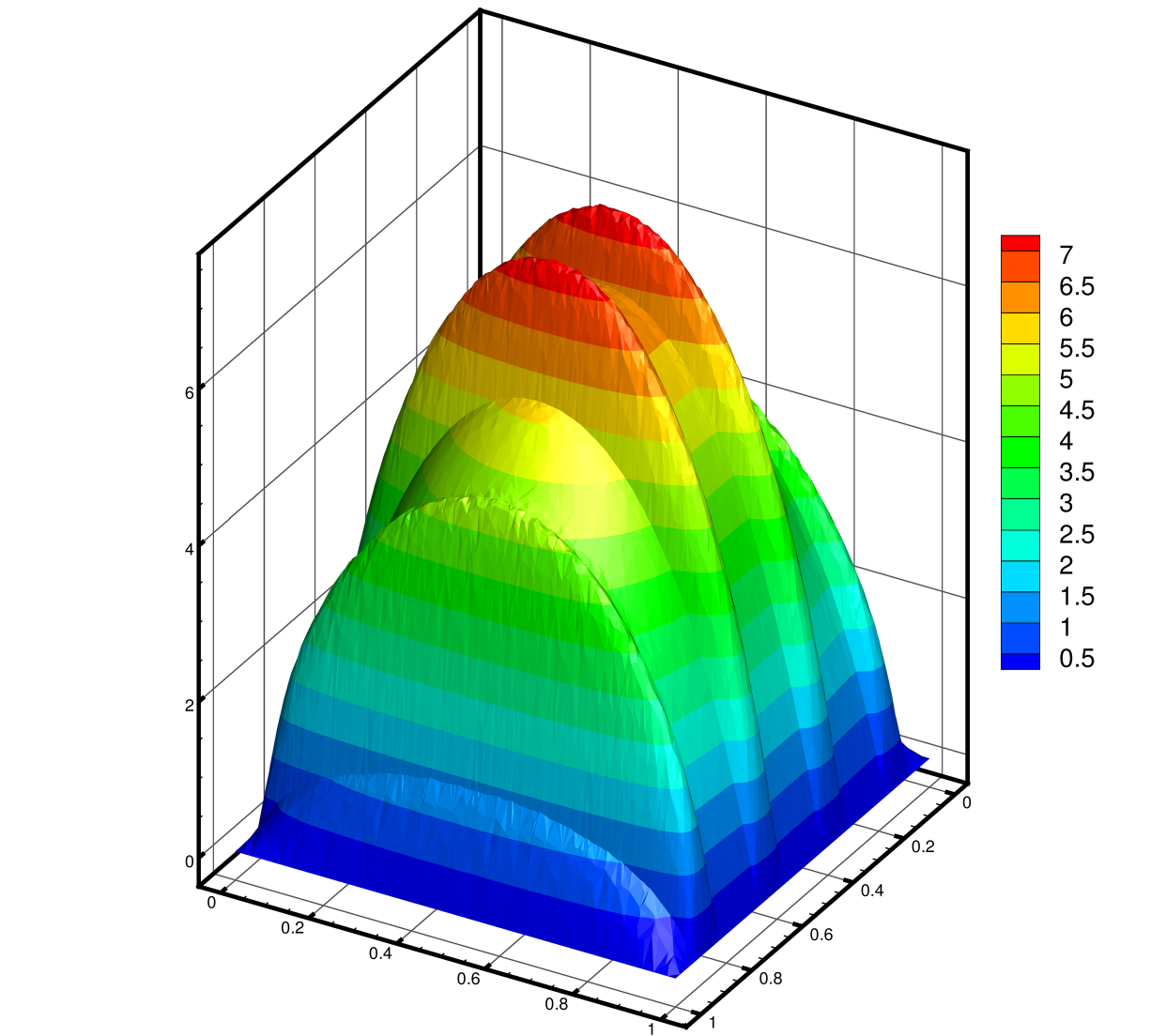} &
\includegraphics[width=0.17\textwidth]{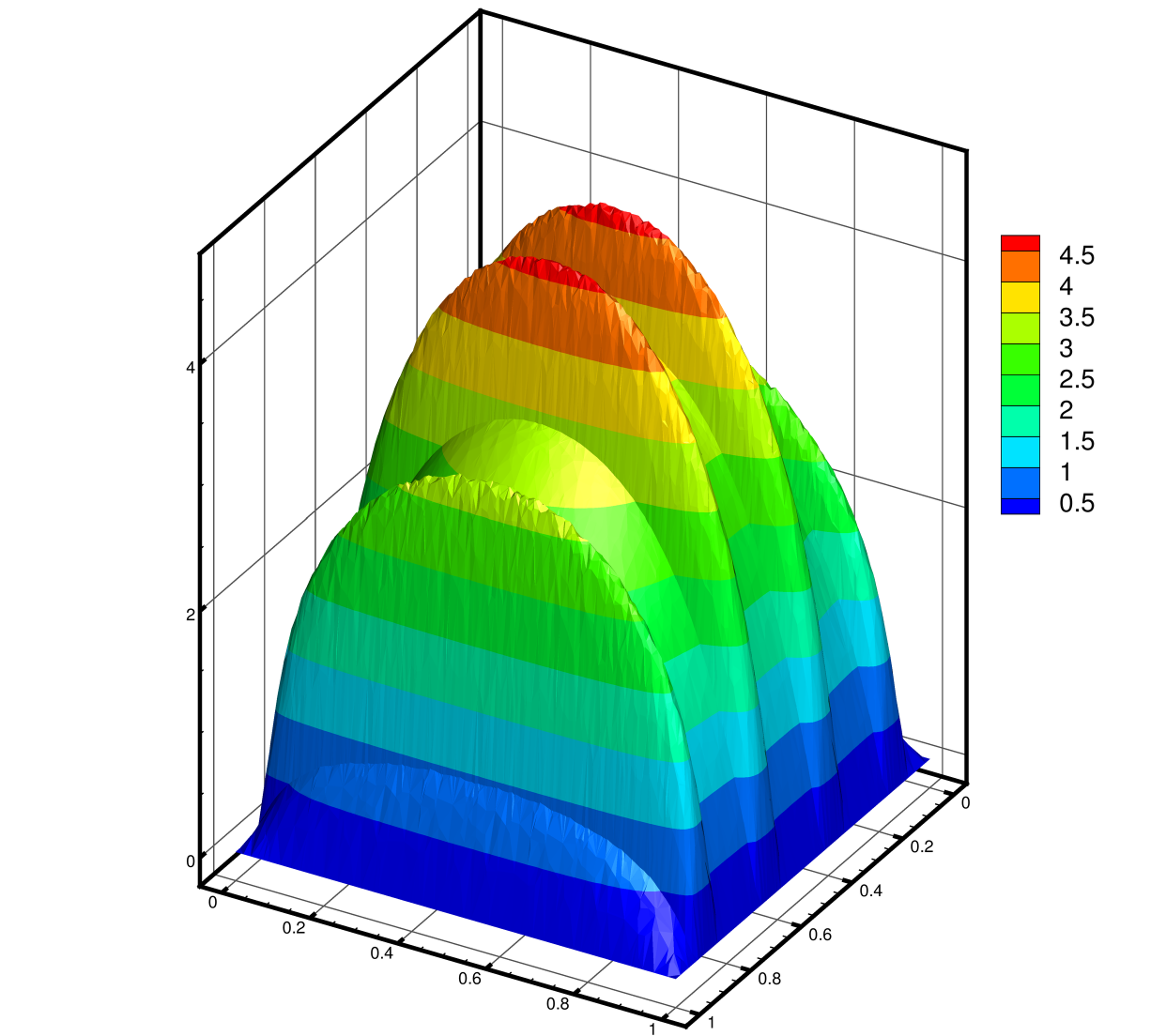} &
\includegraphics[width=0.17\textwidth]{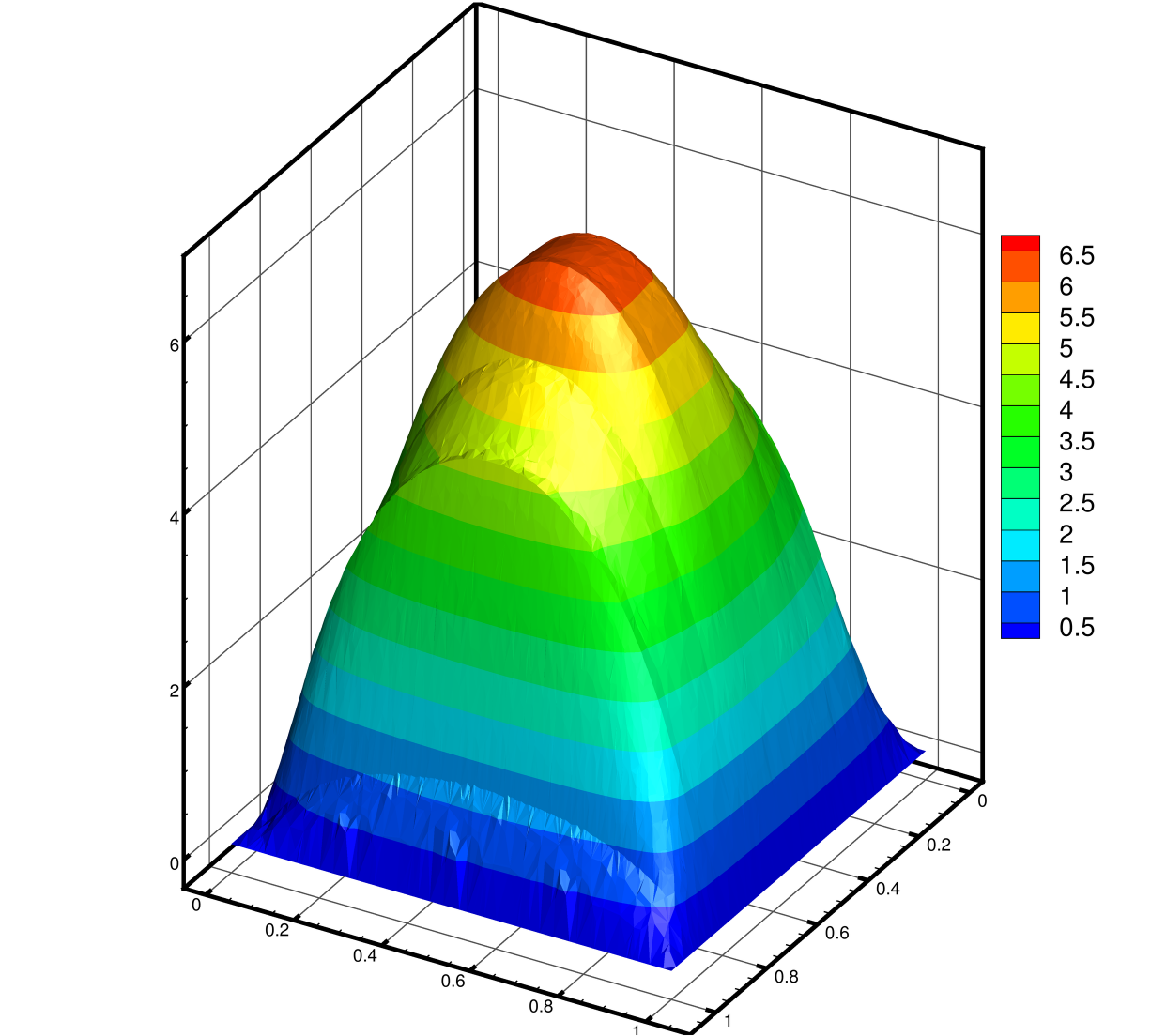} &
\includegraphics[width=0.17\textwidth]{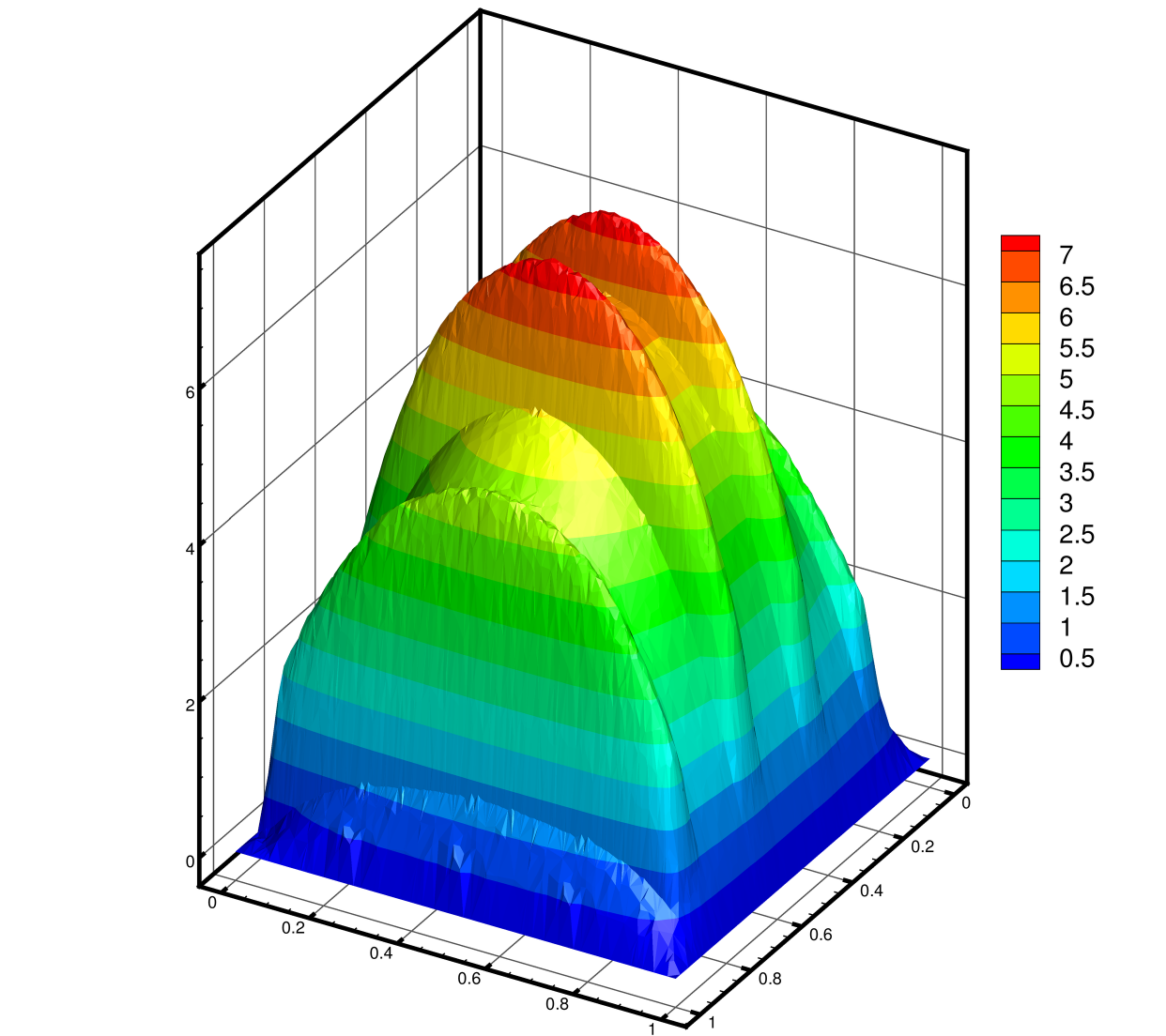} &
\includegraphics[width=0.17\textwidth]{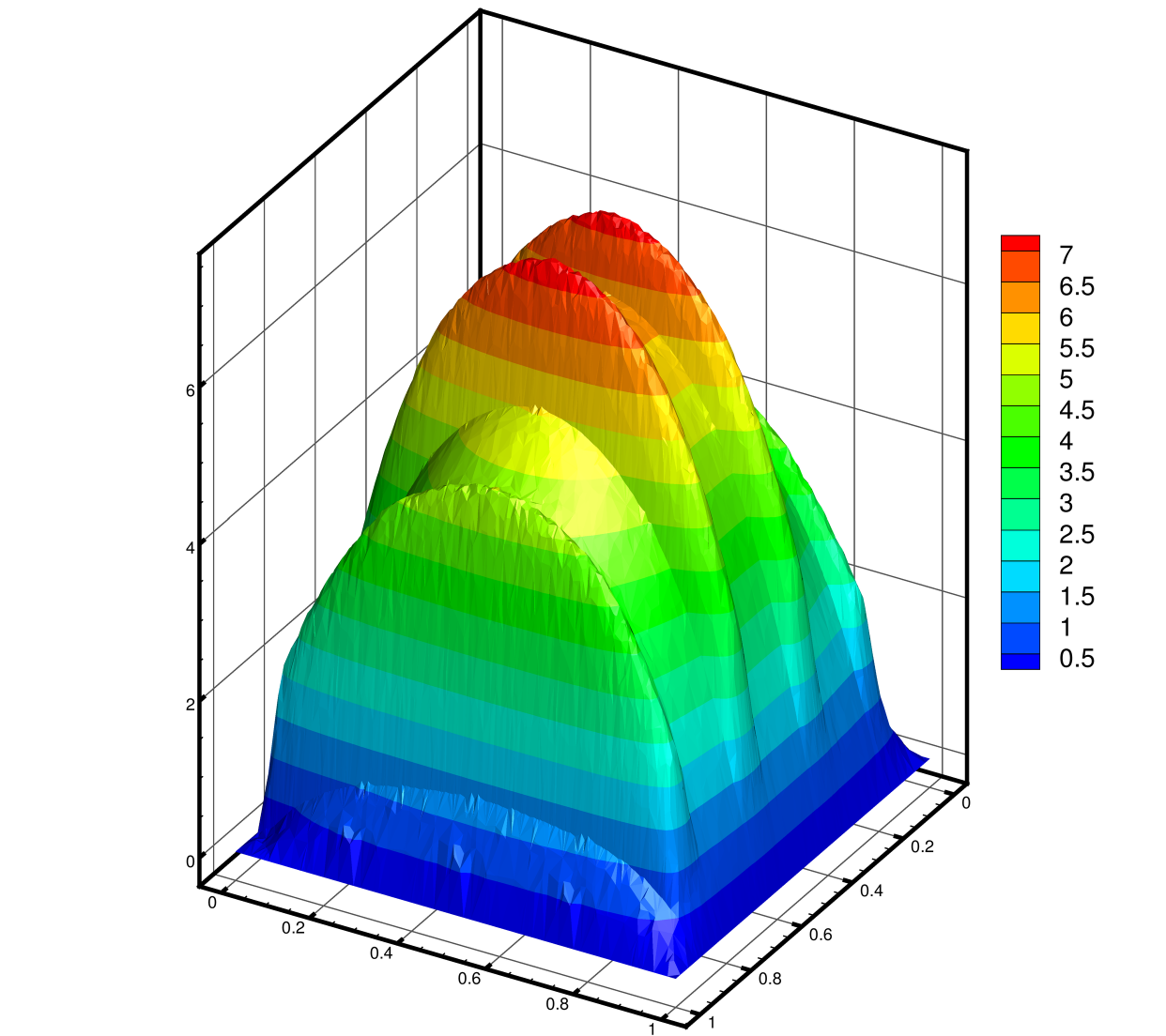}\\
(a) & (b) & (c) & (d) & (e)
\end{tabular}
\caption{Continuum 2 pressure fields for the three-dimensional channel-inclusion problem at $T=1$ (output section): (a) $u_2^\eps$; (b) $u_2^{[\mathrm{FP}]}(T)$; (c) $u_2^{(1\eps)}$; (d) $u_2^{(2\eps)}$; (e) $u_{2,r}^{\eps,n}$.}
\label{fig:channel3d-u2-fields}
\end{figure}

\begin{figure}[htbp]
\centering
\begin{tabular}{@{}ccccc@{}}
\includegraphics[width=0.17\textwidth]{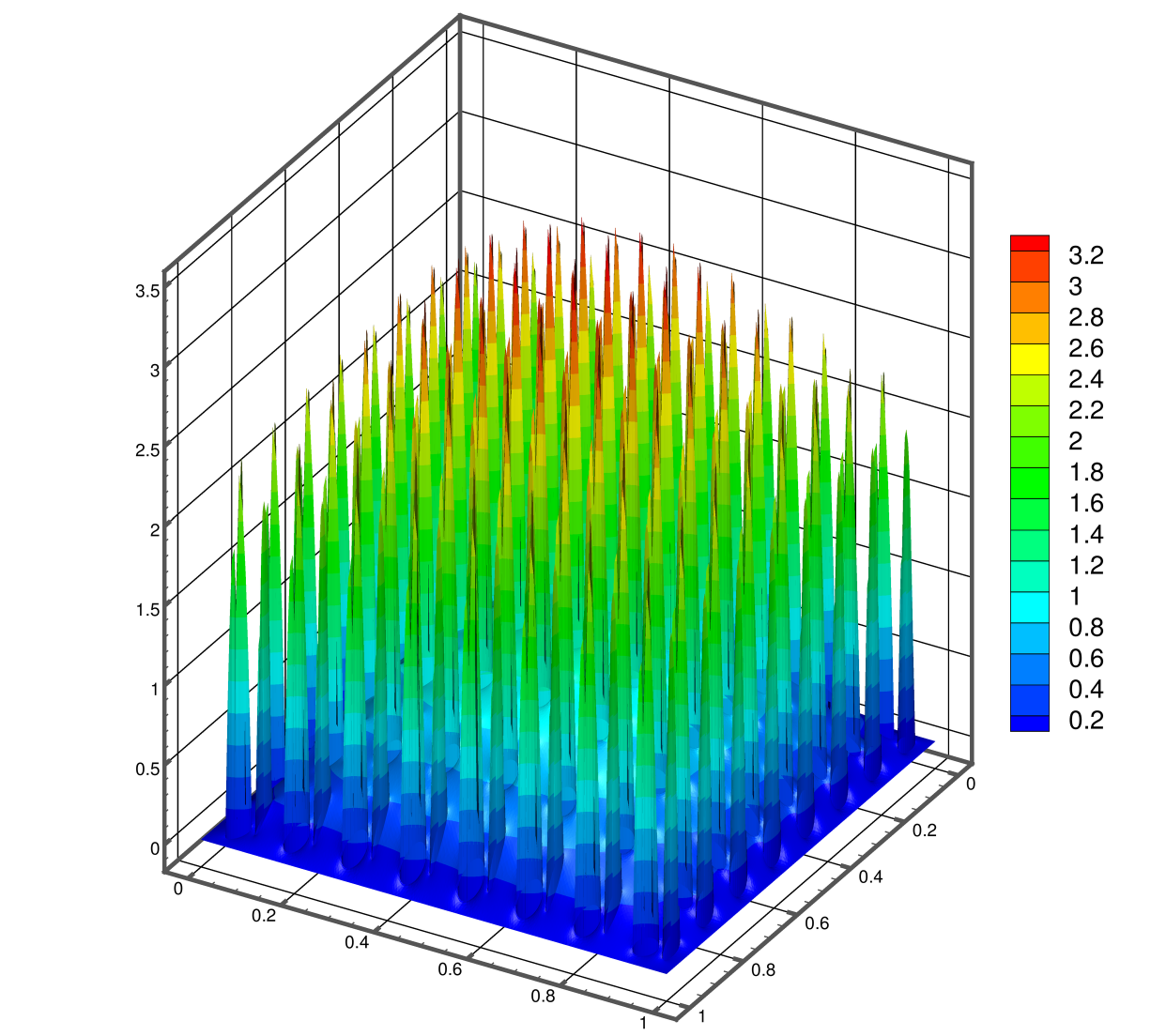} &
\includegraphics[width=0.17\textwidth]{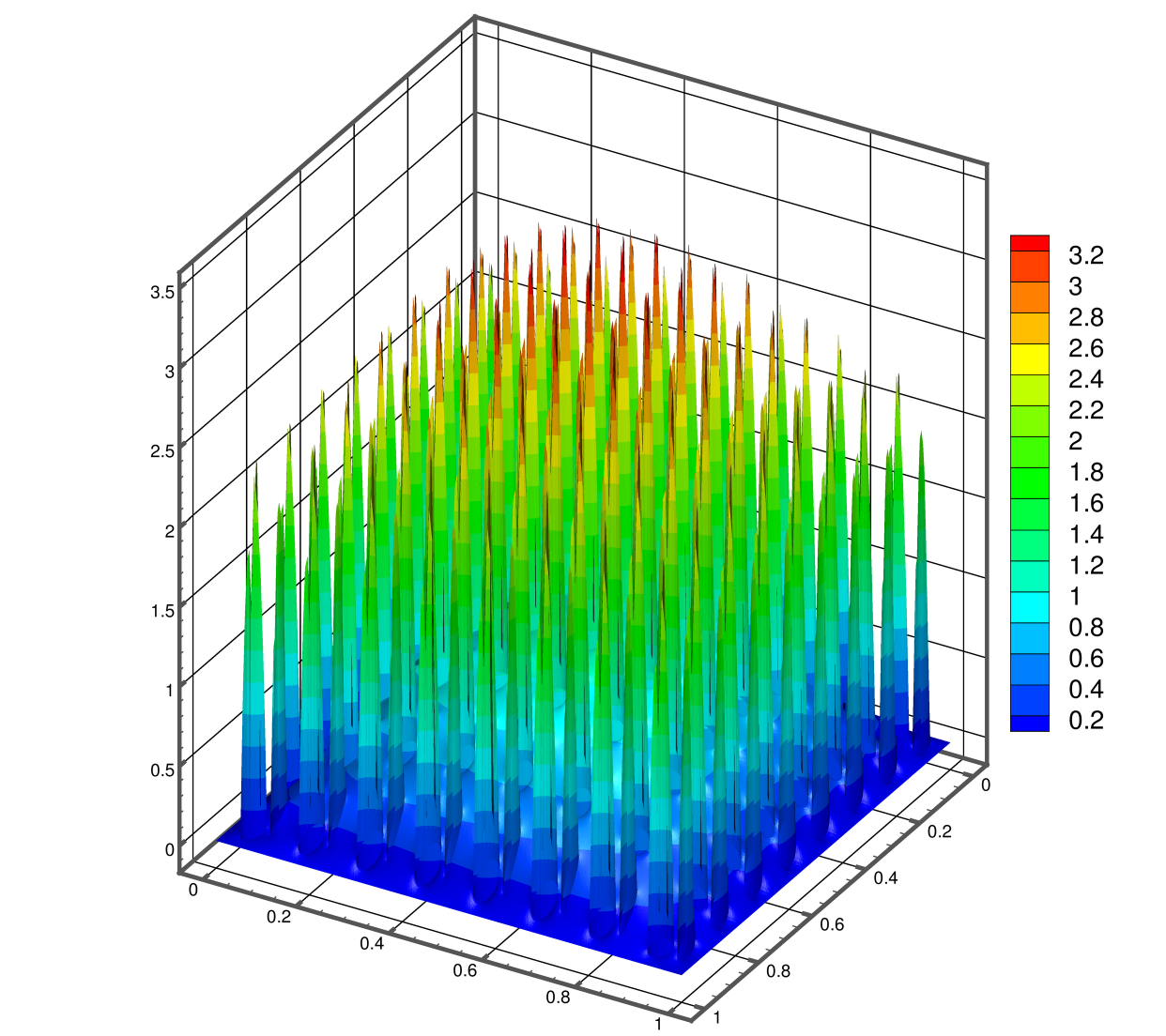} &
\includegraphics[width=0.17\textwidth]{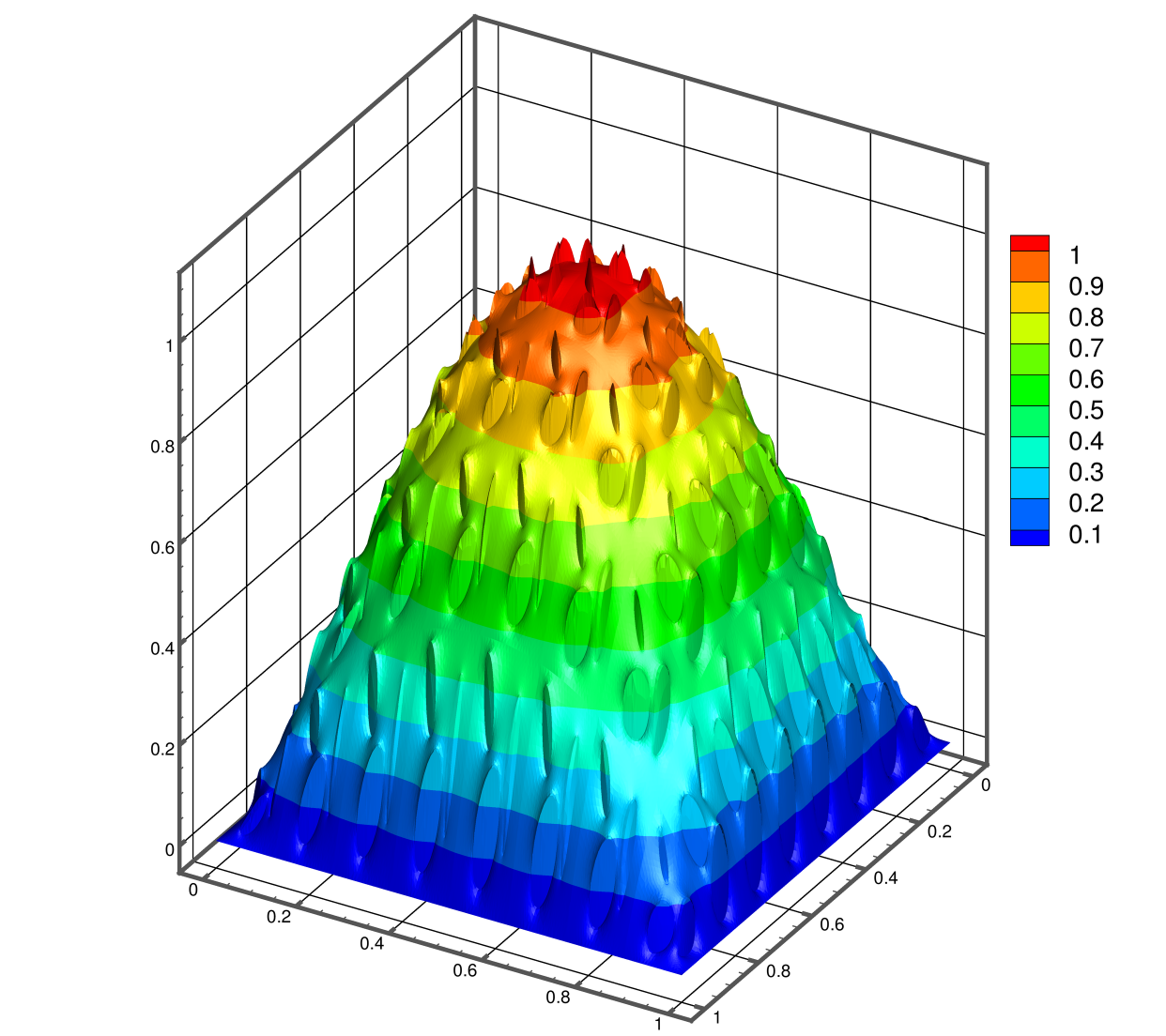} &
\includegraphics[width=0.17\textwidth]{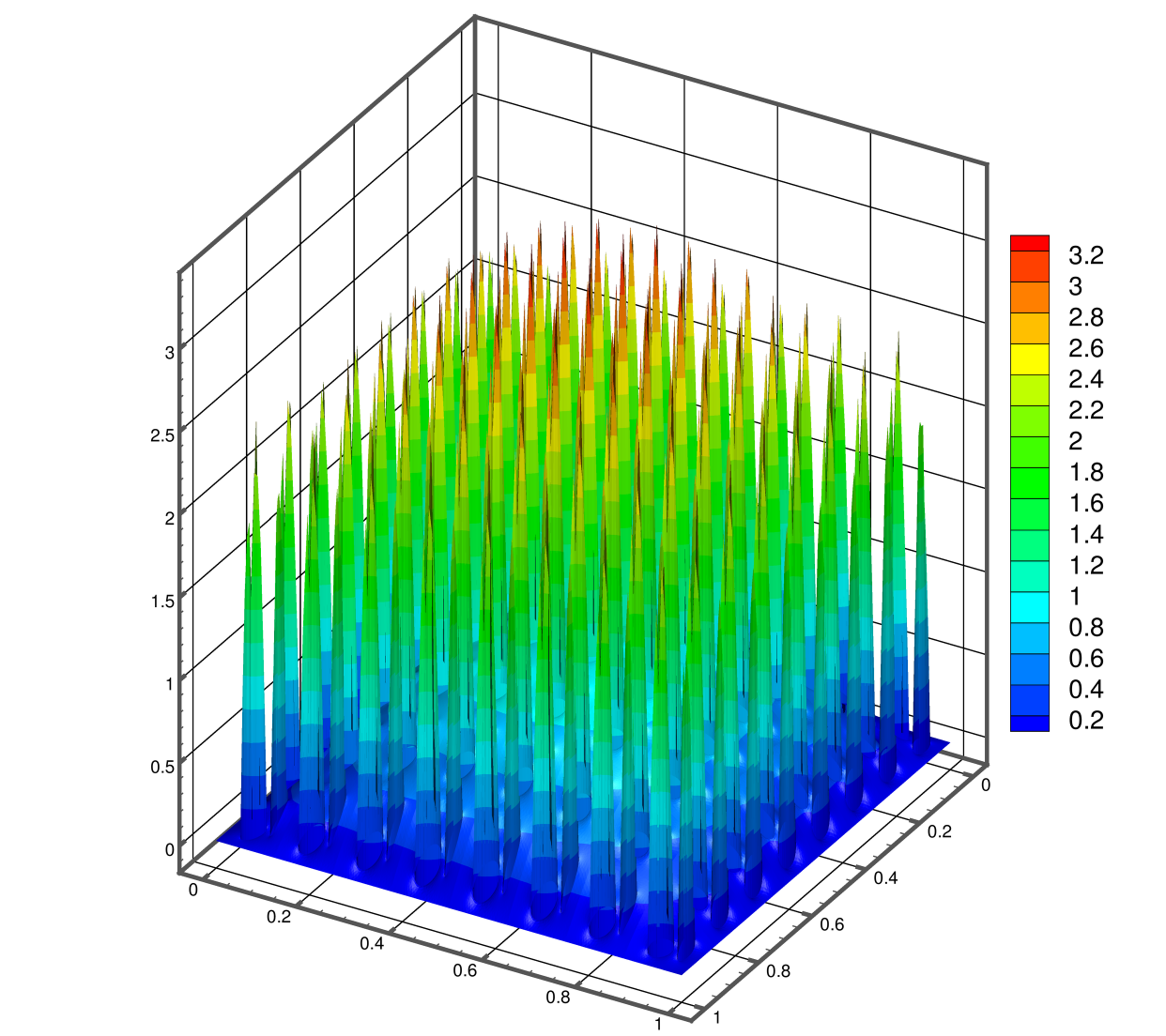} &
\includegraphics[width=0.17\textwidth]{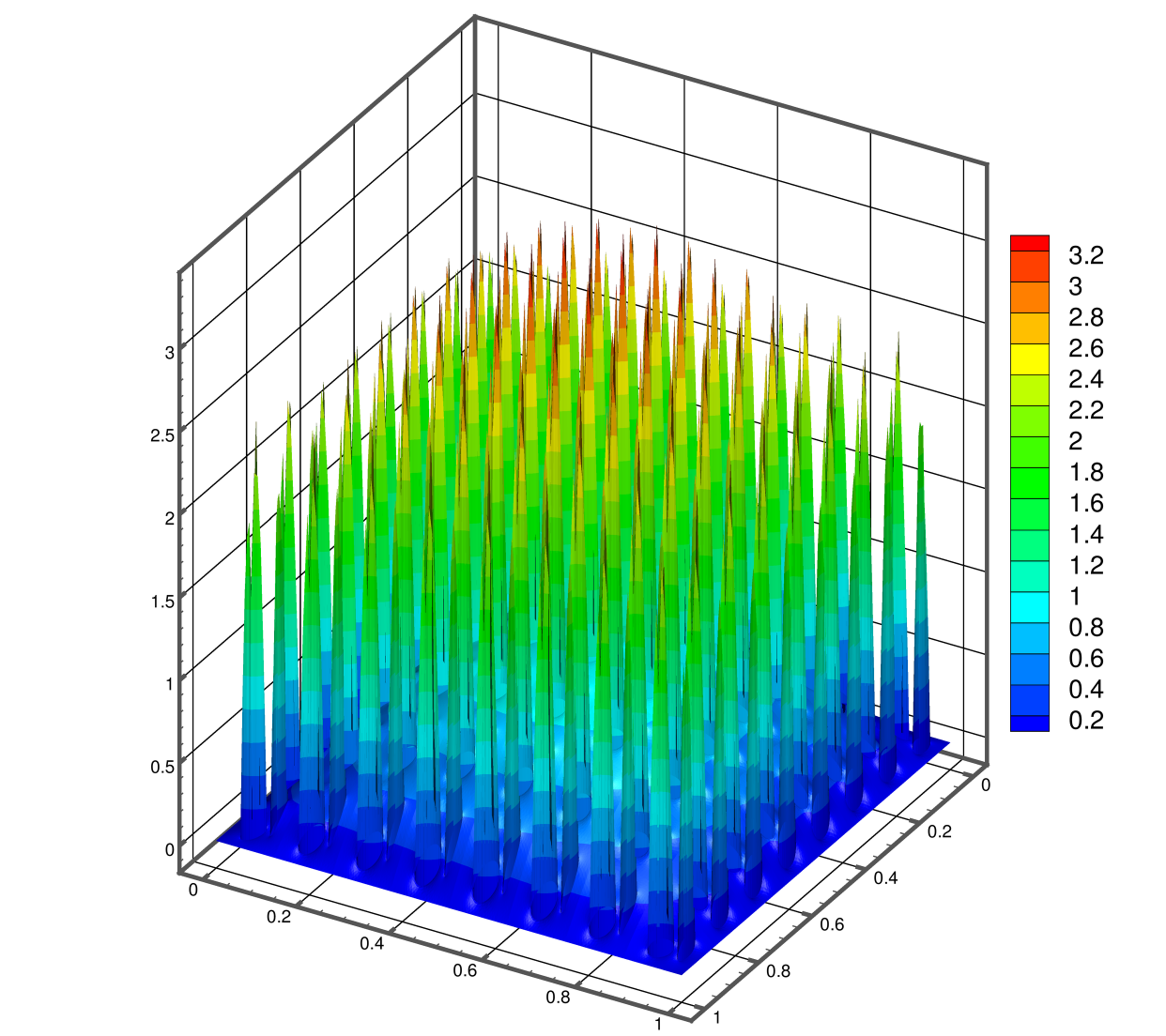}\\
(a) & (b) & (c) & (d) & (e)
\end{tabular}
\caption{Continuum 1 pressure fields for the two-dimensional three-ellipse problem at $T=1$: (a) $u_1^\eps$; (b) $u_1^{[\mathrm{FP}]}(T)$; (c) $u_1^{(1\eps)}$; (d) $u_1^{(2\eps)}$; (e) $u_{1,r}^{\eps,n}$.}
\label{fig:complex2d-u1-fields}
\end{figure}

\begin{figure}[htbp]
\centering
\begin{tabular}{@{}ccccc@{}}
\includegraphics[width=0.17\textwidth]{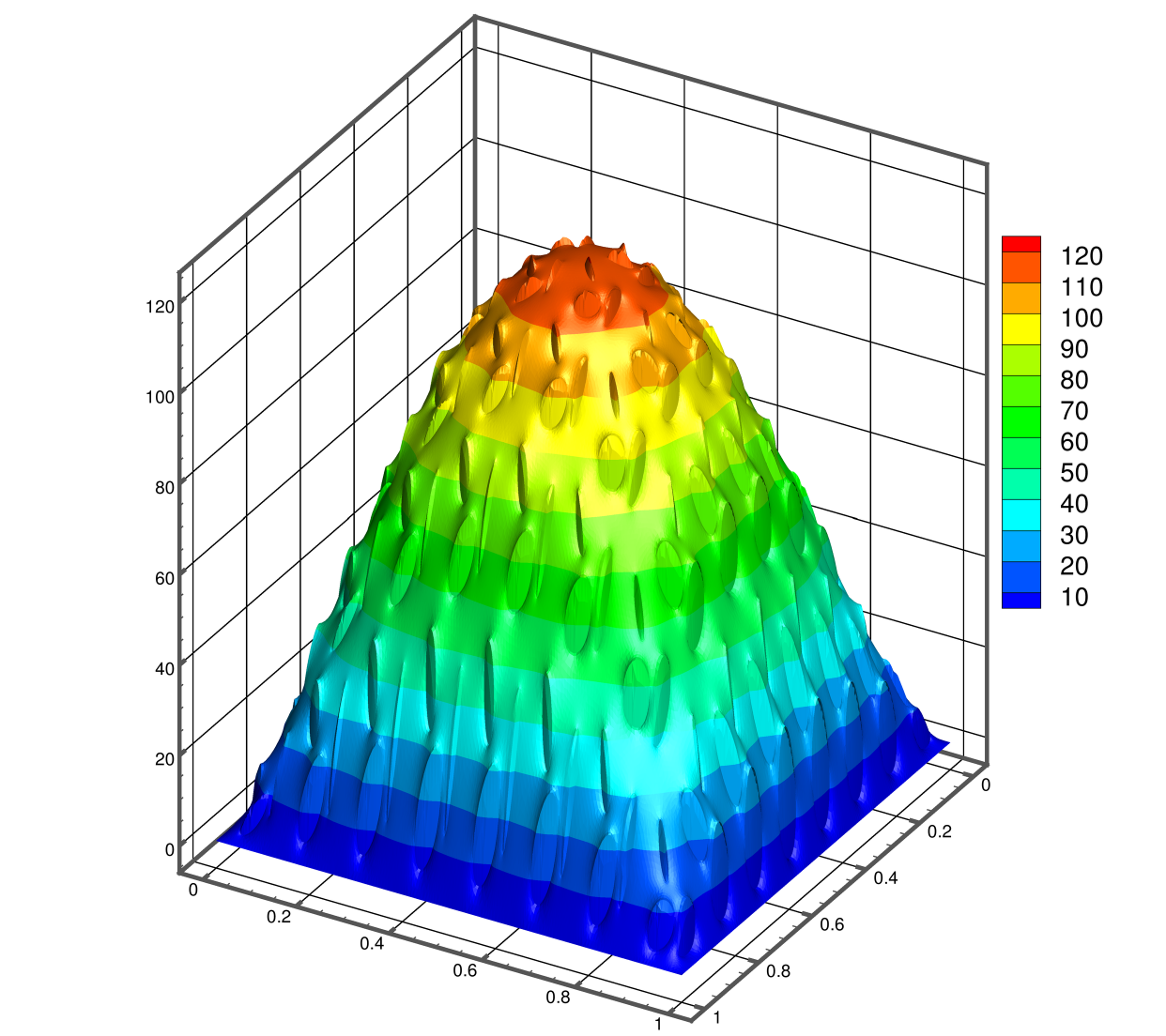} &
\includegraphics[width=0.17\textwidth]{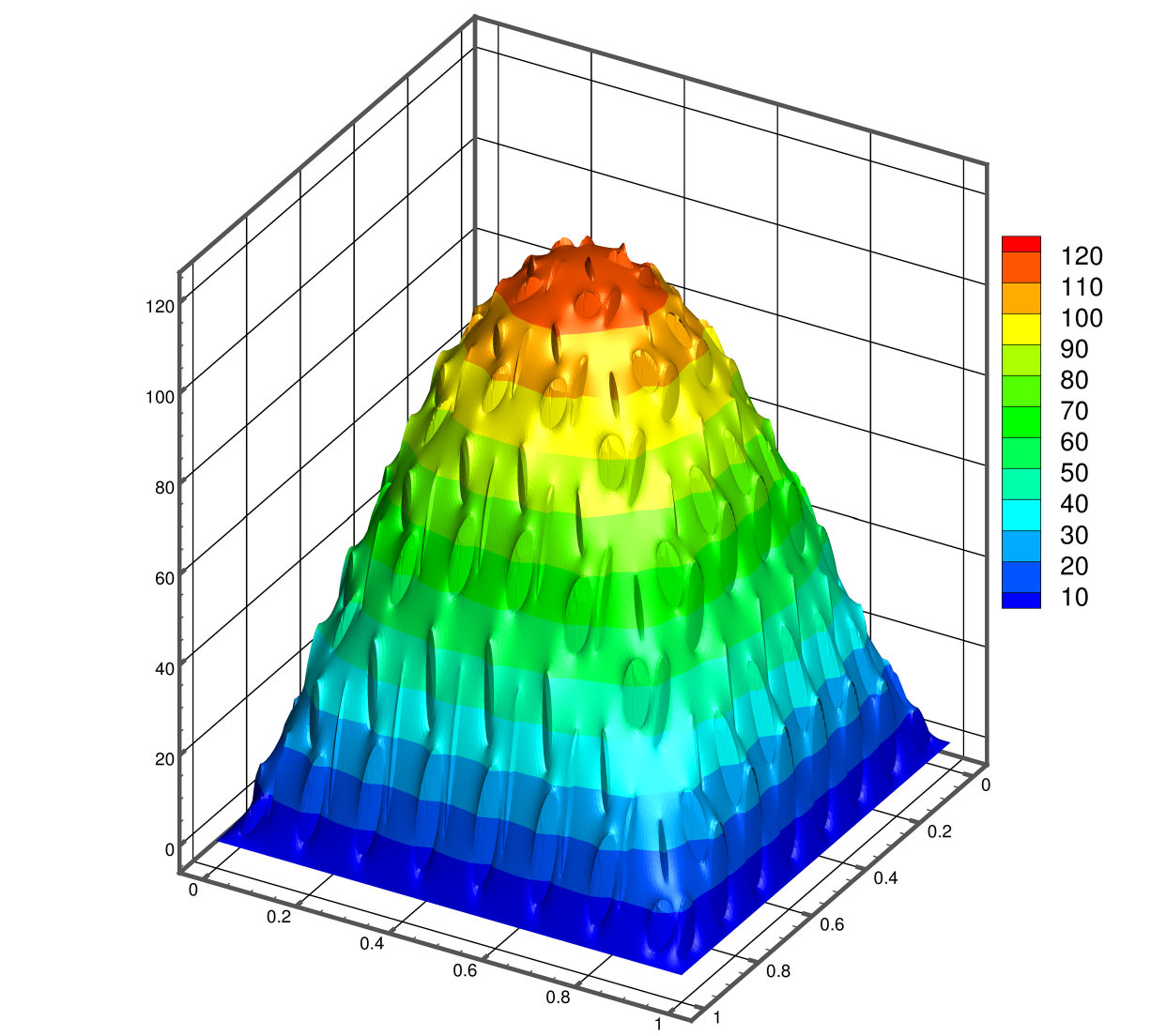} &
\includegraphics[width=0.17\textwidth]{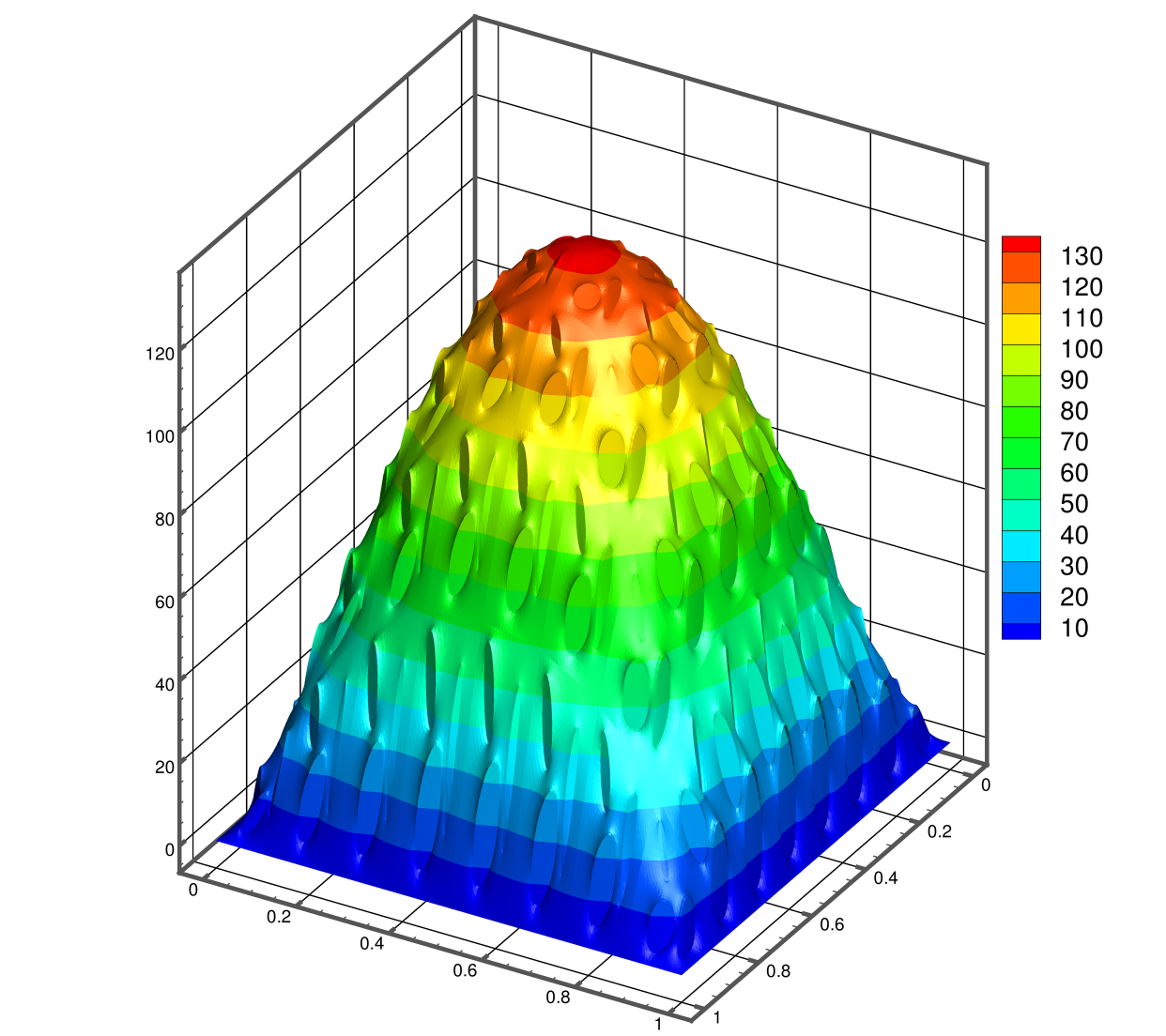} &
\includegraphics[width=0.17\textwidth]{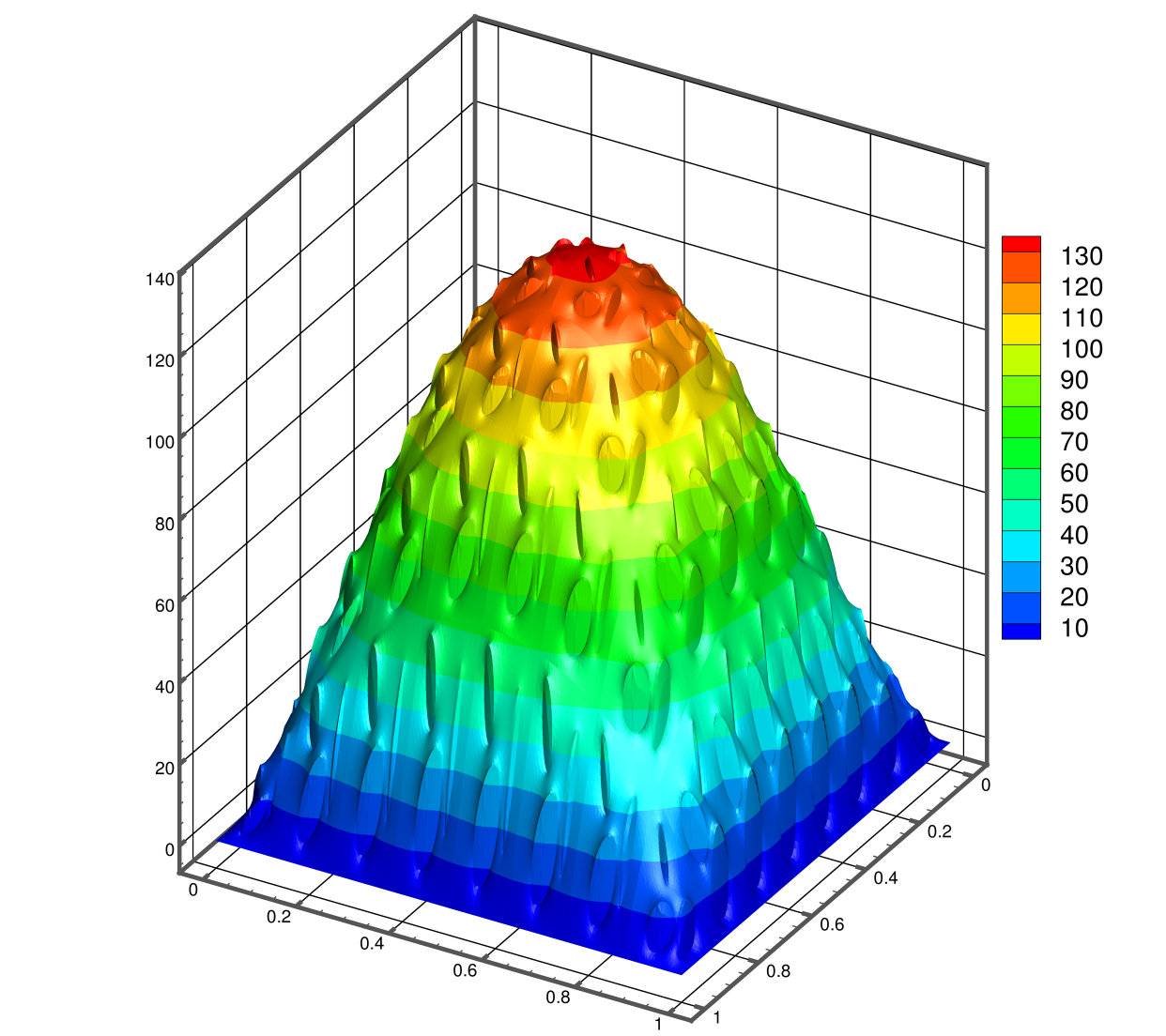} &
\includegraphics[width=0.17\textwidth]{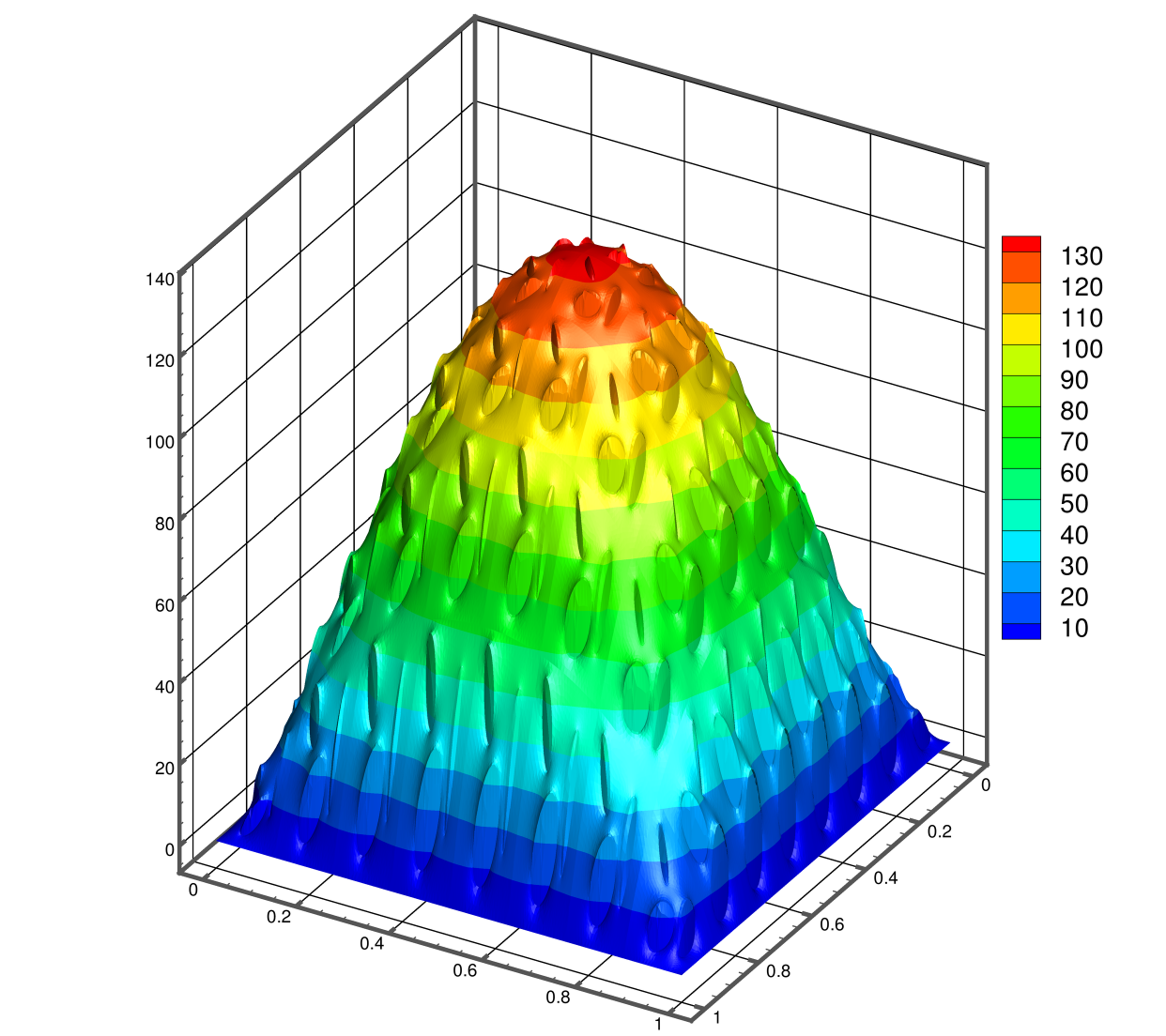}\\
(a) & (b) & (c) & (d) & (e)
\end{tabular}
\caption{Continuum 2 pressure fields for the two-dimensional three-ellipse problem at $T=1$: (a) $u_2^\eps$; (b) $u_2^{[\mathrm{FP}]}(T)$; (c) $u_2^{(1\eps)}$; (d) $u_2^{(2\eps)}$; (e) $u_{2,r}^{\eps,n}$.}
\label{fig:complex2d-u2-fields}
\end{figure}

\clearpage
The HOMS-POD fields in Figures~\ref{fig:circle2d-u1-fields}--\ref{fig:channel2d-u2-fields} retain the inclusion-localized fluctuations of both two-dimensional continuum components for the circular and channel microstructures. Figures~\ref{fig:sphere3d-u1-fields} and~\ref{fig:sphere3d-u2-fields} preserve the localized oscillations generated by the spherical inclusions, while Figures~\ref{fig:channel3d-u1-fields} and~\ref{fig:channel3d-u2-fields} preserve the directional channel pattern for both three-dimensional components. For the three-ellipse medium, Figures~\ref{fig:complex2d-u1-fields} and~\ref{fig:complex2d-u2-fields} show the strongly localized pressure features in continuum~1 and the different amplitude and spatial response of continuum~2. The POD-HOMS fields retain the envelope and inclusion-scale structures of the second-order HOMS solutions for both continua. Taken together, the sensitivity, error-history, and field comparisons show that the joint POD space captures the coupled transient behavior with a small training ensemble and without visible loss of the HOMS oscillatory structure. Because FEM-POD, HOMS, and HOMS-POD are all evaluated against the same fine-scale FEM solution, their accuracy can be compared directly under the two relative-error measures defined above.

\subsection{Computational efficiency of the POD-HOMS algorithms}
\label{subsec:computational-efficiency}

This subsection separates the two cost reductions in the method. POD lowers the cost of advancing the coupled homogenized trajectory, while the corrected-basis formulation removes repeated pointwise evaluation of the value, exchange, first-derivative, second-derivative, and time-derivative recovery terms. Tables~\ref{tab:total-time} and~\ref{tab:assembly-time} therefore report both end-to-end time and fine-scale output-assembly time. The FEM, FEM-POD, HOMS, and Algorithm~2 entries use the low-snapshot assembly run for each geometry; the Algorithm~1 entries use the matching low-snapshot noassembly run. Speedups are computed from the unrounded timing records.

\begin{table}[htbp]
\centering
\caption{End-to-end CPU time in seconds for the standard tests.}
\label{tab:total-time}
\resizebox{\textwidth}{!}{%
\begin{tabular}{llrrrrrr}
\toprule
Dimension & Geometry & FEM & FEM-POD & HOMS & HOMS-POD A1 & HOMS-POD A2 & HOMS/A2\\
\midrule
2D & circle  & 987.654 & 107.534 & 385.986 & 133.069 & 38.562 & 10.01\\
2D & channel & 317.487 & 36.255 & 240.009 & 85.561 & 24.638 & 9.74\\
2D & ellipses & 1266.737 & 208.853 & 373.077 & 148.216 & 63.232 & 5.90\\
3D & sphere  & 9800.250 & 250.874 & 2966.640 & 2599.200 & 99.561 & 29.80\\
3D & channel & 31429.800 & 712.753 & 10845.700 & 7312.000 & 176.372 & 61.49\\
\bottomrule
\end{tabular}}
\end{table}

The end-to-end comparison includes snapshot generation, POD construction, reduced operators, time integration, and recovery. Relative to full HOMS, Algorithm~2 gives total speedups from $5.90$ in the two-dimensional three-ellipse case to $61.49$ in the three-dimensional channel case. For the three-ellipse medium, the recorded FEM, FEM-POD, and HOMS times are $1266.737$, $208.853$, and $373.077$ seconds, respectively, while Algorithms~1 and~2 require $148.216$ and $63.232$ seconds. These method-time counters exclude mesh generation, error evaluation, and file output. FEM-POD can also be fast in the standard tests, but it requires fine-scale training solves and therefore does not retain this feasibility advantage when the fine grid exceeds memory.

\begin{table}[htbp]
\centering
\caption{Fine-scale output-assembly time and speedup.}
\label{tab:assembly-time}
\begin{tabular}{llrrrr}
\toprule
Dimension & Geometry & A1 (s) & A2 (s) & A1/A2 assembly & A1/A2 total\\
\midrule
2D & circle  & 95.933 & 0.805 & 119.1 & 3.45\\
2D & channel & 60.673 & 0.566 & 107.3 & 3.47\\
2D & ellipses & 83.354 & 0.898 & 92.8 & 2.34\\
3D & sphere  & 2427.580 & 3.031 & 800.9 & 26.11\\
3D & channel & 7168.070 & 9.933 & 721.7 & 41.46\\
\bottomrule
\end{tabular}
\end{table}

Algorithm~1 reconstructs the reduced homogenized solution and then evaluates the full HOMS expression at every output point and time level. Algorithm~2 shifts the spatial corrections of the retained modes to its offline stage, leaving only reduced-coefficient combinations and the separate derivative-basis contribution online. The resulting three-dimensional assembly speedups of $800.9$ and $721.7$ confirm the practical benefit of replacing repeated pointwise recovery by the stored-basis combinations described in Section~\ref{subsec:algorithm2}. For the three-ellipse case, corrected-basis gradient and recovery precomputation cost $0.214$ and $1.406$ seconds, respectively; the output-assembly time decreases from $83.354$ to $0.898$ seconds, giving a $92.8$-fold assembly speedup and a $2.34$-fold total A1/A2 speedup. The corresponding A1/A2 total speedups are smaller because both algorithms retain cell solves, snapshot generation, POD construction, and reduced-system assembly. Thus POD and corrected bases address different bottlenecks: POD reduces the transient solve cost, whereas corrected bases remove repeated high-order recovery cost.

\subsection{Large-scale three-dimensional application}
\label{subsec:large-scale}

The final example examines the feasibility of the proposed method for a three-dimensional problem beyond the memory limit of direct fine-scale approaches. The periodically repeated spherical-inclusion configuration in Figure~\ref{fig:large-geometry} uses the same material parameters, source, and time discretization as the preceding three-dimensional spherical test. With $T=1$ and $\dt=0.002$, $500$ time steps are performed and ten coupled homogenized snapshots are used to construct the joint POD space.

\begin{figure}[htbp]
\centering
\includegraphics[width=0.55\textwidth]{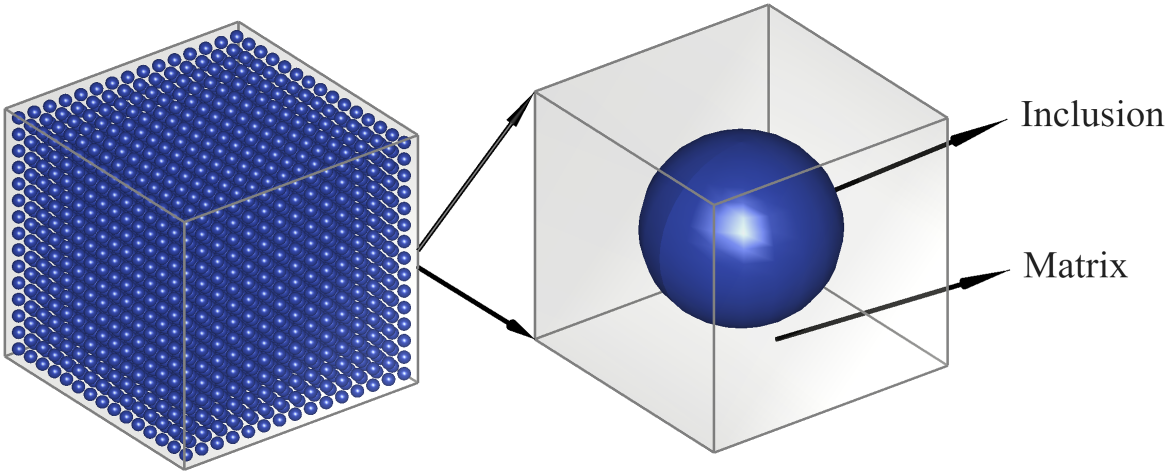}
\caption{Periodically heterogeneous geometry and representative cell for the large-scale test.}
\label{fig:large-geometry}
\end{figure}

The cell, homogenized, and fine output meshes are used for the local corrector problems, macroscopic evolution, and multiscale reconstruction, respectively. As summarized in Table~\ref{tab:large-scale}, the fine output mesh contains $8\,921\,919$ elements, whereas the homogenized mesh contains $93\,750$ elements. Direct FEM and fine-scale FEM-POD exceed the available 32 GB memory because they require fine-scale evolution or fine-scale snapshot generation, while HOMS and HOMS-POD remain feasible.

\begin{table}[htbp]
\centering
\caption{Discretization sizes and direct-method availability for the large-scale three-dimensional test.}
\label{tab:large-scale}
\begin{tabular}{lr}
\toprule
Quantity & Value\\
\midrule
Fine-scale output elements & $8\,921\,919$\\
Cell elements & $145\,815$\\
Homogenized elements & $93\,750$\\
Time steps & $500$\\
Coupled homogenized snapshots & $10$\\
Direct FEM / fine-scale FEM-POD & Exceeds 32 GB\\
\bottomrule
\end{tabular}
\end{table}

Table~\ref{tab:large-scale-timing} reports the end-to-end timing decomposition, including POD training and corrected-basis preprocessing.

\begin{table}[htbp]
\centering
\caption{CPU-time decomposition for the large-scale three-dimensional test (seconds).}
\label{tab:large-scale-timing}
\begin{tabular}{lrr}
\toprule
Computational stage & HOMS & HOMS-POD (Algorithm~2)\\
\midrule
Cell problems & $133.310$ & $133.310$\\
Full homogenized trajectory & $5538.160$ & --\\
Coupled snapshot generation & -- & $115.066$\\
POD mode extraction & -- & $0.353$\\
Reduced-matrix construction & -- & $26.060$\\
Reduced time integration & -- & $0.009$\\
Corrected-basis preprocessing & -- & $7.909$\\
Fine-scale output assembly & $2497.610$ & $22.233$\\
\midrule
Total & $8169.080$ & $304.940$\\
\bottomrule
\end{tabular}
\end{table}

The homogenized POD stage, including snapshot generation, POD construction, reduced-matrix assembly, and reduced time integration, requires $141.488$ s, compared with $5538.160$ s for the full homogenized trajectory. Corrected-basis reconstruction reduces the fine-scale output-assembly time from $2497.610$ s to $22.233$ s. Consequently, the total CPU time decreases from $8169.080$ s for HOMS to $304.940$ s for HOMS-POD, corresponding to an overall speedup of approximately $26.8$.

Figures~\ref{fig:large-u1-fields} and~\ref{fig:large-u2-fields} compare the HOMS and HOMS-POD fields on the central output section at $T=1$. The HOMS-POD solutions reproduce the macroscopic envelopes and periodic microscopic oscillations of the corresponding HOMS fields for both continua, without visible smoothing or phase shift. Since the direct fine-scale FEM reference exceeds the available memory, FEM-referenced errors are not reported for this test.

\begin{figure}[htbp]
\centering
\begin{tabular}{@{}cc@{}}
\includegraphics[width=0.32\textwidth]{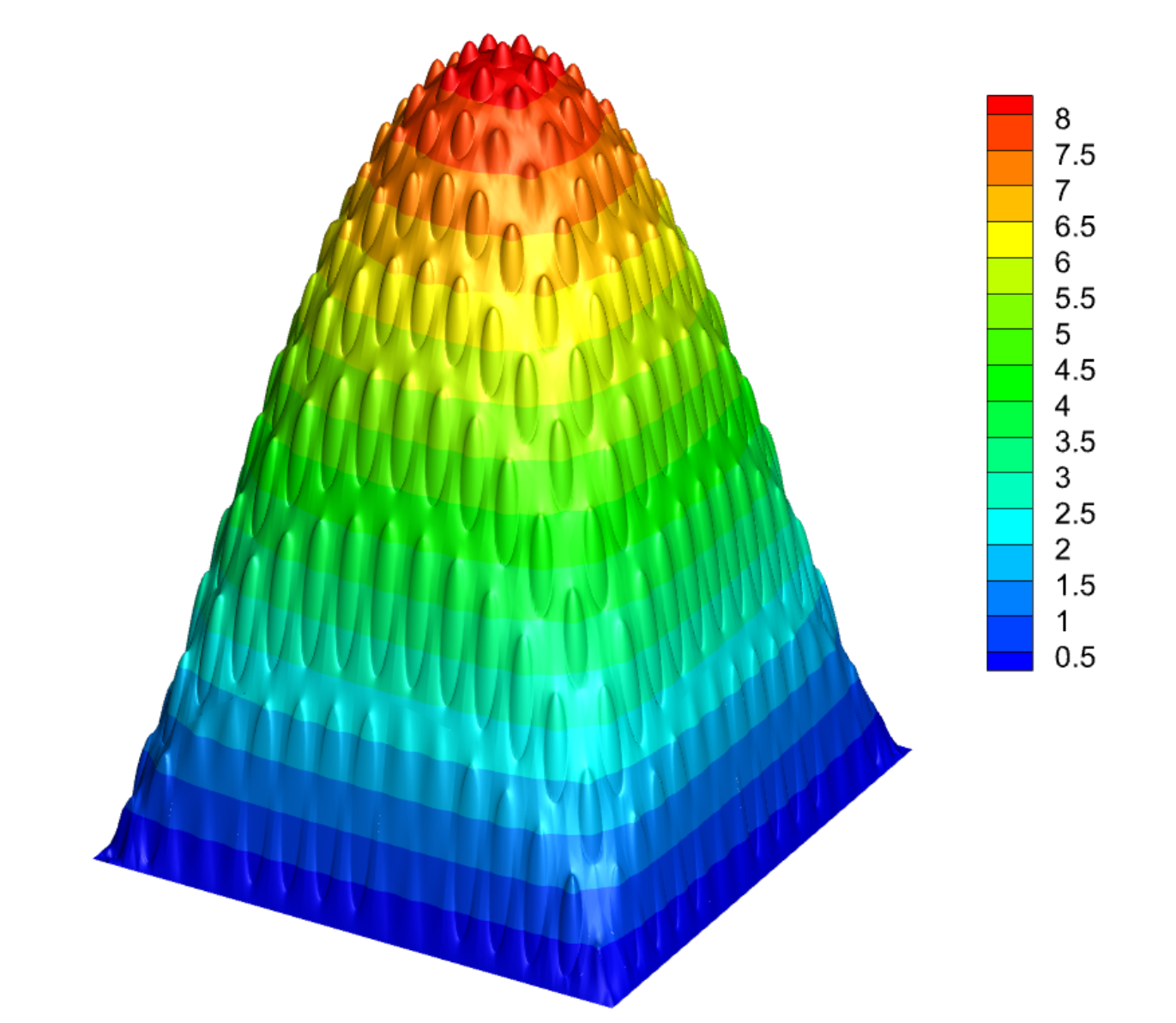} &
\includegraphics[width=0.32\textwidth]{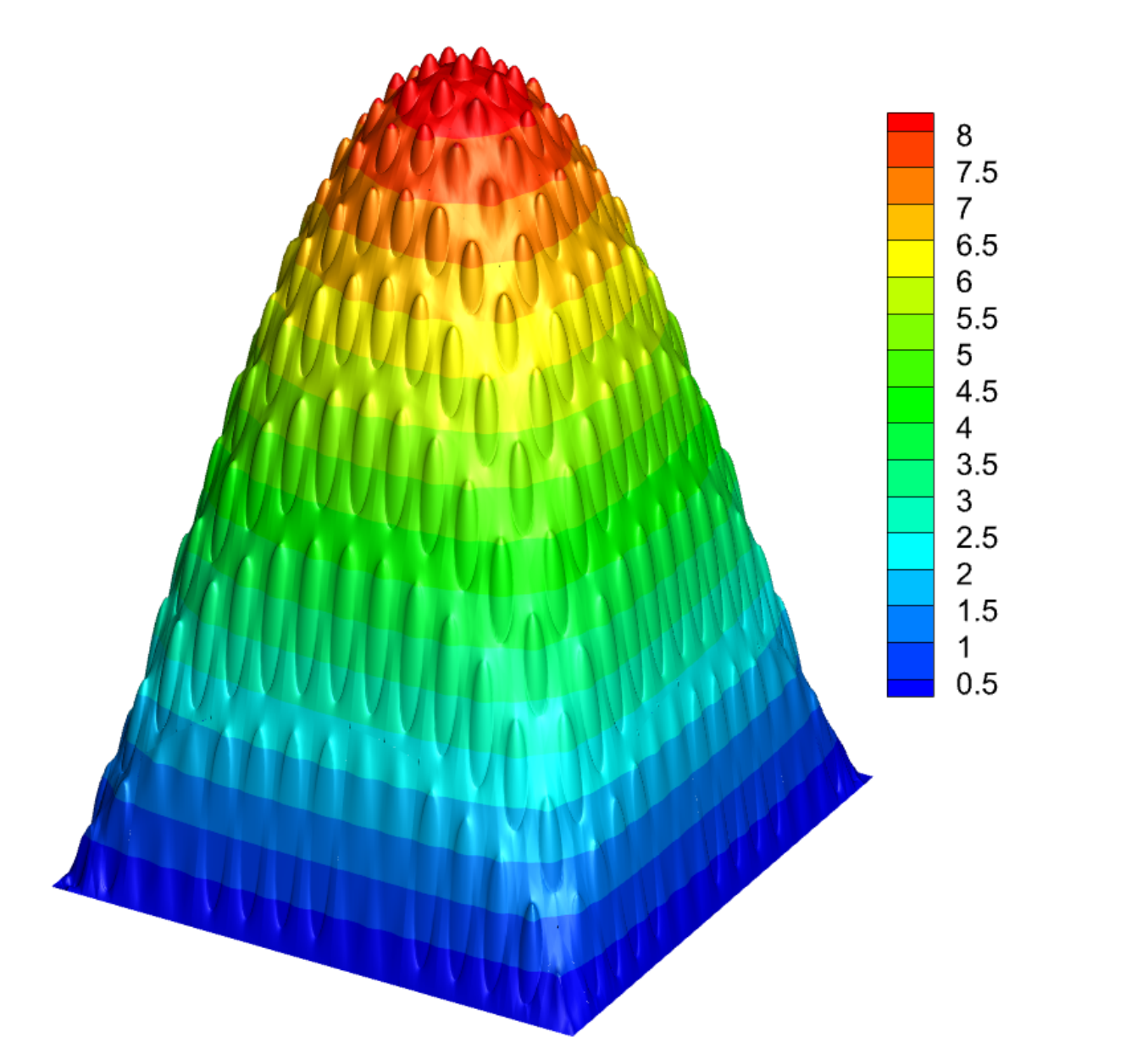}\\
(a) & (b)
\end{tabular}
\caption{Continuum 1 pressure fields for the large-scale three-dimensional spherical-inclusion problem at $T=1$ (central output section): (a) $u_1^{(2\eps)}$; (b) $u_{1,r}^{\eps,n}$.}
\label{fig:large-u1-fields}
\end{figure}

\begin{figure}[htbp]
\centering
\begin{tabular}{@{}cc@{}}
\includegraphics[width=0.32\textwidth]{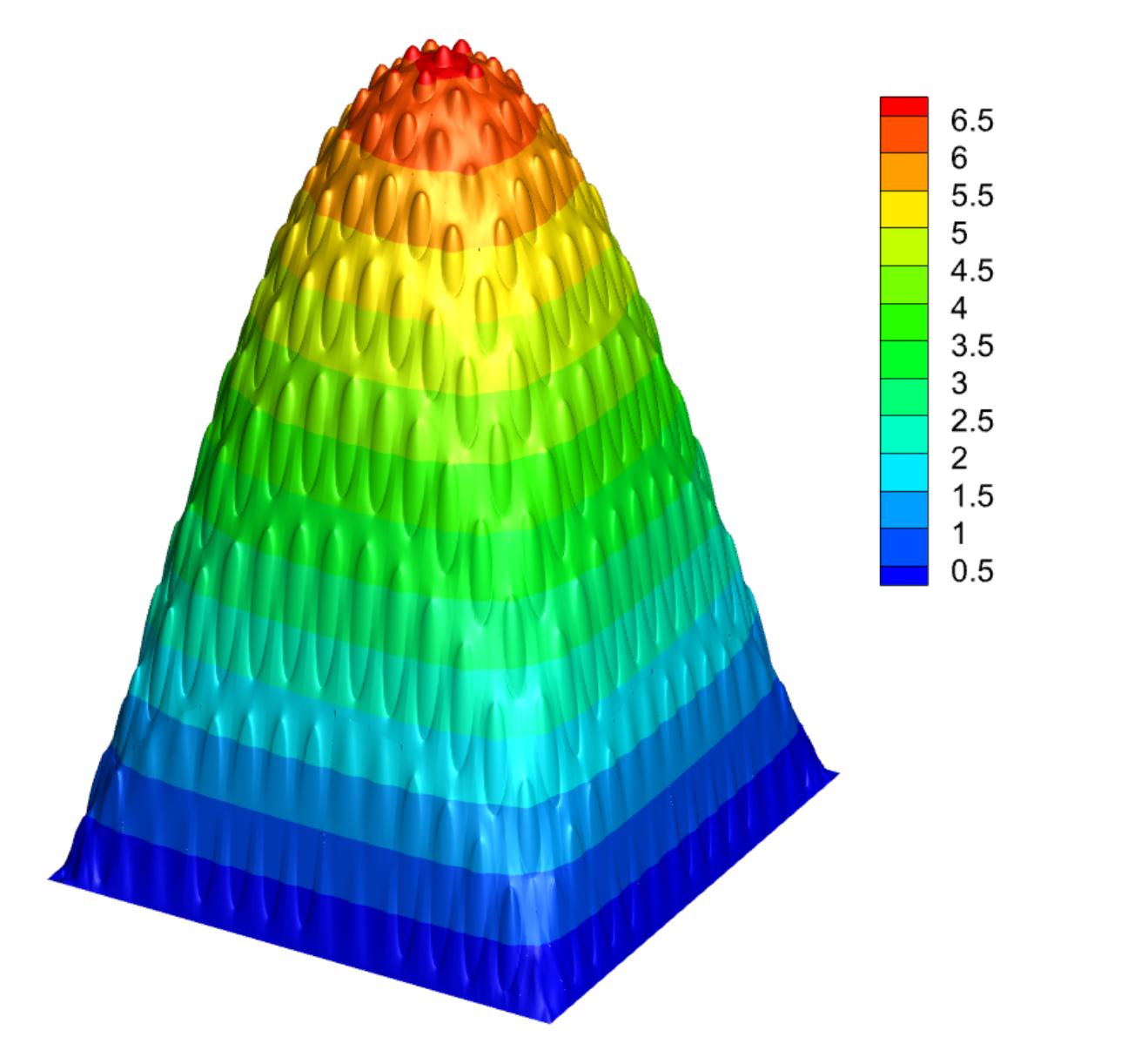} &
\includegraphics[width=0.32\textwidth]{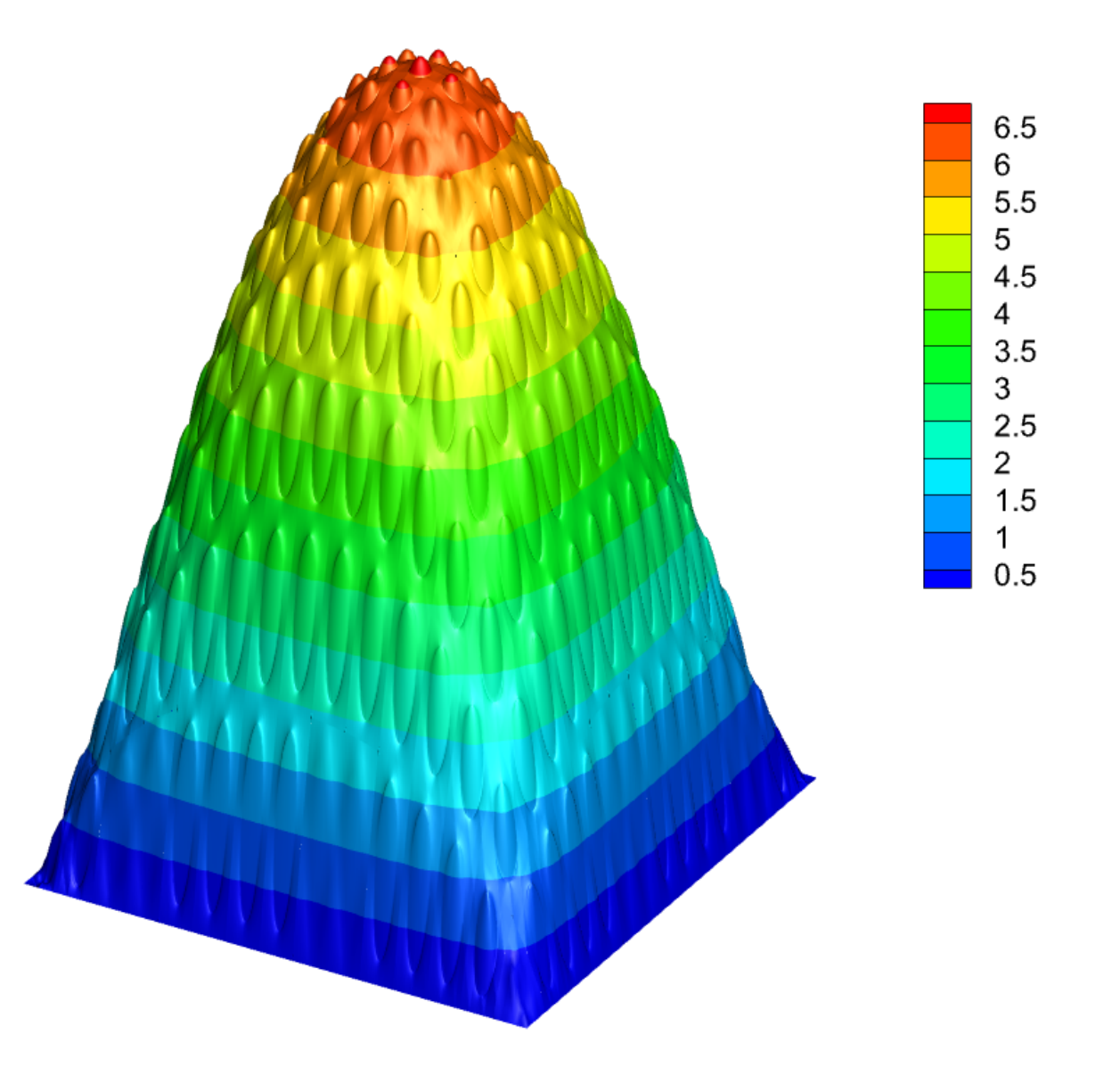}\\
(a) & (b)
\end{tabular}
\caption{Continuum 2 pressure fields for the large-scale three-dimensional spherical-inclusion problem at $T=1$ (central output section): (a) $u_2^{(2\eps)}$; (b) $u_{2,r}^{\eps,n}$.}
\label{fig:large-u2-fields}
\end{figure}

\section{Conclusions}
\label{sec:conclusions}

A POD-accelerated high-order multiscale framework has been developed for transient multi-continuum problems in highly heterogeneous media. The coupled homogenized variables are treated as a unified vector state, from which a mass-weighted joint POD space is constructed and used to reduce the fully discrete homogenized system. On this basis, two recovery strategies are introduced. Algorithm~1 reconstructs the reduced homogenized fields and performs the complete HOMS recovery at each output time, whereas Algorithm~2 applies the spatial recovery operators to the retained POD modes offline and reconstructs the final multiscale field through linear combinations of corrected and time-derivative bases. Owing to the linearity of the recovery operator, the two strategies produce the same POD-HOMS field up to interpolation and floating-point roundoff.

Theoretical analysis was carried out for both the reduced homogenized dynamics and the recovered multiscale field. A G\aa rding framework was used to accommodate the coupled drift and exchange terms, and an error estimate was established between the fully discrete homogenized solution and its POD-reduced approximation. By further accounting for the value, gradient, inter-continuum difference, second-derivative, and time-derivative recovery contributions, a recovered POD-HOMS error estimate was obtained. The numerical experiments show that the reduced solutions preserve the accuracy and local oscillatory structures of the corresponding second-order HOMS solutions. The corrected-basis formulation substantially reduces repeated recovery costs, with pronounced gains in three-dimensional computations. In the large-scale test, direct fine-scale FEM and FEM-POD exceed the available memory, whereas HOMS-POD remains feasible and reduces the total computational time from 8169.080 s to 304.940 s.

Future work will extend the present framework to nonlinear exchange mechanisms and nonperiodic heterogeneous media. Adaptive snapshot enrichment will also be investigated to improve the robustness of the reduced space for more complex transient regimes. In addition, parameter-dependent corrected bases may provide an efficient route toward repeated-query and multi-parameter simulations.

\section*{CRediT authorship contribution statement}

Jiayuan Peng: Methodology, Software, Investigation, Data curation, Writing--original draft. Hao Dong: Conceptualization, Methodology, Supervision, Funding acquisition, Writing--review and editing. Jiale Linghu: Software, Validation, Data curation.

\section*{Declaration of competing interest}

The authors declare no known competing financial interests or personal relationships that could have influenced the work reported in this paper.

\section*{Data availability}

The numerical data are contained in the accompanying multi-continuum POD-HOMS experiment folders. Tables were assembled from the archived timing summaries and relative-error files.

\section*{Acknowledgements}

This work was supported by the National Natural Science Foundation of China (No.~12471387), Xidian University Specially Funded Project for Interdisciplinary Exploration (No.~TZJH2024008), the Fundamental Research Funds for the Central Universities (No.~QTZX25082), and the Innovation Capability Support Program of Shaanxi Province (No.~2024RS-CXTD-88).

\end{document}